\documentclass[10pt,reqno]{amsart}

\usepackage[utf8x]{inputenc}
\usepackage[english]{babel}
\usepackage{amsmath,amsfonts,amssymb,amsthm,shuffle}
\usepackage[T1]{fontenc}
\usepackage[math]{anttor}

\usepackage[top=3.5cm,bottom=3.5cm,left=3.6cm,right=3.6cm]{geometry}

\usepackage{xcolor}

\definecolor{ColBlack}{RGB}{0,0,0} 
\definecolor{ColWhite}{RGB}{255,255,255} 
\definecolor{ColAA}{HTML}{92C271}
\definecolor{ColAB}{HTML}{509421}
\definecolor{ColAC}{HTML}{406235}
\definecolor{ColBA}{HTML}{EDAE77}
\definecolor{ColBB}{HTML}{C07A3E}
\definecolor{ColBC}{HTML}{8C7460}

\usepackage[hyperindex=true,frenchlinks=true,colorlinks=true,
citecolor=ColBB,linkcolor=ColBA,urlcolor=ColBC,linktocpage,
pagebackref=true]{hyperref}

\usepackage{tikz}
\usetikzlibrary{shapes}
\usetikzlibrary{fit}
\usetikzlibrary{decorations.pathmorphing}

\usepackage{mathtools}
\usepackage{dsfont}
\usepackage{wasysym}
\usepackage{stmaryrd}
\usepackage{mleftright}
\usepackage{cite}
\usepackage{subfig}
\usepackage{multirow}
\usepackage{enumitem}
\usepackage{multicol}
\usepackage{cancel}

\allowdisplaybreaks

\numberwithin{equation}{subsection}

\makeatletter
\def\l@section{\@tocline{1}{3pt}{1pc}{5pc}{}}
\def\l@subsection{\@tocline{2}{2pt}{2pc}{5pc}{}}
\makeatother

\newtheorem{Theorem}{Theorem}[subsection]
\newtheorem{Proposition}[Theorem]{Proposition}
\newtheorem{Lemma}[Theorem]{Lemma}

\renewcommand{\leq}{\leqslant}
\renewcommand{\geq}{\geqslant}

\newcommand{\ColAA}[1]{\textcolor{ColAA}{#1}}
\newcommand{\ColAB}[1]{\textcolor{ColAB}{#1}}

\newcommand{\Hide}[1]{\ColAA{\tt HIDEN}}
\newcommand{\Def}[1]{\ColAB{\em #1}}
\newcommand{\Par}[1]{\mleft(#1\mright)}
\newcommand{\Bra}[1]{\mleft\{#1\mright\}}
\newcommand{\Han}[1]{\mleft[#1\mright]}
\newcommand{\OEIS}[1]{\href{http://oeis.org/#1}{{\bf #1}}}

\tikzstyle{Centering}=[{baseline={([yshift=-0.5ex]current
    bounding box.center)}}]

\tikzstyle{MarkAA}=[draw=ColAA!80,fill=ColAA!8]
\tikzstyle{MarkAB}=[draw=ColAB!80,fill=ColAB!8]
\tikzstyle{MarkAC}=[draw=ColAC!80,fill=ColAC!8]
\tikzstyle{MarkBA}=[draw=ColBA!80,fill=ColBA!8]
\tikzstyle{MarkBB}=[draw=ColBB!80,fill=ColBB!8]
\tikzstyle{MarkBC}=[draw=ColBC!80,fill=ColBC!8]

\tikzstyle{Node}=[circle,MarkAA,inner sep=1pt,
minimum size=2mm,thick,font=\scriptsize]
\tikzstyle{Edge}=[draw=ColBB!80,cap=round,thick,rounded corners=2.5pt]
\tikzstyle{Leaf}=[rectangle,MarkBC,inner sep=0pt,minimum size=1mm,thick]
\tikzstyle{NodeST}=[font=\scriptsize]

\newcommand{\N}{\mathbb{N}}
\newcommand{\Z}{\mathbb{Z}}
\newcommand{\Q}{\mathbb{Q}}

\newcommand{\K}{\mathbb{K}}

\newcommand{\TreeS}{\mathfrak{s}}
\newcommand{\TreeR}{\mathfrak{r}}
\newcommand{\TreeT}{\mathfrak{t}}
\newcommand{\SetP}{\mathcal{P}}
\newcommand{\SetQ}{\mathcal{Q}}
\newcommand{\SetR}{\mathcal{R}}
\newcommand{\SetS}{\mathcal{S}}
\newcommand{\Predicate}{\mathbb{P}}
\newcommand{\Monoid}{\mathcal{M}}
\newcommand{\Deg}{\mathrm{deg}}
\newcommand{\GeneratingSet}{\mathfrak{G}}
\newcommand{\MinimalConsistent}{\mathfrak{M}}
\newcommand{\PrefixSet}{\mathrm{Pref}}
\newcommand{\SuffixSet}{\mathrm{Suff}}
\newcommand{\Trace}{\mathrm{tr}}
\newcommand{\AlphabetQ}{\mathbf{Q}}
\newcommand{\SeriesStatistics}{\mathit{H}}
\newcommand{\NormalForms}{\mathcal{N}}
\newcommand{\Statistics}{\mathrm{s}}

\newcommand{\Angle}[1]{\left\langle#1\right\rangle}
\newcommand{\AAngle}[1]{\Angle{\Angle{#1}}}
\newcommand{\Support}[1]{\mathrm{Supp}\Par{#1}}
\newcommand{\CharacteristicSeries}[1]{\mathrm{ch}\Par{#1}}
\newcommand{\PredicateSeries}[1]{\mathrm{pr}\Par{#1}}
\newcommand{\GeneratingSeries}{\mathcal{G}}
\newcommand{\HilbertSeries}{\mathcal{H}}
\newcommand{\FactorSeries}{\mathbf{F}}
\newcommand{\FactorPrefixSeries}{\mathbf{F}}
\newcommand{\FactorPrefixEnumSeries}{\mathit{F}}
\newcommand{\FactorEnumSeries}{\mathit{F}}
\newcommand{\Enum}{\mathrm{en}}
\newcommand{\SeriesF}{\mathbf{f}}
\newcommand{\SeriesG}{\mathbf{g}}

\newcommand{\Operad}{\mathcal{O}}
\newcommand{\Corolla}[1]{\mathfrak{c}\Par{#1}}
\newcommand{\Eval}{\mathrm{ev}}
\newcommand{\T}{\mathbf{T}}
\newcommand{\SyntaxTrees}{\mathbf{S}}
\newcommand{\Unit}{\mathds{1}}

\newcommand{\FCat}[1]{\mathbf{FCat}^{\Par{#1}}}
\newcommand{\Schr}{\mathbf{Schr}}
\newcommand{\Motz}{\mathbf{Motz}}
\newcommand{\DA}{\mathbf{DA}}
\newcommand{\TwoAs}{{2\mathbf{As}}}
\newcommand{\NCT}{\mathbf{NCT}}
\newcommand{\Dipt}{\mathbf{Dipt}}
\newcommand{\Dup}{\mathbf{Dup}}
\newcommand{\FreeOperad}{\mathbf{FO}}
\newcommand{\GenA}{\mathtt{a}}
\newcommand{\GenB}{\mathtt{b}}
\newcommand{\GenC}{\mathtt{c}}

\DeclareMathOperator{\Congr}{\equiv}
\DeclareMathOperator{\Product}{\star}
\DeclareMathOperator{\Composition}{\bar{\circ}}
\DeclareMathOperator{\Factor}{\preccurlyeq_{\mathrm{f}}}
\DeclareMathOperator{\NotFactor}{\cancel{\Factor}}
\DeclareMathOperator{\Prefix}{\preccurlyeq_{\mathrm{p}}}
\DeclareMathOperator{\NotPrefix}{\cancel{\Prefix}}
\DeclareMathOperator{\Suffix}{\preccurlyeq_{\mathrm{s}}}

\DeclareMathOperator{\WordsSum}{\dotplus}
\DeclareMathOperator{\Rew}{\rightarrow}
\DeclareMathOperator{\RewContext}{\Rightarrow}
\DeclareMathOperator{\RewContextRT}{\overset{\ast}{\RewContext}}

\newcommand{\Leaf}{{
\begin{tikzpicture}[xscale=.2,yscale=.22,Centering]
    \draw[Edge,thick](0,0)--(0,-1);
\end{tikzpicture}
}}

\newcommand{\CorollaTwo}[1]{{
\begin{tikzpicture}[Centering,xscale=0.16,yscale=0.24]
    \node(0)at(0.00,-1.50){};
    \node(2)at(2.00,-1.50){};
    \node[NodeST](1)at(1.00,0.00){$#1$};
    \draw[Edge](0)--(1);
    \draw[Edge](2)--(1);
    \node(r)at(1.00,1.5){};
    \draw[Edge](r)--(1);
\end{tikzpicture}
}}

\newcommand{\CorollaThree}[1]{{
\begin{tikzpicture}[Centering,xscale=0.19,yscale=0.19]
    \node(0)at(0.00,-2.00){};
    \node(2)at(1.00,-2.00){};
    \node(3)at(2.00,-2.00){};
    \node[NodeST](1)at(1.00,0.00){$#1$};
    \draw[Edge](0)--(1);
    \draw[Edge](2)--(1);
    \draw[Edge](3)--(1);
    \node(r)at(1.00,1.75){};
    \draw[Edge](r)--(1);
\end{tikzpicture}
}}

\newcommand{\TreeLeft}[2]{{
\begin{tikzpicture}[Centering,xscale=0.26,yscale=0.26]
    \node(0)at(0.00,-3.33){};
    \node(2)at(2.00,-3.33){};
    \node(4)at(4.00,-1.67){};
    \node[NodeST](1)at(1.00,-1.67){$#2$};
    \node[NodeST](3)at(3.00,0.00){$#1$};
    \draw[Edge](0)--(1);
    \draw[Edge](1)--(3);
    \draw[Edge](2)--(1);
    \draw[Edge](4)--(3);
    \node(r)at(3.00,1.5){};
    \draw[Edge](r)--(3);
\end{tikzpicture}
}}

\newcommand{\TreeRight}[2]{{
\begin{tikzpicture}[Centering,xscale=0.26,yscale=0.26]
    \node(0)at(0.00,-1.67){};
    \node(2)at(2.00,-3.33){};
    \node(4)at(4.00,-3.33){};
    \node[NodeST](1)at(1.00,0.00){$#1$};
    \node[NodeST](3)at(3.00,-1.67){$#2$};
    \draw[Edge](0)--(1);
    \draw[Edge](2)--(3);
    \draw[Edge](3)--(1);
    \draw[Edge](4)--(3);
    \node(r)at(1.00,1.25){};
    \draw[Edge](r)--(1);
\end{tikzpicture}
}}

\title[Tree series and patterns]
    {Tree series and pattern avoidance in syntax trees}
\date{\today}
\author{Samuele Giraudo}
\address{\scriptsize LIGM, Univ. Gustave Eiffel, CNRS, ESIEE Paris, F-$77454$
Marne-la-Vall\'ee, France}
\email{samuele.giraudo@u-pem.fr}

\begin{document}

\begin{abstract}
    A syntax tree is a planar rooted tree where internal nodes are labeled on a graded set
    of generators. There is a natural notion of occurrence of contiguous pattern in such
    trees. We describe a way, given a set of generators $\GeneratingSet$ and a set of
    patterns $\SetP$, to enumerate the trees constructed on $\GeneratingSet$ and avoiding
    $\SetP$.  The method is built around inclusion-exclusion formulas forming a system of
    equations on formal power series of trees, and composition operations of trees.  This
    does not require particular conditions on the set of patterns to avoid. We connect this
    result to the theory of nonsymmetric operads. Syntax trees are the elements of such free
    structures, so that any operad can be seen as a quotient of a free operad.  Moreover, in
    some cases, the elements of an operad can be seen as trees avoiding some patterns.
    Relying on this, we use operads as devices for enumeration: given a set of combinatorial
    objects we want enumerate, we endow it with the structure of an operad, understand it in
    term of trees and pattern avoidance, and use our method to count them. Several examples
    are provided.
\end{abstract}

\maketitle

\begin{small}
\tableofcontents
\end{small}

\section*{Introduction}
The general problem of counting objects is of primary importance in
combinatorics. Several approaches exist for this purpose. Here, we focus
on a strategy having an algebraic flavor consisting in endowing a set
$X$ of combinatorial objects with operations in order to form algebraic
structures. The point is that the algebraic study of $X$ (minimal
generating sets, relations between generators, morphisms, {\em etc.})
leads to enumerative results. Operads~\cite{LV12,Men15,Gir18} are
very interesting algebraic structures in this context. They encode the
notion of substitution of combinatorial objects into another one.
Moreover, formal power series on operads~\cite{Cha02,Cha08} or colored
operads~\cite{Gir19} (that are generalizations of usual formal power
series) offer new methods for enumerative questions. This work is
intended to be an application of the theory of operads to combinatorics
and enumeration. As our main contribution, we provide a tool to express the
Hilbert series (that is, the generating series of the sequence of the
dimensions) of an operad $\Operad$ given one of its presentations by
generators and relations (satisfying some properties). When $\Operad$ is
an operad on combinatorial objects, this provides a description of the
generating series of these objects. This is a consequence of the fact
that some operads can be seen as operads of trees avoiding some
patterns, and is related with the deeper notions of Koszul
operads~\cite{GK94}, Poincaré-Birkhoff-Witt bases for
operads~\cite{Hof10}, and Gröbner bases for operads~\cite{DK10}.
\smallbreak

Our main combinatorial result consists, given a set $\SetP$ of syntax trees (that are some
labeled planar rooted trees, where labels are taken from a fixed alphabet), to obtain a
system of equations expressing the formal sum of all the trees avoiding $\SetP$ (as
connected components in the trees). The presented solution is built around an
inclusion-exclusion formula and uses simple grafting operations on trees.  By considering
the projection of this system to usual formal power series, this leads to a system of
equations for the generating series of the trees avoiding $\SetP$. It is also possible to
add formal parameters into these systems to enumerate the trees according to some
statistics.  Methods to enumerate trees that avoid some patterns have been already provided
in~\cite{Row10} for the case of unlabeled binary trees, \cite{GPPT12} for the case of
unlabeled ternary trees, in~\cite{Par93} and~\cite{Lod05} for the case of patterns with two
internal nodes, and in~\cite{KP15} for the general case.  Our method differs from the latter
one both in the approach and in the obtained systems of equations. Indeed, in the previous
reference, the authors use combinatorics and enumerative properties to show algebraic
properties on operads (while in the present work, we use operads to obtain combinatorial
results and to count objects). Moreover, we obtain different systems of equations and we
have fewer requirements about the sets $\SetP$ to avoid (they can be infinite, and some of
their trees can be factors of other ones).  Note that there exist several notions of pattern
avoidance in trees~\cite{DKS20}. We focus here on contiguous patterns.
\smallbreak

This document is organized as follows.
Section~\ref{sec:syntax_trees_series} contains elementary definitions
about syntax trees and formal power series of trees. In
Section~\ref{sec:pattern_avoidance}, we state the main question of the
paper about pattern avoidance in syntax trees and provide its main
result (Theorem~\ref{thm:system_trees_avoiding}). Next,
Section~\ref{sec:operads_for_enumeration} is devoted to explaining how to
use nonsymmetric set-operads as devices for the enumeration of families
of combinatorial objects. For this, the elementary definitions about
operads are exposed, and a notion of refined Hilbert series of an operad
depending on an orientation of one its presentations by generators and
relations is provided. The document ends with Section~\ref{sec:examples}
where examples of enumerations of some families of combinatorial objects
are reviewed. We provide, by using several operad structures, the
enumeration of bicolored Schröder trees, Schröder trees, binary trees,
$m$-trees, noncrossing trees, Motzkin paths, and directed animals. The
tools provided by this work highlight some (already known or not)
statistics on these objects.
\medbreak

\subsubsection*{General notations and conventions}
For any integers $a$ and $c$, $[a, c]$ denotes the set
$\Bra{b \in \N : a \leq b \leq c}$ and $[n]$, the set $[1, n]$. The
cardinality of a finite set $S$ is denoted by $\# S$. If $u$ is a word,
its length is denoted by $|u|$ and for any position $i \in [|u|]$, $u_i$
is the $i$-th letter of $u$.
\medbreak

\section{Syntax trees and series} \label{sec:syntax_trees_series}
This section begins by setting elementary definitions about syntax
trees, the main combinatorial objects of this work. Next, we present
series on trees and some operations on them.
\medbreak

\subsection{Syntax trees}
We set here elementary definitions and notations about graded sets,
syntax trees, and composition operations on syntax trees.
\medbreak

\subsubsection{Graded sets and alphabets}
A \Def{graded set} is a set $\GeneratingSet$ admitting a decomposition
as a disjoint union of the form
\begin{equation}
    \GeneratingSet := \bigsqcup_{n \geq 1} \GeneratingSet(n).
\end{equation}
In the sequel, we shall call such a set an \Def{alphabet} and each of
its elements a \Def{letter}. The \Def{arity} $|x|$ of a letter $x$ of
$\GeneratingSet$ is the unique integer $n$ such that
$x \in \GeneratingSet(n)$. We say that $\GeneratingSet$ is
\Def{combinatorial} if all the $\GeneratingSet(n)$ are finite for all
$n \geq 1$. In this case, the \Def{generating series} of
$\GeneratingSet$ is the series $\GeneratingSeries_{\GeneratingSet}(t)$
defined by
\begin{equation}
    \GeneratingSeries_{\GeneratingSet}(t)
    := \sum_{x \in \GeneratingSet} t^{|x|}.
\end{equation}
The coefficient of $t^n$ in $\GeneratingSeries_{\GeneratingSet}(t)$ is
$\# \GeneratingSet(n)$ for any $n \geq 1$.
\medbreak

\subsubsection{Syntax trees} \label{subsubsec:syntax_trees}
Let $\GeneratingSet$ be an alphabet. A \Def{$\GeneratingSet$-tree} (also
called \Def{$\GeneratingSet$-syntax tree}) is a planar rooted tree such
that its internal nodes of arity $k$ are labeled by letters of arity $k$
of $\GeneratingSet$. Unless otherwise specified, we use in the sequel
the standard terminology (such as \Def{node}, \Def{internal node},
\Def{leaf}, \Def{edge}, \Def{root}, \Def{child}, {\em etc.}) about
planar rooted trees~\cite{Knu97} (see also~\cite{Gir18}). Let us set
here some definitions about $\GeneratingSet$-trees. The \Def{degree}
$\Deg(\TreeT)$ (resp. \Def{arity} $|\TreeT|$) of a
$\GeneratingSet$-tree $\TreeT$ is its number of internal nodes (resp.
leaves). The only $\GeneratingSet$-tree of degree $0$ and arity $1$ is
the \Def{leaf} and is denoted by $\Leaf$. For any
$\GenA \in \GeneratingSet(k)$, the \Def{corolla} labeled by $\GenA$ is
the tree $\Corolla{\GenA}$ consisting in one internal node labeled by
$\GenA$ having as children $k$ leaves. Given an internal node $u$ of
$\TreeT$, due to the planarity of $\TreeT$, the children of $u$ are
totally ordered from left to right and are thus indexed from $1$ to the
arity $k$ of $u$. By assuming that the arity of the root of $\TreeT$ is
$k$, for any $i \in [k]$, the \Def{$i$-th subtree} of $\TreeT$ is the
tree $\TreeT(i)$ rooted at the $i$-th child of $\TreeT$. Similarly, the
leaves of $\TreeT$ are totally ordered from left to right and thus are
indexed from $1$ to~$|\TreeT|$. The \Def{height} of $\TreeT$ is the
number of internal nodes belonging to a longest path connecting the root
of $\TreeT$ to one of its leaves.
\medbreak

For instance, if
\begin{math}
    \GeneratingSet :=  \GeneratingSet(2) \sqcup \GeneratingSet(3)
\end{math}
with
$\GeneratingSet(2) := \{\GenA, \GenB\}$ and
$\GeneratingSet(3) := \{\GenC\}$,
\begin{equation}
    \TreeT :=
    \begin{tikzpicture}[Centering,xscale=.26,yscale=.15]
        \node(0)at(0.00,-6.50){};
        \node(10)at(8.00,-9.75){};
        \node(12)at(10.00,-9.75){};
        \node(2)at(2.00,-6.50){};
        \node(4)at(3.00,-3.25){};
        \node(5)at(4.00,-9.75){};
        \node(7)at(5.00,-9.75){};
        \node(8)at(6.00,-9.75){};
        \node[NodeST](1)at(1.00,-3.25){$\GenB$};
        \node[NodeST](11)at(9.00,-6.50){$\GenA$};
        \node[NodeST](3)at(3.00,0.00){$\GenC$};
        \node[NodeST](6)at(5.00,-6.50){$\GenC$};
        \node[NodeST](9)at(7.00,-3.25){$\GenA$};
        \node(r)at(3.00,2.75){};
        \draw[Edge](0)--(1);
        \draw[Edge](1)--(3);
        \draw[Edge](10)--(11);
        \draw[Edge](11)--(9);
        \draw[Edge](12)--(11);
        \draw[Edge](2)--(1);
        \draw[Edge](4)--(3);
        \draw[Edge](5)--(6);
        \draw[Edge](6)--(9);
        \draw[Edge](7)--(6);
        \draw[Edge](8)--(6);
        \draw[Edge](9)--(3);
        \draw[Edge](r)--(3);
    \end{tikzpicture}
\end{equation}
is a $\GeneratingSet$-tree of degree $5$, arity $8$, and height $3$. Its
root is labeled by $\GenC$ and has arity $3$. Moreover, we have
\begin{equation}
    \TreeT(1) =
    \begin{tikzpicture}[Centering,xscale=.25,yscale=.29]
        \node(0)at(0.00,-1.50){};
        \node(2)at(2.00,-1.50){};
        \node[NodeST](1)at(1.00,0.00){$\GenB$};
        \draw[Edge](0)--(1);
        \draw[Edge](2)--(1);
        \node(r)at(1.00,1.5){};
        \draw[Edge](r)--(1);
    \end{tikzpicture}
    = \Corolla{\GenB},
    \qquad
    \TreeT(2) = \Leaf,
    \qquad
    \TreeT(3) =
    \begin{tikzpicture}[Centering,xscale=.24,yscale=.17]
        \node(0)at(0.00,-5.33){};
        \node(2)at(1.00,-5.33){};
        \node(3)at(2.00,-5.33){};
        \node(5)at(4.00,-5.33){};
        \node(7)at(6.00,-5.33){};
        \node[NodeST](1)at(1.00,-2.67){$\GenC$};
        \node[NodeST](4)at(3.00,0.00){$\GenA$};
        \node[NodeST](6)at(5.00,-2.67){$\GenA$};
        \draw[Edge](0)--(1);
        \draw[Edge](1)--(4);
        \draw[Edge](2)--(1);
        \draw[Edge](3)--(1);
        \draw[Edge](5)--(6);
        \draw[Edge](6)--(4);
        \draw[Edge](7)--(6);
        \node(r)at(3.00,2.50){};
        \draw[Edge](r)--(4);
    \end{tikzpicture}\,.
\end{equation}
\medbreak

Given an alphabet $\GeneratingSet$, we denote by
$\SyntaxTrees(\GeneratingSet)$ the graded set of all the
$\GeneratingSet$-trees where $\SyntaxTrees(\GeneratingSet)(n)$ is the
subset of $\SyntaxTrees(\GeneratingSet)$ restrained on the
$\GeneratingSet$-trees of arity $n$. Observe that when $\GeneratingSet$
is combinatorial and
\begin{math}
    \GeneratingSet(1) = \emptyset$, $\SyntaxTrees(\GeneratingSet)
\end{math}
is combinatorial. In this case, the generating series
$\GeneratingSeries_{\SyntaxTrees(\GeneratingSet)}(t)$ of
$\SyntaxTrees(\GeneratingSet)$, counting its elements with respect to
their arities, satisfies
\begin{equation}
    \GeneratingSeries_{\SyntaxTrees(\GeneratingSet)}(t) =
    t + \GeneratingSeries_{\GeneratingSet}\Par{
        \GeneratingSeries_{\SyntaxTrees(\GeneratingSet)}(t)}.
\end{equation}
\medbreak

\subsubsection{Compositions of syntax trees}
\label{subsubsec:partial_composition_trees}
Given $\TreeT, \TreeS \in \SyntaxTrees(\GeneratingSet)$ and
$i \in [|\TreeT|]$, the \Def{partial composition}
$\TreeT \circ_i \TreeS$ is the $\GeneratingSet$-tree obtained by
grafting the root of $\TreeS$ onto the $i$-th leaf of $\TreeT$. For
instance, by considering the previous graded set $\GeneratingSet$ of
Section~\ref{subsubsec:syntax_trees}, one has
\begin{equation} \label{equ:example_partial_composition_trees}
    \begin{tikzpicture}[Centering,xscale=.2,yscale=.19]
        \node(0)at(0.00,-5.00){};
        \node(2)at(2.00,-7.50){};
        \node(4)at(4.00,-7.50){};
        \node(6)at(6.00,-5.00){};
        \node(8)at(7.00,-5.00){};
        \node(9)at(8.00,-5.00){};
        \node[NodeST](1)at(1.00,-2.50){$\GenB$};
        \node[NodeST](3)at(3.00,-5.00){$\GenA$};
        \node[NodeST](5)at(5.00,0.00){$\GenA$};
        \node[NodeST](7)at(7.00,-2.50){$\GenC$};
        \draw[Edge](0)--(1);
        \draw[Edge](1)--(5);
        \draw[Edge](2)--(3);
        \draw[Edge](3)--(1);
        \draw[Edge](4)--(3);
        \draw[Edge](6)--(7);
        \draw[Edge](7)--(5);
        \draw[Edge](8)--(7);
        \draw[Edge](9)--(7);
        \node(r)at(5.00,2){};
        \draw[Edge](r)--(5);
    \end{tikzpicture}
    \enspace \circ_5 \enspace
    \begin{tikzpicture}[Centering,xscale=.25,yscale=.21]
        \node(0)at(0.00,-2.00){};
        \node(2)at(1.00,-2.00){};
        \node(3)at(2.00,-4.00){};
        \node(5)at(4.00,-4.00){};
        \node[NodeST](1)at(1.00,0.00){$\GenC$};
        \node[NodeST](4)at(3.00,-2.00){$\GenB$};
        \draw[Edge](0)--(1);
        \draw[Edge](2)--(1);
        \draw[Edge](3)--(4);
        \draw[Edge](4)--(1);
        \draw[Edge](5)--(4);
        \node(r)at(1.00,2){};
        \draw[Edge](r)--(1);
    \end{tikzpicture}
    \enspace = \enspace
    \begin{tikzpicture}[Centering,xscale=.22,yscale=.16]
        \node(0)at(0.00,-6.00){};
        \node(10)at(8.00,-9.00){};
        \node(11)at(9.00,-12.00){};
        \node(13)at(11.00,-12.00){};
        \node(14)at(12.00,-6.00){};
        \node(2)at(2.00,-9.00){};
        \node(4)at(4.00,-9.00){};
        \node(6)at(6.00,-6.00){};
        \node(8)at(7.00,-9.00){};
        \node[NodeST](1)at(1.00,-3.00){$\GenB$};
        \node[NodeST](12)at(10.00,-9.00){$\GenB$};
        \node[NodeST](3)at(3.00,-6.00){$\GenA$};
        \node[NodeST](5)at(5.00,0.00){$\GenA$};
        \node[NodeST](7)at(9.00,-3.00){$\GenC$};
        \node[NodeST](9)at(8.00,-6.00){$\GenC$};
        \draw[Edge](0)--(1);
        \draw[Edge](1)--(5);
        \draw[Edge](10)--(9);
        \draw[Edge](11)--(12);
        \draw[Edge](12)--(9);
        \draw[Edge](13)--(12);
        \draw[Edge](14)--(7);
        \draw[Edge](2)--(3);
        \draw[Edge](3)--(1);
        \draw[Edge](4)--(3);
        \draw[Edge](6)--(7);
        \draw[Edge](7)--(5);
        \draw[Edge](8)--(9);
        \draw[Edge](9)--(7);
        \node(r)at(5.00,2.5){};
        \draw[Edge](r)--(5);
    \end{tikzpicture}.
\end{equation}
\medbreak

Furthermore, given $\TreeT \in \SyntaxTrees(\GeneratingSet)$ and
$\TreeS_1, \dots, \TreeS_{|\TreeT|} \in \SyntaxTrees(\GeneratingSet)$,
the \Def{full composition}
$\TreeT \circ \Han{\TreeS_1, \dots, \TreeS_{|\TreeT|}}$ is the
$\GeneratingSet$-tree obtained by grafting $\TreeS_i$ onto the $i$-th
leaf of $\TreeT$, simultaneously for all the $i \in [|\TreeT|]$. For
instance, by considering the previous graded set $\GeneratingSet$, one
has
\begin{equation}
    \begin{tikzpicture}[Centering,xscale=.25,yscale=.26]
        \node(0)at(0.00,-1.67){};
        \node(2)at(2.00,-3.33){};
        \node(4)at(4.00,-3.33){};
        \node[NodeST](1)at(1.00,0.00){$\GenB$};
        \node[NodeST](3)at(3.00,-1.67){$\GenA$};
        \draw[Edge](0)--(1);
        \draw[Edge](2)--(3);
        \draw[Edge](3)--(1);
        \draw[Edge](4)--(3);
        \node(r)at(1.00,1.75){};
        \draw[Edge](r)--(1);
    \end{tikzpicture}
    \enspace \circ \enspace
    \Han{
    \begin{tikzpicture}[Centering,xscale=.21,yscale=.18]
        \node(0)at(0.00,-4.67){};
        \node(2)at(2.00,-4.67){};
        \node(4)at(4.00,-4.67){};
        \node(6)at(6.00,-4.67){};
        \node[NodeST](1)at(1.00,-2.33){$\GenA$};
        \node[NodeST](3)at(3.00,0.00){$\GenA$};
        \node[NodeST](5)at(5.00,-2.33){$\GenB$};
        \draw[Edge](0)--(1);
        \draw[Edge](1)--(3);
        \draw[Edge](2)--(1);
        \draw[Edge](4)--(5);
        \draw[Edge](5)--(3);
        \draw[Edge](6)--(5);
        \node(r)at(3.00,2){};
        \draw[Edge](r)--(3);
    \end{tikzpicture},
    \Leaf,
    \begin{tikzpicture}[Centering,xscale=.24,yscale=.23]
        \node(0)at(0.00,-2.00){};
        \node(2)at(1.00,-2.00){};
        \node(3)at(2.00,-2.00){};
        \node[NodeST](1)at(1.00,0.00){$\GenC$};
        \draw[Edge](0)--(1);
        \draw[Edge](2)--(1);
        \draw[Edge](3)--(1);
        \node(r)at(1.00,1.75){};
        \draw[Edge](r)--(1);
    \end{tikzpicture}}
    \enspace = \enspace
    \begin{tikzpicture}[Centering,xscale=.2,yscale=.14]
        \node(0)at(0.00,-10.50){};
        \node(10)at(10.00,-10.50){};
        \node(12)at(11.00,-10.50){};
        \node(13)at(12.00,-10.50){};
        \node(2)at(2.00,-10.50){};
        \node(4)at(4.00,-10.50){};
        \node(6)at(6.00,-10.50){};
        \node(8)at(8.00,-7.00){};
        \node[NodeST](1)at(1.00,-7.00){$\GenA$};
        \node[NodeST](11)at(11.00,-7.00){$\GenC$};
        \node[NodeST](3)at(3.00,-3.50){$\GenA$};
        \node[NodeST](5)at(5.00,-7.00){$\GenB$};
        \node[NodeST](7)at(7.00,0.00){$\GenB$};
        \node[NodeST](9)at(9.00,-3.50){$\GenA$};
        \draw[Edge](0)--(1);
        \draw[Edge](1)--(3);
        \draw[Edge](10)--(11);
        \draw[Edge](11)--(9);
        \draw[Edge](12)--(11);
        \draw[Edge](13)--(11);
        \draw[Edge](2)--(1);
        \draw[Edge](3)--(7);
        \draw[Edge](4)--(5);
        \draw[Edge](5)--(3);
        \draw[Edge](6)--(5);
        \draw[Edge](8)--(9);
        \draw[Edge](9)--(7);
        \node(r)at(7.00,3){};
        \draw[Edge](r)--(7);
    \end{tikzpicture}.
\end{equation}
\medbreak

By a slight but convenient abuse of notation, we shall in some cases
simply write $\GenA \circ_i \GenB$ instead of
$\Corolla{\GenA} \circ_i \Corolla{\GenB}$, and write
$\GenA \circ \Han{\TreeS_1, \dots, \TreeS_{|\GenA|}}$ instead of
$\Corolla{\GenA} \circ \Han{\TreeS_1, \dots, \TreeS_{|\GenA|}}$ where
$\GenA$ and $\GenB$ are letters of $\GeneratingSet$ and  $\TreeS_1$,
\dots, $\TreeS_{|\GenA|}$ are $\GeneratingSet$-trees. Moreover, when the
context is clear, we shall even write $\GenA$ for~$\Corolla{\GenA}$.
\medbreak

\subsection{Series on combinatorial sets}
We set here elementary definitions and notations about formal power
series on arbitrary sets and about series on trees.
\medbreak

\subsubsection{Series on a set}
Let $\K$ be any field of characteristic zero. It is convenient, for
enumerative purposes, to consider that $\K$ is simply the field~$\Q$.
\medbreak

If $X$ is a set, the linear span of $X$ is denoted by $\K \Angle{X}$.
The dual space of $\K \Angle{X}$, denoted by $\K \AAngle{X}$ is by
definition the space of the maps $\SeriesF : X \to \K$, called
\Def{$X$-series}. Let $\SeriesF \in \K \AAngle{X}$. The coefficient
$\SeriesF(x)$ of any $x \in X$ in $\SeriesF$ is denoted by
$\Angle{x, \SeriesF}$. The \Def{support} of $\SeriesF$ is the set
\begin{math}
    \Support{\SeriesF} := \Bra{x \in X : \Angle{x, \SeriesF} \ne 0}.
\end{math}
We say that $x \in X$ \Def{appears} in $\SeriesF$ if
$x \in \Support{\SeriesF}$. By exploiting the vector space structure of
$\K \AAngle{X}$, any $X$-series $\SeriesF$ expresses as
\begin{equation} \label{equ:definition_series_as_sums}
    \SeriesF = \sum_{x \in X} \Angle{x, \SeriesF} x.
\end{equation}
This notation using potentially infinite sums of elements of $X$
accompanied with coefficients of $\K$ is common in the context of formal
power series. In the sequel, we shall define and handle some $X$-series
using the notation~\eqref{equ:definition_series_as_sums}.
\medbreak

If $\Predicate$ is a predicate on $X$, that is, for any $x \in X$,
either $\Predicate(x)$ holds or $\Predicate(x) $ does not hold, the
\Def{predicate series} of $\Predicate$ is the series
\begin{equation}
    \PredicateSeries{\Predicate} :=
    \sum_{\substack{x \in X \\ \Predicate(x)}} x.
\end{equation}
Moreover, for any subset $Y$ of $X$, the \Def{characteristic series}
$\CharacteristicSeries{Y}$ of $Y$ is the predicate series of
$\Predicate$ where $\Predicate(y)$ holds if and only if $y \in Y$. If
$\Predicate_1$ and $\Predicate_2$ are two predicates on $X$, we denote
by $\Predicate_1 \wedge \Predicate_2$ (resp.
$\Predicate_1 \vee \Predicate_2$) the predicate wherein, for any
$x \in X$, $\Par{\Predicate_1 \wedge \Predicate_2}(x)$ (resp.
$\Par{\Predicate_1 \vee \Predicate_2}(x)$) holds if and only if
$\Predicate_1(x)$ and $\Predicate_2(x)$ (resp. $\Predicate_1(x)$ or
$\Predicate_2(x)$) hold.
\medbreak

\begin{Lemma} \label{lem:series_predicate}
    Let $X$ be a set and $\Predicate_1$, \dots, $\Predicate_n$,
    $n \geq 1$, be predicates on $X$. In $\K \AAngle{X}$, we have
    \begin{equation} \label{equ:series_predicate}
        \PredicateSeries{\bigvee_{i \in [n]} \Predicate_i} =
        \sum_{\substack{
            \ell \geq 1 \\
            \Bra{i_1, \dots, i_\ell} \subseteq [n]
        }}
        (-1)^{1 + \ell} \;
        \PredicateSeries{\bigwedge_{j \in [\ell]} \Predicate_{i_j}}.
    \end{equation}
\end{Lemma}
\begin{proof}
    Let
    \begin{math}
        \SeriesF :=
        \PredicateSeries{\Predicate_1} + \PredicateSeries{\Predicate_2}
        - \PredicateSeries{\Predicate_1 \wedge \Predicate_2}
    \end{math}
    obtained from the right member of~\eqref{equ:series_predicate} in
    the particular case where $n = 2$. In $\SeriesF$, each $x \in X$ has
    a coefficient $0$ or $1$ according to the following rules:
    \begin{enumerate}[label = {\arabic*.}]
        \item if not $\Predicate_1(x)$ and not $\Predicate_2(x)$, then
        the coefficient of $x$ is $0 + 0 - 0 = 0$;
        \item if $\Predicate_1(x)$ and not $\Predicate_2(x)$, then the
        coefficient of $x$ is $1 + 0 - 0 = 1$;
        \item if not $\Predicate_1(x)$ and $\Predicate_2(x)$, then the
        coefficient of $x$ is $0 + 1 - 0 = 1$;
        \item if $\Predicate_1(x)$ and $\Predicate_2(x)$, then the
        coefficient of $x$ is $1 + 1 - 1 = 1$.
    \end{enumerate}
    Therefore, $\SeriesF$ is the series
    $\PredicateSeries{\Predicate_1 \vee \Predicate_2}$, so
    that~\eqref{equ:series_predicate} holds for $n = 2$. Moreover,
    since~\eqref{equ:series_predicate} obviously holds when $n = 1$, by
    induction on $n$, the inclusion-exclusion formula of the
    statement of the lemma follows.
\end{proof}
\medbreak

\subsubsection{Series on syntax trees}
Let $\GeneratingSet$ be an alphabet. We call
\Def{$\GeneratingSet$-tree series} each series of
$\K \AAngle{\SyntaxTrees(\GeneratingSet)}$. For any $n \geq 1$, the
\Def{composition product} of $\GeneratingSet$-tree series is the product
\begin{equation}
    \Composition :
    \K \AAngle{\SyntaxTrees(\GeneratingSet)}
    \otimes
    \K \AAngle{\SyntaxTrees(\GeneratingSet)}^{\otimes n}
    \to \K \AAngle{\SyntaxTrees(\GeneratingSet)}
\end{equation}
defined for any $\GeneratingSet$-tree series $\SeriesF$ and
$\SeriesG_1$, \dots, $\SeriesG_n$ by
\begin{equation} \label{equ:composition_product_tree_series}
    \SeriesF \Composition \Han{\SeriesG_1, \dots, \SeriesG_n} :=
    \sum_{\substack{
        \TreeT \in \SyntaxTrees(\GeneratingSet)(n) \\
        \TreeS_1, \dots, \TreeS_n \in \SyntaxTrees(\GeneratingSet)
    }}
    \Par{\Angle{\TreeT, \SeriesF}
    \prod_{i \in [n]}
    \Angle{\TreeS_i, \SeriesG_i}}
    \TreeT \circ \Han{\TreeS_1, \dots, \TreeS_n}.
\end{equation}
Observe that this product is linear in all its inputs, and that it can
be seen as an extension by linearity of the full composition product of
$\GeneratingSet$-trees.
\medbreak

\subsubsection{Generating series}
Let us define from $\GeneratingSet$ the set
\begin{equation}
    \AlphabetQ_\GeneratingSet :=
    \Bra{q_\GenA : \GenA \in \GeneratingSet}
\end{equation}
of formal parameters. The usual set of the commutative generating series
on the set $\{t, q\} \cup \AlphabetQ_\GeneratingSet$ of parameters is
denoted by $\K \AAngle{t, q, \AlphabetQ_\GeneratingSet}$.
\medbreak

The \Def{trace} $\Trace(\TreeT)$ of a $\GeneratingSet$-tree $\TreeT$ is
the monomial of $\K \AAngle{t, q, \AlphabetQ_\GeneratingSet}$ defined by
\begin{equation}
    \Trace(\TreeT)
    := \prod_{\GenA \in \GeneratingSet} q_\GenA^{\Deg_\GenA(\TreeT)},
\end{equation}
where for any $\GenA \in \GeneratingSet$, $\Deg_\GenA(\TreeT)$ is the
number of internal nodes of $\TreeT$ labeled by $\GenA$. Moreover, the
\Def{enumeration map} on $\K \AAngle{\SyntaxTrees(\GeneratingSet)}$ is
the map
\begin{equation}
    \Enum :
    \K \AAngle{\SyntaxTrees(\GeneratingSet)}
    \to \K \AAngle{t, q, \AlphabetQ}
\end{equation}
defined linearly by
\begin{equation}
    \Enum(\TreeT) := t^{|\TreeT|} q^{\Deg(\TreeT)} \; \Trace(\TreeT).
\end{equation}
For any $\GeneratingSet$-tree series $\SeriesF$, the
\Def{enumerative image} of $\SeriesF$ is the generating series
$\Enum(\SeriesF)$. By definition, the coefficient of
\begin{math}
    t^n q^d q_{\GenA_1}^{\alpha_1} \dots q_{\GenA_\ell}^{\alpha_\ell},
\end{math}
$n \geq 1$, $d \geq 0$, $\alpha_i \geq 0$, $i \in [\ell]$, in the
enumerative image of the characteristic series of a set $S$ of
$\GeneratingSet$-trees is the number of trees $\TreeT$ of $S$ having
$n$ as arity, $d$ as degree, and
\begin{math}
    q_{\GenA_1}^{\alpha_1} \dots q_{\GenA_\ell}^{\alpha_\ell}
\end{math}
as trace.
\medbreak

Observe that for any alphabet $\GeneratingSet$, since there are finitely
many $\GeneratingSet$-trees having a fixed trace, the enumerative image
of any $\GeneratingSet$-tree series is always well-defined. Moreover,
when $\GeneratingSet$ is combinatorial and
$\GeneratingSet(1) = \emptyset$, there are finitely many
$\GeneratingSet$-trees having a given arity $n \geq 1$. For this reason,
for any set $S$ of $\GeneratingSet$-trees, the specialization
\begin{math}
    \CharacteristicSeries{S}
    _{|q := 1, q_\GenA := 1, \GenA \in \GeneratingSet}
\end{math}
is well-defined and is the series wherein the coefficient of $t^n$ is
the number of $\GeneratingSet$-trees of $S$ of arity $n$. Observe
finally that when $\GeneratingSet$ is finite, there are finitely many
$\GeneratingSet$-trees having a given degree $d \geq 0$. For this
reason, the specialization
\begin{math}
    \CharacteristicSeries{S}
    _{|t := 1, q_\GenA := 1, \GenA \in \GeneratingSet}
\end{math}
is well-defined and is the series wherein the coefficient of $q^d$ is
the number of $\GeneratingSet$-trees of $S$ of degree~$d$.
\medbreak

\begin{Proposition} \label{prop:enumeration_composition}
    For any alphabet $\GeneratingSet$, any $\GeneratingSet$-tree
    $\TreeT$ of arity $n \geq 1$, and any $\GeneratingSet$-tree series
    $\SeriesF_1$, \dots, $\SeriesF_n$,
    \begin{equation}
        \Enum\Par{\TreeT \Composition
        \Han{\SeriesF_1, \dots, \SeriesF_n}}
        =
        \frac{1}{t^{|\TreeT|}} \Enum(\TreeT) \;
        \prod_{i \in [n]} \Enum\Par{\SeriesF_i}.
    \end{equation}
\end{Proposition}
\begin{proof}
    The statement of the proposition follows by computing the
    enumerative image of the right member
    of~\eqref{equ:composition_product_tree_series}.
\end{proof}
\medbreak

Proposition~\ref{prop:enumeration_composition} admits the following
practical consequence. Assume that we have a set $S$ of
$\GeneratingSet$-trees we want enumerate (with respect to the arities
and the traces of its elements). A way to accomplish this consists in
providing an expression for $\Enum\Par{\CharacteristicSeries{S}}$. In
the case where we have a description of $\CharacteristicSeries{S}$ as an
expression using the sum, the multiplication by a scalar, and the
composition product of $\GeneratingSet$-tree series, we obtain thanks to
Proposition~\ref{prop:enumeration_composition} an expression for
$\Enum\Par{\CharacteristicSeries{S}}$ using only the sum, the
multiplication by a scalar, and the multiplication product of generating
series. We shall use this observation in the sequel to obtain systems of
equations of generating series from systems of equations of tree series.
\medbreak

\section{Tree series and pattern avoidance}
\label{sec:pattern_avoidance}
This section deals with two notions of pattern avoidance in syntax trees:
factor-avoidance and prefix-avoidance. The aim is to describe a way to
enumerate the syntax trees factor-avoiding a set of patterns. For this,
we begin by introducing some technical tools. Then, we state our main
result, provide some of its consequences, and finish by reviewing some
examples.
\medbreak

\subsection{Patterns in syntax trees}
The notions of prefix, factor, and suffix in syntax trees are set here.
Their immediate properties are stated.
\medbreak

\subsubsection{Factors, prefixes, and suffixes in trees}
Let $\GeneratingSet$ be an alphabet and let $\TreeT$ be a
$\GeneratingSet$-tree. When $\TreeT$ expresses as
\begin{equation} \label{equ:decomposition_tree}
    \TreeT =
    \TreeR \circ_i
    \Par{\TreeS \circ \Han{\TreeR_1, \dots, \TreeR_{|\TreeS|}}}
\end{equation}
for some $\GeneratingSet$-trees $\TreeS$, $\TreeR$, and $\TreeR_1$,
\dots, $\TreeR_{|\TreeS|}$, and $i \in [|\TreeR|]$, $\TreeS$ is a
\Def{factor} of $\TreeT$ and this property is denoted by
$\TreeS \Factor \TreeT$. Intuitively, this says that one can put down
$\TreeS$ at a certain place into $\TreeT$, by possibly superimposing
leaves of $\TreeS$ and internal nodes of $\TreeT$. When
$\TreeR = \Leaf$ in~\eqref{equ:decomposition_tree}, $\TreeS$ is a
\Def{prefix} of $\TreeT$ and this property is denoted
by~$\TreeS \Prefix \TreeT$. Intuitively, this says that $\TreeS$ is a
factor of $\TreeT$ wherein the root of $\TreeS$ can be put down onto the
root of $\TreeT$. Finally, when $\TreeR_j = \Leaf$ for all
$j \in [|\TreeS|]$ in~\eqref{equ:decomposition_tree}, $\TreeS$ is a
\Def{suffix} of $\TreeT$ and this property is denoted by
$\TreeS \Suffix \TreeT$. Let us consider some examples. By setting
\begin{equation}
    \TreeT :=
    \begin{tikzpicture}[Centering,xscale=.2,yscale=.14]
        \node(0)at(0.00,-10.50){};
        \node(10)at(10.00,-10.50){};
        \node(12)at(11.00,-10.50){};
        \node(13)at(12.00,-10.50){};
        \node(2)at(2.00,-10.50){};
        \node(4)at(4.00,-10.50){};
        \node(6)at(6.00,-10.50){};
        \node(8)at(8.00,-7.00){};
        \node[NodeST](1)at(1.00,-7.00){$\GenA$};
        \node[NodeST](11)at(11.00,-7.00){$\GenC$};
        \node[NodeST](3)at(3.00,-3.50){$\GenA$};
        \node[NodeST](5)at(5.00,-7.00){$\GenB$};
        \node[NodeST](7)at(7.00,0.00){$\GenB$};
        \node[NodeST](9)at(9.00,-3.50){$\GenA$};
        \draw[Edge](0)--(1);
        \draw[Edge](1)--(3);
        \draw[Edge](10)--(11);
        \draw[Edge](11)--(9);
        \draw[Edge](12)--(11);
        \draw[Edge](13)--(11);
        \draw[Edge](2)--(1);
        \draw[Edge](3)--(7);
        \draw[Edge](4)--(5);
        \draw[Edge](5)--(3);
        \draw[Edge](6)--(5);
        \draw[Edge](8)--(9);
        \draw[Edge](9)--(7);
        \node(r)at(7.00,3){};
        \draw[Edge](r)--(7);
    \end{tikzpicture},
\end{equation}
we have
\begin{equation}
    \begin{tikzpicture}[Centering,xscale=0.19,yscale=0.2]
        \node(0)at(0.00,-4.67){};
        \node(2)at(2.00,-4.67){};
        \node(4)at(4.00,-4.67){};
        \node(6)at(6.00,-4.67){};
        \node[NodeST](1)at(1.00,-2.33){$\GenA$};
        \node[NodeST](3)at(3.00,0.00){$\GenB$};
        \node[NodeST](5)at(5.00,-2.33){$\GenA$};
        \draw[Edge](0)--(1);
        \draw[Edge](1)--(3);
        \draw[Edge](2)--(1);
        \draw[Edge](4)--(5);
        \draw[Edge](5)--(3);
        \draw[Edge](6)--(5);
        \node(r)at(3.00,2){};
        \draw[Edge](r)--(3);
    \end{tikzpicture}
    \Factor
    \TreeT,
    \qquad
    \begin{tikzpicture}[Centering,xscale=0.22,yscale=0.25]
        \node(0)at(0.00,-1.67){};
        \node(2)at(2.00,-3.33){};
        \node(4)at(4.00,-3.33){};
        \node[NodeST](1)at(1.00,0.00){$\GenB$};
        \node[NodeST](3)at(3.00,-1.67){$\GenA$};
        \draw[Edge](0)--(1);
        \draw[Edge](2)--(3);
        \draw[Edge](3)--(1);
        \draw[Edge](4)--(3);
        \node(r)at(1.00,1.75){};
        \draw[Edge](r)--(1);
    \end{tikzpicture}
    \Prefix
    \TreeT,
    \qquad
    \begin{tikzpicture}[Centering,xscale=0.2,yscale=0.26]
        \node(0)at(0.00,-3.50){};
        \node(2)at(2.00,-5.25){};
        \node(4)at(4.00,-5.25){};
        \node(6)at(6.00,-1.75){};
        \node[NodeST](1)at(1.00,-1.75){$\GenA$};
        \node[NodeST](3)at(3.00,-3.50){$\GenB$};
        \node[NodeST](5)at(5.00,0.00){$\GenB$};
        \draw[Edge](0)--(1);
        \draw[Edge](1)--(5);
        \draw[Edge](2)--(3);
        \draw[Edge](3)--(1);
        \draw[Edge](4)--(3);
        \draw[Edge](6)--(5);
        \node(r)at(5.00,1.75){};
        \draw[Edge](r)--(5);
    \end{tikzpicture}
    \Prefix
    \TreeT,
    \qquad
    \begin{tikzpicture}[Centering,xscale=0.19,yscale=0.2]
        \node(0)at(0.00,-4.67){};
        \node(2)at(2.00,-4.67){};
        \node(4)at(4.00,-4.67){};
        \node(6)at(6.00,-4.67){};
        \node[NodeST](1)at(1.00,-2.33){$\GenA$};
        \node[NodeST](3)at(3.00,0.00){$\GenA$};
        \node[NodeST](5)at(5.00,-2.33){$\GenB$};
        \draw[Edge](0)--(1);
        \draw[Edge](1)--(3);
        \draw[Edge](2)--(1);
        \draw[Edge](4)--(5);
        \draw[Edge](5)--(3);
        \draw[Edge](6)--(5);
        \node(r)at(3.00,2){};
        \draw[Edge](r)--(3);
    \end{tikzpicture}
    \Suffix
    \TreeT.
\end{equation}
\medbreak

\begin{Proposition} \label{prop:posets_factors_prefixes}
    For any alphabet $\GeneratingSet$, $\Factor$, $\Prefix$, and
    $\Suffix$ endow $\SyntaxTrees(\GeneratingSet)$ with poset
    structures. Moreover, the poset
    $\Par{\SyntaxTrees(\GeneratingSet), \Factor}$ is an extension of
    $\Par{\SyntaxTrees(\GeneratingSet), \Prefix}$.
\end{Proposition}
\begin{proof}
    The fact that $\Factor$, $\Prefix$, and $\Suffix$ are order
    relations is straightforward from their definitions. Moreover, since
    for any $\GeneratingSet$-trees $\TreeS$ and $\TreeT$,
    $\TreeS \Prefix \TreeT$ implies $\TreeS \Factor \TreeT$, the second
    part of the statement of the proposition holds.
\end{proof}
\medbreak

When $\TreeS$ is not a factor (resp. a prefix) of $\TreeT$, $\TreeT$
\Def{factor-avoids} (resp. \Def{prefix-avoids}) $\TreeS$. This property
is denoted by $\TreeS \NotFactor \TreeT$ (resp.
$\TreeS \NotPrefix \TreeT$). By extension, when $\SetP$ is any subset of
$\SyntaxTrees(\GeneratingSet)$, $\TreeT$ \Def{factor-avoids} (resp.
\Def{prefix-avoids}) $\SetP$ if for all $\TreeS \in \SetP$,
$\TreeS \NotFactor \TreeT$ (resp. $\TreeS \NotPrefix \TreeT$). By a
slight abuse of notation, this property is denoted by
$\SetP \NotFactor \TreeT$ (resp. $\SetP \NotPrefix \TreeT$).
\medbreak

\begin{Lemma} \label{lem:prefix_recursive}
    Let $\GeneratingSet$ be an alphabet, and $\TreeS$ and $\TreeT$ be
    two $\GeneratingSet$-trees. Then, $\TreeS$ is a prefix of $\TreeT$
    if and only if $\TreeS = \Leaf$ or there exists a letter
    $\GenA \in \GeneratingSet(k)$ such that
    $\TreeS = \GenA \circ \Han{\TreeS(1), \dots, \TreeS(k)}$,
    $\TreeT = \GenA \circ \Han{\TreeT(1), \dots, \TreeT(k)}$, and for
    all $i \in [k]$, $\TreeS(i) \Prefix \TreeT(i)$.
\end{Lemma}
\begin{proof}
    This follows directly from the definition of the relation $\Prefix$.
\end{proof}
\medbreak

\subsubsection{Tree series avoiding factors}
For any subset $\SetP$ of
\begin{math}
    \SyntaxTrees(\GeneratingSet) \setminus \Bra{\Leaf},
\end{math}
let $\Predicate_\SetP$ be the predicate on
$\SyntaxTrees(\GeneratingSet)$ wherein $\Predicate_\SetP(\TreeT)$ holds
if and only $\SetP \NotFactor \TreeT$. Let also $\FactorSeries(\SetP)$
be the $\GeneratingSet$-tree series defined by
\begin{equation} \label{equ:factor_avoiding_series}
    \FactorSeries(\SetP) := \PredicateSeries{\Predicate_\SetP}.
\end{equation}
In other terms, $\FactorSeries(\SetP)$ is the characteristic series of
all $\GeneratingSet$-trees factor-avoiding all trees of~$\SetP$. In this
context, we say that the elements of $\SetP$ are \Def{patterns}. Notice
that we consider only sets of patterns $\SetP$ such that
$\Leaf \notin \SetP$ since there exists no $\GeneratingSet$-tree
factor-avoiding~$\Leaf$. Notice also that, for the while, there is no
restriction on $\GeneratingSet$ or $\SetP$. This set $\SetP$
of patterns can be infinite, and some trees can be themselves factors of
another one. The aim of the next section is to provide a system of
equations to describe~$\FactorSeries(\SetP)$ within the more general
possible context.
\medbreak

\subsection{Pattern avoidance and enumeration}
We provide here a way to obtain a system of equations to describe the
$\GeneratingSet$-tree series $\FactorSeries(\SetP)$. For this, we start
by introducing tools, namely consistent words and admissible trees.
From now, to not overload the notation, sets of patterns are denoted by
omitting the braces and the commas. Hence, sets of patterns can be seen
as unordered forests of $\GeneratingSet$-trees without repeated trees.
\medbreak

Moreover, all examples of this section are based upon the finite set of
patterns
\begin{equation} \label{equ:example_pattern_set}
    \SetP :=
    \begin{tikzpicture}[Centering,xscale=0.2,yscale=0.21]
        \node(0)at(0.00,-4.00){};
        \node(2)at(1.00,-4.00){};
        \node(3)at(2.00,-4.00){};
        \node(5)at(4.00,-2.00){};
        \node[NodeST](1)at(1.00,-2.00){$\GenC$};
        \node[NodeST](4)at(3.00,0.00){$\GenA$};
        \draw[Edge](0)--(1);
        \draw[Edge](1)--(4);
        \draw[Edge](2)--(1);
        \draw[Edge](3)--(1);
        \draw[Edge](5)--(4);
        \node(r)at(3.00,2){};
        \draw[Edge](r)--(4);
    \end{tikzpicture}
    \begin{tikzpicture}[Centering,xscale=.23,yscale=.2]
        \node(0)at(0.00,-4.00){};
        \node(2)at(2.00,-4.00){};
        \node(4)at(3.00,-2.00){};
        \node(5)at(4.00,-2.00){};
        \node[NodeST](1)at(1.00,-2.00){$\GenA$};
        \node[NodeST](3)at(3.00,0.00){$\GenC$};
        \draw[Edge](0)--(1);
        \draw[Edge](1)--(3);
        \draw[Edge](2)--(1);
        \draw[Edge](4)--(3);
        \draw[Edge](5)--(3);
        \node(r)at(3.00,2){};
        \draw[Edge](r)--(3);
    \end{tikzpicture}
    \begin{tikzpicture}[Centering,xscale=0.2,yscale=0.18]
        \node(0)at(0.00,-5.33){};
        \node(2)at(2.00,-5.33){};
        \node(4)at(3.00,-5.33){};
        \node(6)at(5.00,-5.33){};
        \node(7)at(6.00,-2.67){};
        \node[NodeST](1)at(1.00,-2.67){$\GenB$};
        \node[NodeST](3)at(4.00,0.00){$\GenC$};
        \node[NodeST](5)at(4.00,-2.67){$\GenB$};
        \draw[Edge](0)--(1);
        \draw[Edge](1)--(3);
        \draw[Edge](2)--(1);
        \draw[Edge](4)--(5);
        \draw[Edge](5)--(3);
        \draw[Edge](6)--(5);
        \draw[Edge](7)--(3);
        \node(r)at(4.00,2.00){};
        \draw[Edge](r)--(3);
    \end{tikzpicture}
    \begin{tikzpicture}[Centering,xscale=.2,yscale=.18]
        \node(0)at(0.00,-5.33){};
        \node(2)at(2.00,-5.33){};
        \node(4)at(3.00,-2.67){};
        \node(5)at(4.00,-5.33){};
        \node(7)at(6.00,-5.33){};
        \node[NodeST](1)at(1.00,-2.67){$\GenB$};
        \node[NodeST](3)at(3.00,0.00){$\GenC$};
        \node[NodeST](6)at(5.00,-2.67){$\GenA$};
        \draw[Edge](0)--(1);
        \draw[Edge](1)--(3);
        \draw[Edge](2)--(1);
        \draw[Edge](4)--(3);
        \draw[Edge](5)--(6);
        \draw[Edge](6)--(3);
        \draw[Edge](7)--(6);
        \node(r)at(3.00,2.00){};
        \draw[Edge](r)--(3);
    \end{tikzpicture}
    \begin{tikzpicture}[Centering,xscale=.175,yscale=.15]
        \node(0)at(0.00,-6.00){};
        \node(11)at(8.00,-6.00){};
        \node(2)at(1.00,-6.00){};
        \node(3)at(2.00,-9.00){};
        \node(5)at(3.00,-9.00){};
        \node(6)at(4.00,-9.00){};
        \node(8)at(5.00,-3.00){};
        \node(9)at(6.00,-6.00){};
        \node[NodeST](1)at(1.00,-3.00){$\GenC$};
        \node[NodeST](10)at(7.00,-3.00){$\GenA$};
        \node[NodeST](4)at(3.00,-6.00){$\GenC$};
        \node[NodeST](7)at(5.00,0.00){$\GenC$};
        \draw[Edge](0)--(1);
        \draw[Edge](1)--(7);
        \draw[Edge](10)--(7);
        \draw[Edge](11)--(10);
        \draw[Edge](2)--(1);
        \draw[Edge](3)--(4);
        \draw[Edge](4)--(1);
        \draw[Edge](5)--(4);
        \draw[Edge](6)--(4);
        \draw[Edge](8)--(7);
        \draw[Edge](9)--(10);
        \node(r)at(5.00,2.5){};
        \draw[Edge](r)--(7);
    \end{tikzpicture}.
\end{equation}
\medbreak

\subsubsection{Consistent words}
Let $\GeneratingSet$ be an alphabet and $\SetP$ be a subset of
\begin{math}
    \SyntaxTrees(\GeneratingSet) \setminus \Bra{\Leaf}.
\end{math}
For any $\GenA \in \GeneratingSet(k)$, $k \geq 1$, let
\begin{equation}
    \SetP_\GenA :=
    \Bra{\TreeS \in \SetP : \Corolla{\GenA} \Prefix \TreeS}.
\end{equation}
In other words, $\SetP_\GenA$ is the subset of $\SetP$ of the patterns
having roots labeled by $\GenA$. A word
$\SetS := \Par{\SetS_1, \dots, \SetS_k}$ where each $\SetS_i$ is a
subset of $\SyntaxTrees(\GeneratingSet)$, $i \in [k]$, is
\Def{$\SetP_\GenA$-consistent} if for any $\TreeS \in \SetP_\GenA$,
there is an $i \in [k]$ such that $\TreeS(i) \ne \Leaf$ and
$\TreeS(i) \in \SetS_i$. Observe that when $\Corolla{\GenA} \in \SetP$,
there is no $\SetP_\GenA$-consistent words. Moreover, a
$\GeneratingSet$-tree $\TreeT$ is \Def{$\SetS$-admissible} if the root
of $\TreeT$ is labeled by $\GenA$ and for all $i \in [k]$, $\TreeT(i)$
prefix-avoid~$\SetS_i$.
\medbreak

For instance, by considering the
set~\eqref{equ:example_pattern_set} of patterns, the word
\begin{equation}
    \SetS :=
    \Par{
    \begin{tikzpicture}[Centering,xscale=0.17,yscale=0.25]
        \node(0)at(0.00,-1.50){};
        \node(2)at(2.00,-1.50){};
        \node[NodeST](1)at(1.00,0.00){$\GenA$};
        \draw[Edge](0)--(1);
        \draw[Edge](2)--(1);
        \node(r)at(1.00,1.5){};
        \draw[Edge](r)--(1);
    \end{tikzpicture}, \enspace
    \begin{tikzpicture}[Centering,xscale=0.17,yscale=0.25]
        \node(0)at(0.00,-1.50){};
        \node(2)at(2.00,-1.50){};
        \node[NodeST](1)at(1.00,0.00){$\GenB$};
        \draw[Edge](0)--(1);
        \draw[Edge](2)--(1);
        \node(r)at(1.00,1.5){};
        \draw[Edge](r)--(1);
    \end{tikzpicture}
    \begin{tikzpicture}[Centering,xscale=0.22,yscale=0.22]
        \node(0)at(0.00,-2.00){};
        \node(2)at(1.00,-4.00){};
        \node(4)at(3.00,-4.00){};
        \node(5)at(4.00,-2.00){};
        \node[NodeST](1)at(2.00,0.00){$\GenC$};
        \node[NodeST](3)at(2.00,-2.00){$\GenA$};
        \draw[Edge](0)--(1);
        \draw[Edge](2)--(3);
        \draw[Edge](3)--(1);
        \draw[Edge](4)--(3);
        \draw[Edge](5)--(1);
        \node(r)at(2.00,1.75){};
        \draw[Edge](r)--(1);
    \end{tikzpicture}, \enspace
    \begin{tikzpicture}[Centering,xscale=0.17,yscale=0.25]
        \node(0)at(0.00,-1.50){};
        \node(2)at(2.00,-1.50){};
        \node[NodeST](1)at(1.00,0.00){$\GenA$};
        \draw[Edge](0)--(1);
        \draw[Edge](2)--(1);
        \node(r)at(1.00,1.5){};
        \draw[Edge](r)--(1);
    \end{tikzpicture}
    \begin{tikzpicture}[Centering,xscale=0.2,yscale=0.24]
        \node(0)at(0.00,-1.67){};
        \node(2)at(2.00,-3.33){};
        \node(4)at(4.00,-3.33){};
        \node[NodeST](1)at(1.00,0.00){$\GenA$};
        \node[NodeST](3)at(3.00,-1.67){$\GenA$};
        \draw[Edge](0)--(1);
        \draw[Edge](2)--(3);
        \draw[Edge](3)--(1);
        \draw[Edge](4)--(3);
        \node(r)at(1.00,1.75){};
        \draw[Edge](r)--(1);
    \end{tikzpicture}}
\end{equation}
is $\SetP_\GenC$-consistent. Moreover, the tree
\begin{equation}
    \TreeT :=
    \begin{tikzpicture}[Centering,xscale=0.2,yscale=0.15]
        \node(0)at(0.00,-3.20){};
        \node(10)at(8.00,-12.80){};
        \node(12)at(10.00,-12.80){};
        \node(14)at(11.00,-9.60){};
        \node(15)at(12.00,-9.60){};
        \node(2)at(1.00,-9.60){};
        \node(4)at(2.00,-9.60){};
        \node(5)at(3.00,-9.60){};
        \node(7)at(5.00,-6.40){};
        \node(8)at(6.00,-6.40){};
        \node[NodeST](1)at(4.00,0.00){$\GenC$};
        \node[NodeST](11)at(9.00,-9.60){$\GenA$};
        \node[NodeST](13)at(11.00,-6.40){$\GenC$};
        \node[NodeST](3)at(2.00,-6.40){$\GenC$};
        \node[NodeST](6)at(4.00,-3.20){$\GenA$};
        \node[NodeST](9)at(7.00,-3.20){$\GenB$};
        \draw[Edge](0)--(1);
        \draw[Edge](10)--(11);
        \draw[Edge](11)--(13);
        \draw[Edge](12)--(11);
        \draw[Edge](13)--(9);
        \draw[Edge](14)--(13);
        \draw[Edge](15)--(13);
        \draw[Edge](2)--(3);
        \draw[Edge](3)--(6);
        \draw[Edge](4)--(3);
        \draw[Edge](5)--(3);
        \draw[Edge](6)--(1);
        \draw[Edge](7)--(6);
        \draw[Edge](8)--(9);
        \draw[Edge](9)--(1);
        \node(r)at(4.00,2.5){};
        \draw[Edge](r)--(1);
    \end{tikzpicture}
\end{equation}
is $\SetS_\GenC$-admissible. Observe however that $\TreeT$ does not
factor-avoids $\SetP$ or $\SetP_\GenC$.
\medbreak

\begin{Lemma} \label{lem:consistent_words}
    Let $\GeneratingSet$ be an alphabet, $\SetP$ be a subset of
    \begin{math}
        \SyntaxTrees(\GeneratingSet) \setminus \Bra{\Leaf},
    \end{math}
    $\GenA \in \GeneratingSet$, and $\SetS$ be a
    $\SetP_\GenA$-consistent word. If $\TreeT$ is an $\SetS$-admissible
    $\GeneratingSet$-tree, then $\TreeT$ prefix-avoids~$\SetP_\GenA$.
\end{Lemma}
\begin{proof}
    Let us denote by $k$ the arity of $\GenA$. Since $\TreeT$ is
    $\SetS$-admissible, for all $i \in [k]$ and $\TreeS \in \SetS_i$, we
    have $\TreeS \NotPrefix \TreeT(i)$. Since for any
    $\TreeR \in \SetP_\GenA$, there is a $j \in [k]$ such that
    $\TreeR(j) \ne \Leaf$ and $\TreeR(j) \in \SetS_j$, we have in
    particular that $\TreeR(j) \NotPrefix \TreeT(j)$. Since moreover the
    root of $\TreeT$ is labeled by $\GenA$, by
    Lemma~\ref{lem:prefix_recursive}, one deduces
    that~$\TreeS \NotPrefix \TreeT$.
\end{proof}
\medbreak

If $\Par{\SetS_1, \dots, \SetS_k}$ and
$\Par{\SetS'_1, \dots, \SetS'_k}$ are two words of a same length $k$
where each $\SetS_i$ and $\SetS'_i$ is a subset of
$\SyntaxTrees(\GeneratingSet)$, their \Def{sum} is the word
\begin{equation}
    \Par{\SetS_1, \dots, \SetS_k}
    \WordsSum
    \Par{\SetS'_1, \dots, \SetS'_k}
    :=
    \Par{\SetS_1 \cup \SetS'_1, \dots, \SetS_k \cup \SetS'_k}.
\end{equation}
A $\SetP_\GenA$-consistent word $\Par{\SetS_1, \dots, \SetS_k}$ is
\Def{minimal} if any decomposition
\begin{equation}
    \Par{\SetS_1, \dots, \SetS_k} =
    \Par{\SetS'_1, \dots, \SetS'_k}
    \WordsSum
    \Par{\SetS''_1, \dots, \SetS''_k}
\end{equation}
where $\Par{\SetS'_1, \dots, \SetS'_k}$ is a $\SetP_\GenA$-consistent
word and $\Par{\SetS''_1, \dots, \SetS''_k}$ is a word where each
$\SetS''_i$, $i \in [k]$, is a subset of
$\SyntaxTrees(\GeneratingSet)$, implies
$\Par{\SetS_1, \dots, \SetS_k} = \Par{\SetS'_1, \dots, \SetS'_k}$.
Intuitively, this says that a $\SetP_\GenA$-consistent word is minimal
if the suppression of any tree in one of its letters leads to a word
which is not $\SetP_\GenA$-consistent. We denote by
$\MinimalConsistent\Par{\SetP_\GenA}$ the set of all minimal
$\SetP_\GenA$-consistent words.
\medbreak

For instance, by considering the set~\eqref{equ:example_pattern_set} of
patterns,
\begin{subequations}
\begin{equation}
    \MinimalConsistent\Par{\SetP_\GenA} =
    \Bra{\Par{
    \begin{tikzpicture}[Centering,xscale=0.2,yscale=0.2]
        \node(0)at(0.00,-2.00){};
        \node(2)at(1.00,-2.00){};
        \node(3)at(2.00,-2.00){};
        \node[NodeST](1)at(1.00,0.00){$\GenC$};
        \draw[Edge](0)--(1);
        \draw[Edge](2)--(1);
        \draw[Edge](3)--(1);
        \node(r)at(1.00,1.75){};
        \draw[Edge](r)--(1);
    \end{tikzpicture}, \enspace
    \emptyset}},
\end{equation}
\begin{equation}
    \MinimalConsistent\Par{\SetP_\GenB} =
    \Bra{\Par{\emptyset, \emptyset}},
\end{equation}
\begin{equation}
    \MinimalConsistent\Par{\SetP_\GenC} =
    \Bra{
    \Par{
    \begin{tikzpicture}[Centering,xscale=0.17,yscale=0.25]
        \node(0)at(0.00,-1.50){};
        \node(2)at(2.00,-1.50){};
        \node[NodeST](1)at(1.00,0.00){$\GenA$};
        \draw[Edge](0)--(1);
        \draw[Edge](2)--(1);
        \node(r)at(1.00,1.5){};
        \draw[Edge](r)--(1);
    \end{tikzpicture}, \enspace
    \begin{tikzpicture}[Centering,xscale=0.17,yscale=0.25]
        \node(0)at(0.00,-1.50){};
        \node(2)at(2.00,-1.50){};
        \node[NodeST](1)at(1.00,0.00){$\GenB$};
        \draw[Edge](0)--(1);
        \draw[Edge](2)--(1);
        \node(r)at(1.00,1.5){};
        \draw[Edge](r)--(1);
    \end{tikzpicture}, \enspace
    \begin{tikzpicture}[Centering,xscale=0.17,yscale=0.25]
        \node(0)at(0.00,-1.50){};
        \node(2)at(2.00,-1.50){};
        \node[NodeST](1)at(1.00,0.00){$\GenA$};
        \draw[Edge](0)--(1);
        \draw[Edge](2)--(1);
        \node(r)at(1.00,1.5){};
        \draw[Edge](r)--(1);
    \end{tikzpicture}},
    \Par{
    \begin{tikzpicture}[Centering,xscale=0.17,yscale=0.25]
        \node(0)at(0.00,-1.50){};
        \node(2)at(2.00,-1.50){};
        \node[NodeST](1)at(1.00,0.00){$\GenA$};
        \draw[Edge](0)--(1);
        \draw[Edge](2)--(1);
        \node(r)at(1.00,1.5){};
        \draw[Edge](r)--(1);
    \end{tikzpicture}
    \begin{tikzpicture}[Centering,xscale=0.17,yscale=0.25]
        \node(0)at(0.00,-1.50){};
        \node(2)at(2.00,-1.50){};
        \node[NodeST](1)at(1.00,0.00){$\GenB$};
        \draw[Edge](0)--(1);
        \draw[Edge](2)--(1);
        \node(r)at(1.00,1.5){};
        \draw[Edge](r)--(1);
    \end{tikzpicture}, \enspace
    \emptyset, \enspace
    \begin{tikzpicture}[Centering,xscale=0.17,yscale=0.25]
        \node(0)at(0.00,-1.50){};
        \node(2)at(2.00,-1.50){};
        \node[NodeST](1)at(1.00,0.00){$\GenA$};
        \draw[Edge](0)--(1);
        \draw[Edge](2)--(1);
        \node(r)at(1.00,1.5){};
        \draw[Edge](r)--(1);
    \end{tikzpicture}},
    \Par{
    \begin{tikzpicture}[Centering,xscale=0.17,yscale=0.25]
        \node(0)at(0.00,-1.50){};
        \node(2)at(2.00,-1.50){};
        \node[NodeST](1)at(1.00,0.00){$\GenA$};
        \draw[Edge](0)--(1);
        \draw[Edge](2)--(1);
        \node(r)at(1.00,1.5){};
        \draw[Edge](r)--(1);
    \end{tikzpicture}
    \begin{tikzpicture}[Centering,xscale=0.17,yscale=0.25]
        \node(0)at(0.00,-1.50){};
        \node(2)at(2.00,-1.50){};
        \node[NodeST](1)at(1.00,0.00){$\GenB$};
        \draw[Edge](0)--(1);
        \draw[Edge](2)--(1);
        \node(r)at(1.00,1.5){};
        \draw[Edge](r)--(1);
    \end{tikzpicture}
    \begin{tikzpicture}[Centering,xscale=0.19,yscale=0.17]
        \node(0)at(0.00,-2.33){};
        \node(2)at(1.00,-2.33){};
        \node(3)at(2.00,-4.67){};
        \node(5)at(3.00,-4.67){};
        \node(6)at(4.00,-4.67){};
        \node[NodeST](1)at(1.00,0.00){$\GenC$};
        \node[NodeST](4)at(3.00,-2.33){$\GenC$};
        \draw[Edge](0)--(1);
        \draw[Edge](2)--(1);
        \draw[Edge](3)--(4);
        \draw[Edge](4)--(1);
        \draw[Edge](5)--(4);
        \draw[Edge](6)--(4);
        \node(r)at(1.00,2){};
        \draw[Edge](r)--(1);
    \end{tikzpicture}, \enspace
    \emptyset, \enspace \emptyset}}.
\end{equation}
\end{subequations}
\medbreak

For any $\GeneratingSet$-tree, we denote by $\PrefixSet(\TreeT)$ the set
of all prefixes of~$\TreeT$.
\medbreak

\begin{Lemma} \label{lem:union_minimal_consistent_words}
    Let $\GeneratingSet$ be an alphabet and $\SetP$ be a subset of
    \begin{math}
        \SyntaxTrees(\GeneratingSet) \setminus \Bra{\Leaf}.
    \end{math}
    If $\TreeT$ is a $\GeneratingSet$-tree having its root labeled by
    $\GenA \in \GeneratingSet$ and prefix-avoiding $\SetP_\GenA$, then
    there is a minimal $\SetP_\GenA$-consistent word $\SetS$ such that
    $\TreeT$ is $\SetS$-admissible.
\end{Lemma}
\begin{proof}
    Let us denote by $k$ the arity of $\GenA$ and let
    $\SetS := \Par{\SetS_1, \dots, \SetS_k}$ be the word of subsets of
    $\SyntaxTrees(\GeneratingSet)$ defined by
    \begin{math}
        \SetS_i :=
        \SyntaxTrees(\GeneratingSet) \setminus \PrefixSet(\TreeT(i)).
    \end{math}
    Since $\TreeT$ prefix-avoids $\SetP_\GenA$, by
    Lemma~\ref{lem:prefix_recursive}, for any $\TreeR \in \SetP_\GenA$,
    there is an $i \in [k]$ such that $\TreeR(i) \ne \Leaf$ and
    $\TreeR(i) \NotPrefix \TreeT(i)$. This leads to the fact that
    $\TreeR(i) \notin \PrefixSet(\TreeT(i))$, so that
    $\TreeR(i) \in \SetS_i$. For this reason, $\SetS$ is
    $\SetP_\GenA$-consistent. Moreover, it follows directly from the
    definition of $\SetS$ that $\TreeT$ is $\SetS$-admissible. Finally,
    by definition of minimal $\SetP_\GenA$-consistent words, there
    exists a minimal $\SetP_\GenA$-consistent word
    $\SetS' := \Par{\SetS'_1, \dots, \SetS'_k}$ such that
    $\SetS'_i \subseteq \SetS_i$ for all $i \in [k]$. The statement of
    the lemma follows.
\end{proof}
\medbreak

By combining Lemmas~\ref{lem:consistent_words}
and~\ref{lem:union_minimal_consistent_words} together, it follows that
for any subset $\SetP$ of
\begin{math}
    \SyntaxTrees(\GeneratingSet) \setminus \Bra{\Leaf}
\end{math}
and any letter $\GenA \in \GeneratingSet$, a $\GeneratingSet$-tree
$\TreeT$ having its root labeled by $\GenA$ prefix-avoids $\SetP$ if and
only if there exists a minimal $\SetP_\GenA$-consistent word $\SetS$
such that $\TreeT$ is $\SetS$-admissible.
\medbreak

\begin{Lemma} \label{lem:admissible_words_factor_prefix}
    Let $\GeneratingSet$ be an alphabet, $\SetP$ and $\SetQ$ be two
    subsets of
    \begin{math}
        \SyntaxTrees(\GeneratingSet) \setminus \Bra{\Leaf},
    \end{math}
    and $\TreeT$ be a $\GeneratingSet$-tree having its root labeled by
    $\GenA \in \GeneratingSet(k)$. Then, $\TreeT$ factor-avoids $\SetP$
    and prefix-avoids $\SetQ$ if and only if for all $i \in [k]$,
    $\TreeT(i)$ factor-avoid $\SetP$ and there exists a minimal
    $(\SetP \cup \SetQ)_\GenA$-consistent word $\SetS$ such that
    $\TreeT$ is $\SetS$-admissible.
\end{Lemma}
\begin{proof}
    Assume that $\TreeT$ factor-avoids $\SetP$ and prefix-avoids
    $\SetQ$. The fact that $\TreeT$ factor-avoids $\SetP$ implies in
    particular that $\TreeT$ prefix-avoids $\SetP$ (see
    Proposition~\ref{prop:posets_factors_prefixes}). Hence, $\TreeT$
    prefix-avoids $\SetP \cup \SetQ$. Now, by
    Lemma~\ref{lem:union_minimal_consistent_words}, and since the root
    of $\TreeT$ is labeled by $\GenA$, there exists a minimal
    $\Par{\SetP \cup \SetQ}_\GenA$-consistent word $\SetS$ such that
    $\TreeT$ is $\SetS$-admissible. Conversely, assume that for all
    $i \in [k]$, $\TreeT(i)$ factor-avoid $\SetP$ and that there exists
    a minimal $(\SetP \cup \SetQ)_\GenA$-consistent word $\SetS$ such
    that $\TreeT$ is $\SetS$-admissible. By
    Lemma~\ref{lem:consistent_words}, $\TreeT$ prefix-avoids
    $(\SetP \cup \SetQ)_\GenA$. Therefore, since $\TreeT$ prefix-avoids
    $\SetP$ and since for each $i \in [k]$, $\TreeT(i)$ factor-avoids
    $\SetP$, we have that $\TreeT$ factor-avoids $\SetP$. Since
    moreover $\TreeT$ prefix-avoids $\SetQ$, we finally have that
    $\TreeT$ factor-avoids $\SetP$ and prefix-avoids~$\SetQ$.
\end{proof}
\medbreak

\subsubsection{Equations for tree series}
For any subsets $\SetP$ and $\SetQ$ of
\begin{math}
    \SyntaxTrees(\GeneratingSet) \setminus \Bra{\Leaf},
\end{math}
let $\Predicate_{\SetP, \SetQ}$ be the predicate on
$\SyntaxTrees(\GeneratingSet)$ wherein
$\Predicate_{\SetP, \SetQ}(\TreeT)$ holds if and only if
$\SetP \NotFactor \TreeT$ and $\SetQ \NotPrefix \TreeT$. Let also
$\FactorPrefixSeries(\SetP, \SetQ)$ be the $\GeneratingSet$-tree series
defined by
\begin{equation}
    \FactorPrefixSeries(\SetP, \SetQ)
    := \PredicateSeries{\Predicate_{\SetP, \SetQ}}.
\end{equation}
In other terms, $\FactorPrefixSeries(\SetP, \SetQ)$ is the
characteristic series of all $\GeneratingSet$-trees factor-avoiding all
trees of $\SetP$ and prefix-avoiding all trees of $\SetQ$. Since
\begin{math}
    \FactorPrefixSeries(\SetP, \emptyset) = \FactorSeries(\SetP),
\end{math}
we can regard $\FactorPrefixSeries(\SetP, \SetQ)$ as a refinement
of~$\FactorSeries(\SetP)$. Observe also that
\begin{math}
    \FactorPrefixSeries\Par{\SetP, \SetP'} = \FactorSeries(\SetP)
\end{math}
for all subsets $\SetP'$ of $\SetP$. As a side remark, observe that
$\FactorPrefixSeries(\emptyset, \SetQ)$ is the characteristic series of
the $\GeneratingSet$-trees prefix-avoiding~$\SetQ$.
\medbreak

\begin{Theorem} \label{thm:system_trees_avoiding}
    Let $\GeneratingSet$ be an alphabet, and $\SetP$ and $\SetQ$ be two
    subsets of
    \begin{math}
        \SyntaxTrees(\GeneratingSet) \setminus \Bra{\Leaf}
    \end{math}
    such that for any $\GenA \in \GeneratingSet$, there are finitely
    many minimal $(\SetP \cup \SetQ)_\GenA$-consistent words. The
    $\GeneratingSet$-tree series $\FactorPrefixSeries(\SetP, \SetQ)$
    satisfies
    \begin{equation} \label{equ:system_trees_avoiding}
        \FactorPrefixSeries(\SetP, \SetQ) =
        \Leaf
        +
        \sum_{\substack{
            k \geq 1 \\
            \GenA \in \GeneratingSet(k)
        }}
        \enspace
        \sum_{\substack{
            \ell \geq 1 \\
            \Bra{\SetR^{(1)}, \dots, \SetR^{(\ell)}} \subseteq
            \MinimalConsistent\Par{(\SetP \cup \SetQ)_\GenA} \\
            \Par{\SetS_1, \dots, \SetS_k} =
            \SetR^{(1)} \WordsSum \cdots \WordsSum \SetR^{(\ell)}
        }}
        (-1)^{1 + \ell}
        \;
        \GenA \Composition
        \Han{
        \FactorPrefixSeries\Par{\SetP, \SetS_1}, \dots,
        \FactorPrefixSeries\Par{\SetP, \SetS_k}}.
    \end{equation}
\end{Theorem}
\begin{proof}
    For any $\GenA \in \GeneratingSet(k)$ and any
    $\SetS \in \MinimalConsistent\Par{(\SetP \cup \SetQ)_\GenA}$, let
    $\Predicate_{\GenA, \SetS}$ be the predicate on
    $\SyntaxTrees(\GeneratingSet)$ wherein
    $\Predicate_{\GenA, \SetS}(\TreeT)$ holds if and only if
    $\SetP \NotFactor \TreeT$, $\SetQ \NotPrefix \TreeT$, and $\TreeT$
    is $\SetS$-admissible. As a consequence of
    Lemma~\ref{lem:admissible_words_factor_prefix}, we have
    \begin{equation}
        \PredicateSeries{\Predicate_{\GenA, \SetS}} =
        \GenA \Composition
        \Han{\FactorPrefixSeries\Par{\SetP, \SetS_1}, \dots,
        \FactorPrefixSeries\Par{\SetP, \SetS_k}}.
    \end{equation}
    Now, observe that for any
    \begin{math}
        \SetS, \SetS'
        \in
        \MinimalConsistent\Par{(\SetP \cup \SetQ)_\GenA},
    \end{math}
    the predicates $\Predicate_{\GenA, \SetS \WordsSum \SetS'}$ and
    $\Predicate_{\GenA, \SetS} \wedge \Predicate_{\GenA, \SetS'}$ are
    equal. Observe also that  the characteristic series
    $\SeriesF_\GenA$ of the $\GeneratingSet$-trees factor-avoiding
    $\SetP$, prefix-avoiding $\SetQ$, and with a root labeled by
    $\GenA$, satisfies
    \begin{equation}
        \SeriesF_\GenA =
        \PredicateSeries{\bigvee_{
            \SetS \in \MinimalConsistent\Par{(\SetP \cup \SetQ)_\GenA
        }}
        \Predicate_{\GenA, \SetS}}.
    \end{equation}
    Since, by hypothesis,
    $\MinimalConsistent\Par{(\SetP \cup \SetQ)_\GenA}$ is finite, these
    three previous properties lead, by using
    Lemma~\ref{lem:series_predicate}, to the relation
    \begin{equation}
        \SeriesF_\GenA
        =
        \sum_{\substack{
            \ell \geq 1 \\
            \Bra{\SetR^{(1)}, \dots, \SetR^{(\ell)}} \subseteq
            \MinimalConsistent\Par{(\SetP \cup \SetQ)_\GenA} \\
            \Par{\SetS_1, \dots, \SetS_k} =
            \SetR^{(1)} \WordsSum \cdots \WordsSum \SetR^{(\ell)}
        }}
        (-1)^{1 + \ell}
        \;
        \GenA \Composition
        \Han{
        \FactorPrefixSeries\Par{\SetP, \SetS_1}, \dots,
        \FactorPrefixSeries\Par{\SetP, \SetS_k}}.
    \end{equation}
    Finally, since any tree factor-avoiding $\SetP$ and prefix-avoiding
    $\SetQ$ can be either empty of have a root labeled by $\GenA$ for
    any $\GenA \in \GeneratingSet$, we have
    \begin{equation}
        \FactorPrefixSeries(\SetP, \SetQ) =
        \Leaf + \sum_{\GenA \in \GeneratingSet} \SeriesF_\GenA.
    \end{equation}
    This last relation shows that~\eqref{equ:system_trees_avoiding}
    holds.
\end{proof}
\medbreak

Let us consider an example brought by
Theorem~\ref{thm:system_trees_avoiding} by considering the
set~\eqref{equ:example_pattern_set} of patterns. We have
\begin{footnotesize}
\begin{multline} \label{equ:example_series_avoiding}
    \FactorPrefixSeries(\SetP, \emptyset) =
    \Leaf
    + \GenA \Composition \Han{
    \FactorPrefixSeries\Par{\SetP, \CorollaThree{\GenC}},
    \FactorPrefixSeries\Par{\SetP, \emptyset}}
    + \GenB \Composition \Han{
    \FactorPrefixSeries\Par{\SetP, \emptyset},
    \FactorPrefixSeries\Par{\SetP, \emptyset}} \\
    + \GenC \Composition \Han{
    \FactorPrefixSeries\Par{\SetP, \CorollaTwo{\GenA}},
    \FactorPrefixSeries\Par{\SetP, \CorollaTwo{\GenB}},
    \FactorPrefixSeries\Par{\SetP, \CorollaTwo{\GenA}}}
    + \GenC \Composition \Han{
    \FactorPrefixSeries\Par{\SetP, \CorollaTwo{\GenA}
        \CorollaTwo{\GenB}},
    \FactorPrefixSeries\Par{\SetP, \emptyset},
    \FactorPrefixSeries\Par{\SetP, \CorollaTwo{\GenA}}} \\
    + \GenC \Composition \Han{
    \FactorPrefixSeries\Par{\SetP, \CorollaTwo{\GenA} \CorollaTwo{\GenB}
    \begin{tikzpicture}[Centering,xscale=0.18,yscale=0.17]
        \node(0)at(0.00,-2.33){};
        \node(2)at(1.00,-2.33){};
        \node(3)at(2.00,-4.67){};
        \node(5)at(3.00,-4.67){};
        \node(6)at(4.00,-4.67){};
        \node[NodeST](1)at(1.00,0.00){$\GenC$};
        \node[NodeST](4)at(3.00,-2.33){$\GenC$};
        \draw[Edge](0)--(1);
        \draw[Edge](2)--(1);
        \draw[Edge](3)--(4);
        \draw[Edge](4)--(1);
        \draw[Edge](5)--(4);
        \draw[Edge](6)--(4);
        \node(r)at(1.00,2){};
        \draw[Edge](r)--(1);
    \end{tikzpicture}},
    \FactorPrefixSeries\Par{\SetP, \emptyset},
    \FactorPrefixSeries\Par{\SetP, \emptyset}}
    - \GenC \Composition \Han{
    \FactorPrefixSeries\Par{\SetP,
        \CorollaTwo{\GenA} \CorollaTwo{\GenB}},
    \FactorPrefixSeries\Par{\SetP, \CorollaTwo{\GenB}},
    \FactorPrefixSeries\Par{\SetP, \CorollaTwo{\GenA}}} \\
    - \GenC \Composition \Han{
    \FactorPrefixSeries\Par{\SetP, \CorollaTwo{\GenA} \CorollaTwo{\GenB}
    \begin{tikzpicture}[Centering,xscale=0.18,yscale=0.17]
        \node(0)at(0.00,-2.33){};
        \node(2)at(1.00,-2.33){};
        \node(3)at(2.00,-4.67){};
        \node(5)at(3.00,-4.67){};
        \node(6)at(4.00,-4.67){};
        \node[NodeST](1)at(1.00,0.00){$\GenC$};
        \node[NodeST](4)at(3.00,-2.33){$\GenC$};
        \draw[Edge](0)--(1);
        \draw[Edge](2)--(1);
        \draw[Edge](3)--(4);
        \draw[Edge](4)--(1);
        \draw[Edge](5)--(4);
        \draw[Edge](6)--(4);
        \node(r)at(1.00,2){};
        \draw[Edge](r)--(1);
    \end{tikzpicture}},
    \FactorPrefixSeries\Par{\SetP, \CorollaTwo{\GenB}},
    \FactorPrefixSeries\Par{\SetP, \CorollaTwo{\GenA}}}
    - \GenC \Composition \Han{
    \FactorPrefixSeries\Par{\SetP, \CorollaTwo{\GenA} \CorollaTwo{\GenB}
    \begin{tikzpicture}[Centering,xscale=0.18,yscale=0.17]
        \node(0)at(0.00,-2.33){};
        \node(2)at(1.00,-2.33){};
        \node(3)at(2.00,-4.67){};
        \node(5)at(3.00,-4.67){};
        \node(6)at(4.00,-4.67){};
        \node[NodeST](1)at(1.00,0.00){$\GenC$};
        \node[NodeST](4)at(3.00,-2.33){$\GenC$};
        \draw[Edge](0)--(1);
        \draw[Edge](2)--(1);
        \draw[Edge](3)--(4);
        \draw[Edge](4)--(1);
        \draw[Edge](5)--(4);
        \draw[Edge](6)--(4);
        \node(r)at(1.00,2){};
        \draw[Edge](r)--(1);
    \end{tikzpicture}},
    \FactorPrefixSeries\Par{\SetP, \emptyset},
    \FactorPrefixSeries\Par{\SetP, \CorollaTwo{\GenA}}} \\
    + \GenC \Composition \Han{
    \FactorPrefixSeries\Par{\SetP, \CorollaTwo{\GenA} \CorollaTwo{\GenB}
    \begin{tikzpicture}[Centering,xscale=0.18,yscale=0.17]
        \node(0)at(0.00,-2.33){};
        \node(2)at(1.00,-2.33){};
        \node(3)at(2.00,-4.67){};
        \node(5)at(3.00,-4.67){};
        \node(6)at(4.00,-4.67){};
        \node[NodeST](1)at(1.00,0.00){$\GenC$};
        \node[NodeST](4)at(3.00,-2.33){$\GenC$};
        \draw[Edge](0)--(1);
        \draw[Edge](2)--(1);
        \draw[Edge](3)--(4);
        \draw[Edge](4)--(1);
        \draw[Edge](5)--(4);
        \draw[Edge](6)--(4);
        \node(r)at(1.00,2){};
        \draw[Edge](r)--(1);
    \end{tikzpicture}},
    \FactorPrefixSeries\Par{\SetP, \CorollaTwo{\GenB}},
    \FactorPrefixSeries\Par{\SetP, \CorollaTwo{\GenA}}}.
\end{multline}
\end{footnotesize}%
Observe that the last term of~\eqref{equ:example_series_avoiding} is the
opposite of the antepenultimate term so that they annihilate.
\medbreak

\subsection{Properties and applications}
Consequences of Theorem~\ref{thm:system_trees_avoiding} are now
presented. In particular, we explain how to obtain a system of
equations of generating series to enumerate the syntax trees
factor-avoiding a set $\SetP$ of patterns and prefix-avoiding a set
$\SetQ$ of patterns. We also apply the aforementioned result for
particular sets of patterns consisting in stringy trees.
\medbreak

\subsubsection{Systems of equations}
Given two subsets $\SetP$ and $\SetQ$ of
\begin{math}
    \SyntaxTrees(\GeneratingSet) \setminus \Bra{\Leaf}
\end{math}
satisfying the conditions of Theorem~\ref{thm:system_trees_avoiding},
one can express the series $\FactorPrefixSeries(\SetP, \SetQ)$
through~\eqref{equ:system_trees_avoiding}. Some other series
$\FactorPrefixSeries\Par{\SetP, \SetS_i}$ could appear in the
expression, and these series can themselves be expressed
through~\eqref{equ:system_trees_avoiding} when the conditions of the
theorem are satisfied. When it is the case,
Theorem~\ref{thm:system_trees_avoiding} leads to a (possibly infinite)
system of equations describing the
series~$\FactorPrefixSeries(\SetP, \SetQ)$, called the \Def{system}
of~$\FactorPrefixSeries(\SetP, \SetQ)$.
\medbreak

\begin{Lemma} \label{lem:finite_set_minimal_consistent_words}
    Let $\GeneratingSet$ be an alphabet, $\SetP$ be a subset of
    \begin{math}
        \SyntaxTrees(\GeneratingSet) \setminus \Bra{\Leaf},
    \end{math}
    and $\GenA \in \GeneratingSet(k)$. If $\SetP_\GenA$ is finite,
    then the set of all minimal $\SetP_\GenA$-consistent words is
    finite and its cardinality is no greater than~$k^{\# \SetP_\GenA}$.
\end{Lemma}
\begin{proof}
    We proceed by induction on the cardinality $\ell$ of $\SetP_\GenA$.
    If $\ell = 0$, the only $\SetP_\GenA$-consistent word is the word
    $\Par{\SetS_1, \dots, \SetS_k}$ such that $\SetS_i := \emptyset$ for
    all $i \in [k]$. Hence, the statement of the lemma holds in this
    case. Assume now that the statement of the lemma holds when
    $\SetP_\GenA$ has cardinality $\ell$. Let $\TreeS$ be a
    $\GeneratingSet$-tree having its root labeled by $\GenA$. If
    $\SetS := \Par{\SetS_1, \dots, \SetS_k}$ is a
    $\SetP_\GenA$-consistent word, when $j \in [k]$ is an index such
    that $\TreeS(j) \ne \Leaf$, let us denote by
    $\SetS^{(j)} := \Par{\SetS'_1, \dots, \SetS'_k}$ the word defined by
    $\SetS'_j := \SetS_j \cup \Bra{\TreeS(j)}$ and $\SetS'_i := \SetS_i$
    for any $i \in [k] \setminus \{j\}$. By construction, $\SetS^{(j)}$
    is a minimal $\Par{\SetP \cup \{\TreeS\}}_\GenA$-consistent word and
    there are at most $k$ such words. By induction hypothesis, there are
    at most $k^\ell$ minimal $\SetP_\GenA$-consistent words and
    therefore, at most $k^{\ell + 1}$ minimal
    $\Par{\SetP \cup \{\TreeS\}}_\GenA$-consistent words.
\end{proof}
\medbreak

For any $\GeneratingSet$-tree, we denote by $\SuffixSet(\TreeT)$ the set
of all suffixes of~$\TreeT$.
\medbreak

\begin{Proposition} \label{prop:system_trees_avoiding_finiteness}
    Let $\GeneratingSet$ be an alphabet, and $\SetP$ and $\SetQ$ be two
    subsets of
    \begin{math}
        \SyntaxTrees(\GeneratingSet) \setminus \Bra{\Leaf}.
    \end{math}
    If $\SetP$ and $\SetQ$ are finite, then the system of
    $\FactorPrefixSeries(\SetP, \SetQ)$ is well-defined and contains
    finitely many equations.
\end{Proposition}
\begin{proof}
    Let $\GenA \in \GeneratingSet(k)$. Since $\SetP$ and $\SetQ$ are
    finite, $(\SetP \cup \SetQ)_\GenA$ is finite. Therefore, by
    Lemma~\ref{lem:finite_set_minimal_consistent_words},
    $\MinimalConsistent\Par{(\SetP \cup \SetQ)_\GenA}$ is finite.
    Moreover, any minimal $(\SetP \cup \SetQ)_\GenA$-consistent word
    $\Par{\SetS_1, \dots, \SetS_k}$ is such that each $\SetS_i$,
    $i \in [k]$, contains only suffixes of trees of
    $(\SetP \cup \SetQ)_\GenA$. For this reason, all terms
    $\FactorPrefixSeries\Par{\SetP, \SetS_i}$ appearing in the
    equation~\eqref{equ:system_trees_avoiding} of
    $\FactorPrefixSeries(\SetP, \SetQ)$ satisfy
    \begin{equation} \label{equ:system_trees_avoiding_finiteness}
        \SetS_i \subseteq
        \bigcup_{\TreeT \in \SetP \cup \SetQ} \SuffixSet(\TreeT).
    \end{equation}
    Since any $\GeneratingSet$-tree has a finite number of suffixes,
    there are finitely many sets $\SetS_i$
    satisfying~\eqref{equ:system_trees_avoiding_finiteness}. The
    statement of the proposition follows.
\end{proof}
\medbreak

\subsubsection{Limits}
Let $\SetP$ be a subset of
\begin{math}
    \SyntaxTrees(\GeneratingSet) \setminus \Bra{\Leaf}.
\end{math}
For any integer $d \geq 0$, let
\begin{equation}
    \SetP_{|d} := \Bra{\TreeT \in \SetP : \Deg(\TreeT) \leq d}.
\end{equation}
In other words, $\SetP_{|d}$ is the subset of $\SetP$ of the patterns
having degrees no greater than~$d$.
\medbreak

\begin{Proposition} \label{prop:series_trees_avoiding_limit}
    Let $\GeneratingSet$ be an alphabet, and $\SetP$ and $\SetQ$ be two
    subsets of
    \begin{math}
        \SyntaxTrees(\GeneratingSet) \setminus \Bra{\Leaf}.
    \end{math}
    Then,
    \begin{equation}
        \lim_{d \to \infty}
        \FactorPrefixSeries\Par{\SetP_{|d}, \SetQ_{|d}}
        =
        \FactorPrefixSeries(\SetP, \SetQ).
    \end{equation}
\end{Proposition}
\begin{proof}
    Since any $\GeneratingSet$-tree $\TreeT$ factor-avoids (resp.
    prefix-avoids) all patterns of degrees greater than $\Deg(\TreeT)$,
    for any $d \geq \Deg(\TreeT)$,
    \begin{equation}
        \Angle{\TreeT, \FactorPrefixSeries\Par{\SetP, \SetQ}}
        =
        \Angle{\TreeT, \FactorPrefixSeries\Par{\SetP_{|d}, \SetQ_{|d}}}.
    \end{equation}
    This implies that the coefficients of the series
    $\FactorPrefixSeries\Par{\SetP, \SetQ}$ and
    $\FactorPrefixSeries\Par{\SetP_{|d}, \SetQ_{|d}}$ coincide for all
    the $\GeneratingSet$-trees of degrees no greater than $d$. The
    statement of the proposition follows.
\end{proof}
\medbreak

Theorem~\ref{thm:system_trees_avoiding} and
Proposition~\ref{prop:series_trees_avoiding_limit} allow us together to
obtain systems of equations for $\FactorPrefixSeries\Par{\SetP, \SetQ}$
even when $\SetP$ and $\SetQ$ are infinite subsets of
\begin{math}
    \SyntaxTrees(\GeneratingSet) \setminus \Bra{\Leaf}
\end{math}
that do not satisfy the hypothesis of
Theorem~\ref{thm:system_trees_avoiding}.
\medbreak

\subsubsection{Generating series and systems of equations}
For any subset $\SetP$ of
\begin{math}
    \SyntaxTrees(\GeneratingSet) \setminus \Bra{\Leaf},
\end{math}
let $\FactorEnumSeries(\SetP)$ be the series of
$\K \AAngle{t, q, \AlphabetQ_\GeneratingSet}$ defined by
\begin{math}
    \FactorEnumSeries(\SetP) := \Enum(\FactorSeries(\SetP)).
\end{math}
In the same way, for any subsets $\SetP$ and $\SetQ$ of
\begin{math}
    \SyntaxTrees(\GeneratingSet) \setminus \Bra{\Leaf},
\end{math}
let $\FactorPrefixEnumSeries(\SetP, \SetQ)$ be the series of
$\K \AAngle{t, q, \AlphabetQ_\GeneratingSet}$ defined by
\begin{math}
    \FactorPrefixEnumSeries(\SetP, \SetQ)
    := \Enum(\FactorPrefixSeries(\SetP, \SetQ)).
\end{math}
The series $\FactorEnumSeries(\SetP)$ is the generating series of the
set of the $\GeneratingSet$-trees factor-avoiding $\SetP$, and
$\FactorPrefixEnumSeries(\SetP, \SetQ)$ is the generating series of the
set of the $\GeneratingSet$-trees factor-avoiding $\SetP$ and
prefix-avoiding~$\SetQ$.
\medbreak

\begin{Proposition} \label{prop:system_enumeration_trees_avoiding}
    Let $\GeneratingSet$ be an alphabet, and $\SetP$ and $\SetQ$ be two
    subsets of
    \begin{math}
        \SyntaxTrees(\GeneratingSet) \setminus \Bra{\Leaf}
    \end{math}
    such that for any $\GenA \in \GeneratingSet$,
    $(\SetP \cup \SetQ)_\GenA$ is finite.  The generating series
    $\FactorPrefixEnumSeries(\SetP, \SetQ)$ satisfies
    \begin{equation} \label{equ:system_enumeration_trees_avoiding}
        \FactorPrefixEnumSeries(\SetP, \SetQ) =
        t
        +
        q
        \sum_{\substack{
            k \geq 1 \\
            \GenA \in \GeneratingSet(k)
        }}
        \enspace
        q_\GenA
        \sum_{\substack{
            \ell \geq 1 \\
            \Bra{\SetR^{(1)}, \dots, \SetR^{(\ell)}} \subseteq
            \MinimalConsistent\Par{(\SetP \cup \SetQ)_\GenA} \\
            \Par{\SetS_1, \dots, \SetS_k} =
            \SetR^{(1)} \WordsSum \cdots \WordsSum \SetR^{(\ell)}
        }}
        (-1)^{1 + \ell} \;
        \prod_{i \in [k]} \FactorPrefixEnumSeries\Par{\SetP, \SetS_i}.
    \end{equation}
\end{Proposition}
\begin{proof}
    Relation~\eqref{equ:system_enumeration_trees_avoiding} is obtained
    by considering the enumerative images of the left and right members
    of~\eqref{equ:system_trees_avoiding} provided by
    Theorem~\ref{thm:system_trees_avoiding}, together with
    Proposition~\ref{prop:enumeration_composition}.
\end{proof}
\medbreak

\subsubsection{Avoiding stringy trees}
A $\GeneratingSet$-tree $\TreeT$ is \Def{stringy} if the height of
$\TreeT$ is equal to the degree of $\TreeT$. This is equivalent to the
fact that any internal node of $\TreeT$ has at most one child being an
internal node.
\medbreak

For any set $\SetP$ of $\GeneratingSet$-trees,
$\GenA \in \GeneratingSet(k)$, and $i \in [k]$, let
\begin{equation}
    \partial_{\GenA, i}(\SetP) :=
    \Bra{\TreeS \in \SyntaxTrees(\GeneratingSet) :
    \GenA \circ_i \TreeS \in \SetP}.
\end{equation}
In other words, $\partial_{\GenA, i}(\SetP)$ is the set of the
$\GeneratingSet$-trees obtained by keeping the $i$-th subtrees of the
trees whose roots are labeled by $\GenA$ in~$\SetP$.
\medbreak

\begin{Proposition} \label{prop:system_avoiding_stringy}
    Let $\GeneratingSet$ be an alphabet and $\SetP$ and $\SetQ$ be two
    subsets of
    \begin{math}
        \SyntaxTrees(\GeneratingSet) \setminus \Bra{\Leaf}
    \end{math}
    consisting only in stringy trees. The $\GeneratingSet$-tree series
    $\FactorPrefixSeries(\SetP, \SetQ)$ satisfies
    \begin{equation} \label{equ:system_avoiding_stringy}
        \FactorPrefixSeries(\SetP, \SetQ) =
        \Leaf
        +
        \sum_{\substack{
            k \geq 1 \\
            \GenA \in \GeneratingSet(k) \\
            \Corolla{\GenA} \notin \SetP \cup \SetQ
        }}
        \GenA
        \Composition \Han{
        \FactorPrefixSeries\Par{
            \SetP, \partial_{\GenA, 1}(\SetP \cup \SetQ)},
        \dots,
        \FactorPrefixSeries\Par{
            \SetP, \partial_{\GenA, k}(\SetP \cup \SetQ)}}.
    \end{equation}
\end{Proposition}
\begin{proof}
    Let $\GenA \in \GeneratingSet(k)$. When $\Corolla{\GenA}$ is in
    $\SetP \cup \SetQ$, by definition of consistent words, there is no
    $(\SetP \cup \SetQ)_\GenA$-consistent word. When $\Corolla{\GenA}$
    is not in $\SetP \cup \SetQ$, by definition of minimal consistent
    words, the only minimal $(\SetP \cup \SetQ)_\GenA$-consistent word
    is the word $\SetS := \Par{\SetS_1, \dots, \SetS_k}$ where
    $\SetS_i := \partial_{\GenA, i}(\SetP \cup \SetQ)$ for any
    $i \in [k]$. Now, \eqref{equ:system_avoiding_stringy} is a
    consequence of Theorem~\ref{thm:system_trees_avoiding}.
\end{proof}
\medbreak

Let us call \Def{$\GeneratingSet$-word} any $\GeneratingSet$-tree where
$\GeneratingSet$ is an alphabet concentrated in arity $1$. This
designation is justified by the fact that one can encode any word
$\GenA_1 \dots \GenA_d$ on $\GeneratingSet$ through the tree
$\GenA_1 \circ_1 \dots \circ_1 \GenA_d$. When $\SetP$ contains only
$\GeneratingSet$-words different from the leaf, $\SetP$ specifies
forbidden configurations of word factors. Since a $\GeneratingSet$-word
is obviously stringy, Proposition~\ref{prop:system_avoiding_stringy}
provides in this context a system of equations to describe the series of
words avoiding factors. This problem consisting in enumerating words
avoiding as factors a given set was originally stated and solved
in~\cite{GJ79} (see also~\cite{NZ99}).
\medbreak

Besides, when $\GeneratingSet$ is any alphabet, let us call
\Def{$\GeneratingSet$-edge} any $\GeneratingSet$-tree of degree $2$.
This appellation is justified by the fact that any tree of degree $2$
contains exactly one edge connecting two internal nodes. When $\SetP$
contains only $\GeneratingSet$-edges, $\SetP$ specifies forbidden
configurations of edges. Since a $\GeneratingSet$-edge is obviously
stringy, Proposition~\ref{prop:system_avoiding_stringy} provides in this
context a system of equations to describe the series of trees avoiding
edges. This particular case of pattern avoidance in trees was
studied in~\cite{Lod05} (see also~\cite{Par93}).
\medbreak

\subsubsection{Sets of patterns for some algebraic series}
Let us assume here that $\K$ is the field~$\Q$. A series $\SeriesF$ of
$\K \AAngle{t}$ is \Def{$\N$-algebraic} if $\SeriesF$ satisfies the
equation
\begin{equation} \label{equ:n_algebraic_series}
    \SeriesF = \sum_{0 \leq n \leq d} P_n \; \SeriesF^n
\end{equation}
where $d$ is a certain nonnegative integer, for all $0 \leq n \leq d$,
the $P_n$ are polynomials of $\Q \Angle{t}$ having all coefficients in
$\N$, and $\Angle{t^0, P_1} = 0$. For instance, the series $\SeriesF$
satisfying
\begin{equation} \label{equ:example_n_algebraic_series}
    \SeriesF = t + t^3 + \Par{t + t^2} \SeriesF
    + \Par{1 + 2 t^3} \SeriesF^2
\end{equation}
is $\N$-algebraic.
\medbreak

\begin{Proposition} \label{prop:realization_n_algebraic_series}
    Let $\SeriesF$ be an $\N$-algebraic series of the
    form~\eqref{equ:n_algebraic_series} such that $\Angle{t^0, P_0} = 0$
    and $\Angle{t^1, P_0} = 1$. Let the alphabet
    \begin{math}
        \GeneratingSet := \bigsqcup_{n \geq 2} \GeneratingSet(n)
    \end{math}
    where, for any $n \geq 2$,
    \begin{equation}
        \GeneratingSet(n) :=
        \bigsqcup_{\substack{
            k, \ell \geq 0 \\
            k + \ell = n}}
        \Bra{\GenA_{k, \ell}^{(m)} : 1 \leq m \leq \Angle{t^\ell, P_k}}
    \end{equation}
    and the set of patterns
    \begin{equation}
        \SetP :=
        \bigsqcup_{\substack{
            \GenA_{k, \ell}^{(m)} \in \GeneratingSet \\
            i \in [\ell]
        }}
        \Bra{\GenA_{k, \ell}^{(m)} \circ_i \GenB
        : \GenB \in \GeneratingSet}.
    \end{equation}
    The specialization
    \begin{math}
        \FactorPrefixEnumSeries(\SetP, \emptyset)_
            {|q := 1, q_\GenA := 1, \GenA \in \GeneratingSet}
    \end{math}
    satisfies the same algebraic equation as the one satisfied
    by~$\SeriesF$.
\end{Proposition}
\begin{proof}
    Observe that $\SetP$ contains only stringy trees. Therefore, the
    characteristic series  $\FactorPrefixSeries(\SetP, \emptyset)$ of
    the trees factor-avoiding $\SetP$ is described by
    Proposition~\ref{prop:system_avoiding_stringy} and satisfies
    \begin{equation}
        \FactorPrefixSeries(\SetP, \emptyset) =
        \Leaf
        + \sum_{\GenA_{k, \ell}^{(m)} \in \GeneratingSet}
        \GenA_{k, \ell}^{(m)} \Composition
        \Han{
        \underbrace{\FactorPrefixSeries\Par{\SetP, \SetQ}, \dots,
        \FactorPrefixSeries\Par{\SetP, \SetQ}}_{\times \ell},
        \underbrace{\FactorPrefixSeries(\SetP, \emptyset), \dots,
        \FactorPrefixSeries(\SetP, \emptyset)}_{\times k}}
    \end{equation}
    where $\SetQ := \Bra{\Corolla{\GenB} : \GenB \in \GeneratingSet}$.
    Now, due to
    Proposition~\ref{prop:system_enumeration_trees_avoiding},
    the enumerative image $\FactorPrefixEnumSeries(\SetP, \emptyset)$ of
    $\FactorPrefixSeries(\SetP, \emptyset)$ satisfies
    \begin{equation}
        \FactorPrefixEnumSeries(\SetP, \emptyset) =
        t
        + q \sum_{\GenA_{k, \ell}^{(m)} \in \GeneratingSet}
        q_{\GenA_{k, \ell}^{(m)}}
        \FactorPrefixEnumSeries\Par{\SetP, \SetQ}^\ell
        \FactorPrefixEnumSeries(\SetP, \emptyset)^k
    \end{equation}
    where $\FactorPrefixEnumSeries\Par{\SetP, \SetQ} = t$. The statement
    of the proposition follows.
\end{proof}
\medbreak

Observe that the alphabet $\GeneratingSet$ provided by
Proposition~\ref{prop:realization_n_algebraic_series} has
\begin{equation}
    \# \GeneratingSet =
    \sum_{\substack{
        k, \ell \geq 0 \\
        k + \ell \geq 2
    }}
    \Angle{t^\ell, P_k}
\end{equation}
letters, and the set $\SetP$ is made of
\begin{equation}
    \# \SetP =
    \Par{\# \GeneratingSet}
    \sum_{\substack{
        k \geq 0, \ell \geq 0 \\
        k + \ell \geq 2
    }}
    \Angle{t^\ell, P_k} \ell
\end{equation}
patterns.
\medbreak

Let us consider for example the series $\SeriesF$
of~\eqref{equ:example_n_algebraic_series}. The alphabet and set of
patterns specified by
Proposition~\ref{prop:realization_n_algebraic_series}, are
\begin{equation}
    \GeneratingSet := \Bra{
    \GenA_{0, 3}^{(1)}, \GenA_{1, 1}^{(1)}, \GenA_{1, 2}^{(1)},
    \GenA_{2, 0}^{(1)}, \GenA_{2, 3}^{(1)}, \GenA_{2, 3}^{(2)}}
\end{equation}
and
\begin{equation}
    \SetP := \Bra{
    \begin{tikzpicture}[Centering,xscale=0.24,yscale=0.24]
        \node(0)at(0.00,-4.67){};
        \node(2)at(1.00,-4.67){};
        \node(3)at(2.00,-4.67){};
        \node(5)at(3.00,-2.33){};
        \node(6)at(4.00,-2.33){};
        \node[NodeST](1)at(1.00,-2.33){$\GenA_{0, 3}^{(1)}$};
        \node[NodeST](4)at(3.00,0.00){$\GenA_{0, 3}^{(1)}$};
        \draw[Edge](0)--(1);
        \draw[Edge](1)--(4);
        \draw[Edge](2)--(1);
        \draw[Edge](3)--(1);
        \draw[Edge](5)--(4);
        \draw[Edge](6)--(4);
        \node(r)at(3.00,2.5){};
        \draw[Edge](r)--(4);
    \end{tikzpicture},
    \begin{tikzpicture}[Centering,xscale=0.24,yscale=0.28]
        \node(0)at(0.00,-4.00){};
        \node(2)at(2.00,-4.00){};
        \node(4)at(3.00,-2.00){};
        \node(5)at(4.00,-2.00){};
        \node[NodeST](1)at(1.00,-2.00){$\GenA_{1, 1}^{(1)}$};
        \node[NodeST](3)at(3.00,0.00){$\GenA_{0, 3}^{(1)}$};
        \draw[Edge](0)--(1);
        \draw[Edge](1)--(3);
        \draw[Edge](2)--(1);
        \draw[Edge](4)--(3);
        \draw[Edge](5)--(3);
        \node(r)at(3.00,2){};
        \draw[Edge](r)--(3);
    \end{tikzpicture},
    \dots,
    \begin{tikzpicture}[Centering,xscale=0.23,yscale=0.21]
        \node(0)at(0.00,-6.00){};
        \node(1)at(1.00,-6.00){};
        \node(3)at(2.00,-6.00){};
        \node(4)at(3.00,-6.00){};
        \node(5)at(4.00,-6.00){};
        \node(7)at(5.00,-3.00){};
        \node(8)at(6.00,-3.00){};
        \node[NodeST](2)at(2.00,-3.00){$\GenA_{2, 3}^{(2)}$};
        \node[NodeST](6)at(5.00,0.00){$\GenA_{0, 3}^{(1)}$};
        \draw[Edge](0)--(2);
        \draw[Edge](1)--(2);
        \draw[Edge](2)--(6);
        \draw[Edge](3)--(2);
        \draw[Edge](4)--(2);
        \draw[Edge](5)--(2);
        \draw[Edge](7)--(6);
        \draw[Edge](8)--(6);
        \node(r)at(5.00,2.5){};
        \draw[Edge](r)--(6);
    \end{tikzpicture},
    \dots,
    \begin{tikzpicture}[Centering,xscale=0.23,yscale=0.21]
        \node(0)at(0.00,-3.00){};
        \node(2)at(1.00,-3.00){};
        \node(3)at(2.00,-6.00){};
        \node(4)at(3.00,-6.00){};
        \node(6)at(4.00,-6.00){};
        \node(7)at(5.00,-6.00){};
        \node(8)at(6.00,-6.00){};
        \node[NodeST](1)at(1.00,0.00){$\GenA_{0, 3}^{(1)}$};
        \node[NodeST](5)at(4.00,-3.00){$\GenA_{2, 3}^{(2)}$};
        \draw[Edge](0)--(1);
        \draw[Edge](2)--(1);
        \draw[Edge](3)--(5);
        \draw[Edge](4)--(5);
        \draw[Edge](5)--(1);
        \draw[Edge](6)--(5);
        \draw[Edge](7)--(5);
        \draw[Edge](8)--(5);
        \node(r)at(1.00,2.5){};
        \draw[Edge](r)--(1);
    \end{tikzpicture},
    \dots,
    \begin{tikzpicture}[Centering,xscale=0.23,yscale=0.2]
        \node(0)at(0.00,-6.00){};
        \node(2)at(1.00,-6.00){};
        \node(3)at(2.00,-6.00){};
        \node(4)at(3.00,-3.00){};
        \node(6)at(4.00,-3.00){};
        \node(7)at(5.00,-3.00){};
        \node(8)at(6.00,-3.00){};
        \node[NodeST](1)at(1.00,-3.00){$\GenA_{0, 3}^{(1)}$};
        \node[NodeST](5)at(4.00,0.00){$\GenA_{2, 3}^{(2)}$};
        \draw[Edge](0)--(1);
        \draw[Edge](1)--(5);
        \draw[Edge](2)--(1);
        \draw[Edge](3)--(1);
        \draw[Edge](4)--(5);
        \draw[Edge](6)--(5);
        \draw[Edge](7)--(5);
        \draw[Edge](8)--(5);
        \node(r)at(4.00,2.75){};
        \draw[Edge](r)--(5);
    \end{tikzpicture},
    \dots,
    \begin{tikzpicture}[Centering,xscale=0.22,yscale=0.18]
        \node(0)at(0.00,-3.67){};
        \node(1)at(1.00,-3.67){};
        \node(10)at(8.00,-3.67){};
        \node(3)at(2.00,-7.33){};
        \node(4)at(3.00,-7.33){};
        \node(6)at(4.00,-7.33){};
        \node(7)at(5.00,-7.33){};
        \node(8)at(6.00,-7.33){};
        \node(9)at(7.00,-3.67){};
        \node[NodeST](2)at(4.00,0.00){$\GenA_{2, 3}^{(2)}$};
        \node[NodeST](5)at(4.00,-3.67){$\GenA_{2, 3}^{(2)}$};
        \draw[Edge](0)--(2);
        \draw[Edge](1)--(2);
        \draw[Edge](10)--(2);
        \draw[Edge](3)--(5);
        \draw[Edge](4)--(5);
        \draw[Edge](5)--(2);
        \draw[Edge](6)--(5);
        \draw[Edge](7)--(5);
        \draw[Edge](8)--(5);
        \draw[Edge](9)--(2);
        \node(r)at(4.00,3){};
        \draw[Edge](r)--(2);
    \end{tikzpicture}}.
\end{equation}
The cardinality of $\SetP$ is
\begin{math}
    6 \times \Par{1 \times 3 + 1 \times 1 + 1 \times 2 + 2 \times 3}
    = 72.
\end{math}
\medbreak

\subsubsection{Examples}
Let us consider some complete examples of systems.
\begin{enumerate}[fullwidth,label={\it $\bullet$ Example \arabic*.}]
\item Let the alphabet
\begin{math}
    \GeneratingSet := \GeneratingSet(2) := \Bra{\GenA_i : i \in \N}
\end{math}
and the set of patterns
\begin{equation}
    \SetP :=
    \Bra{\TreeLeft{\GenA_i}{\GenA_i} : i \in \N}.
\end{equation}
By Proposition~\ref{prop:system_avoiding_stringy}, we obtain the system
\begin{subequations}
\begin{equation}
    \FactorPrefixSeries\Par{\SetP, \emptyset} =
    \Leaf + \sum_{i \in \N} \GenA_i \Composition
    \Han{\FactorPrefixSeries\Par{\SetP, \CorollaTwo{\GenA_i}},
    \FactorPrefixSeries\Par{\SetP, \emptyset}},
\end{equation}
\begin{equation}
    \FactorPrefixSeries\Par{\SetP, \CorollaTwo{\GenA_i}} =
    \Leaf + \sum_{\substack{j \in \N \\ j \ne i}}
    \GenA_j \Composition
    \Han{\FactorPrefixSeries\Par{\SetP, \CorollaTwo{\GenA_j}},
    \FactorPrefixSeries\Par{\SetP, \emptyset}},
    \qquad i \in \N,
\end{equation}
\end{subequations}
for the $\GeneratingSet$-trees factor-avoiding $\SetP$. Observe that
we work here with an infinite alphabet and an infinite set of stringy
patterns. This system contains an infinite number of equations.
\medbreak

\item Let the alphabet
\begin{math}
    \GeneratingSet := \GeneratingSet(1) := \Bra{\GenA, \GenB}
\end{math}
and the set of patterns
\begin{equation}
    \SetP :=
    \Bra{
    \GenA \circ_1 \underbrace{\GenB \circ_1 \dots \circ_1 \GenB}
        _{\times k}
    \circ_1 \GenA
    : k \in \N}.
\end{equation}
By Proposition~\ref{prop:system_avoiding_stringy}, we obtain the system
\begin{subequations}
\begin{equation}
    \FactorPrefixSeries\Par{\SetP, \emptyset}
    = \Leaf
    + \GenA \Composition \Han{\FactorPrefixSeries\Par{\SetP, \SetQ}}
    + \GenB \Composition \Han{\FactorPrefixSeries\Par{\SetP,
        \emptyset}},
\end{equation}
\begin{equation}
    \FactorPrefixSeries\Par{\SetP, \SetQ}
    = \Leaf
    + \GenB
    \Composition \Han{\FactorPrefixSeries\Par{\SetP, \SetQ}},
\end{equation}
\end{subequations}
for the $\GeneratingSet$-trees factor-avoiding $\SetP$, where
\begin{equation}
    \SetQ := \partial_{\GenA, 1}(\SetP) =
    \Bra{
    \underbrace{\GenB \circ_1 \dots \circ_1 \GenB}_{\times k}
    \circ_1 \GenA
    : k \in \N}.
\end{equation}
Observe that even if $\SetP$ is an infinite set of stringy patterns,
this system contains a finite number of equations. By
Proposition~\ref{prop:system_enumeration_trees_avoiding}, we obtain the
system
\begin{subequations}
\begin{equation}
    \FactorPrefixEnumSeries\Par{\SetP, \emptyset} =
    t
    + q q_\GenA \FactorPrefixEnumSeries\Par{\SetP, \SetQ}
    + q q_\GenB \FactorPrefixEnumSeries\Par{\SetP, \emptyset},
\end{equation}
\begin{equation}
    \FactorPrefixEnumSeries\Par{\SetP, \SetQ} =
    t
    + q q_\GenB \FactorPrefixEnumSeries\Par{\SetP, \SetQ},
\end{equation}
\end{subequations}
for the enumerative image of the characteristic series of the
$\GeneratingSet$-trees factor-avoiding~$\SetP$.
\medbreak

\item Let the alphabet
\begin{math}
    \GeneratingSet := \GeneratingSet(2) := \Bra{\GenA_i : i \in \Z}
\end{math}
and the set of patterns
\begin{equation}
    \SetP :=
    \Bra{\TreeLeft{\GenA_i}{\GenA_j} : i, j \in \Z, j \leq i}
    \cup
    \Bra{\TreeRight{\GenA_i}{\GenA_j} : i, j \in \Z, j \leq i}.
\end{equation}
By a direct inspection of $\SetP$, there is a one-to-one
correspondence between the set of the trees factor-avoiding $\SetP$ and
the set of \Def{increasing binary trees}, which are binary trees where
internal nodes are labeled on $\Z$ in such a way that the label of any
node is smaller than the ones of its children. By
Proposition~\ref{prop:system_avoiding_stringy}, we obtain the system
\begin{subequations}
\begin{equation}
    \FactorPrefixSeries\Par{\SetP, \emptyset}
    = \Leaf
    + \sum_{i \in \Z}
    \GenA_i
    \Composition \Han{\FactorPrefixSeries\Par{\SetP, \SetQ^{(i)}},
    \FactorPrefixSeries\Par{\SetP, \SetQ^{(i)}}},
\end{equation}
\begin{equation}
    \FactorPrefixSeries\Par{\SetP, \SetQ^{(i)}}
    = \Leaf
    + \sum_{\substack{j \in \Z \\ j \geq i + 1}}
    \GenA_j
    \Composition \Han{\FactorPrefixSeries\Par{\SetP, \SetQ^{(j)}},
    \FactorPrefixSeries\Par{\SetP, \SetQ^{(j)}}},
    \qquad i \in \Z,
\end{equation}
\end{subequations}
for the $\GeneratingSet$-trees factor-avoiding $\SetP$, where for any
$j \in \Z$,
\begin{equation}
    \SetQ^{(j)} := \Bra{\CorollaTwo{\GenA_i} : i \in \Z, i \leq j}.
\end{equation}
Observe that we work here with an infinite alphabet and an infinite set
of stringy patterns. This system contains an infinite number of
equations.
\medbreak

\item Let the alphabet
\begin{math}
    \GeneratingSet := \GeneratingSet(2) := \{\GenA\}
\end{math}
and the set of patterns
\begin{equation}
    \SetP :=
    \begin{tikzpicture}[Centering,xscale=0.22,yscale=0.22]
        \node(0)at(0.00,-5.25){};
        \node(2)at(2.00,-5.25){};
        \node(4)at(4.00,-3.50){};
        \node(6)at(6.00,-1.75){};
        \node[NodeST](1)at(1.00,-3.50){$\GenA$};
        \node[NodeST](3)at(3.00,-1.75){$\GenA$};
        \node[NodeST](5)at(5.00,0.00){$\GenA$};
        \draw[Edge](0)--(1);
        \draw[Edge](1)--(3);
        \draw[Edge](2)--(1);
        \draw[Edge](3)--(5);
        \draw[Edge](4)--(3);
        \draw[Edge](6)--(5);
        \node(r)at(5.00,1.75){};
        \draw[Edge](r)--(5);
    \end{tikzpicture}
    \begin{tikzpicture}[Centering,xscale=0.2,yscale=0.2]
        \node(0)at(0.00,-1.80){};
        \node(2)at(2.00,-3.60){};
        \node(4)at(4.00,-7.20){};
        \node(6)at(6.00,-7.20){};
        \node(8)at(8.00,-5.40){};
        \node[NodeST](1)at(1.00,0.00){$\GenA$};
        \node[NodeST](3)at(3.00,-1.80){$\GenA$};
        \node[NodeST](5)at(5.00,-5.40){$\GenA$};
        \node[NodeST](7)at(7.00,-3.60){$\GenA$};
        \draw[Edge](0)--(1);
        \draw[Edge](2)--(3);
        \draw[Edge](3)--(1);
        \draw[Edge](4)--(5);
        \draw[Edge](5)--(7);
        \draw[Edge](6)--(5);
        \draw[Edge](7)--(3);
        \draw[Edge](8)--(7);
        \node(r)at(1.00,1.5){};
        \draw[Edge](r)--(1);
    \end{tikzpicture}.
\end{equation}
By Proposition~\ref{prop:system_avoiding_stringy}, we obtain the system
\begin{footnotesize}
\begin{subequations}
\begin{equation}
    \FactorPrefixSeries\Par{\SetP, \emptyset}
    = \Leaf
    + \GenA
    \Composition \Han{
    \FactorPrefixSeries\Par{\SetP,
    \begin{tikzpicture}[Centering,xscale=0.22,yscale=0.22]
        \node(0)at(0.00,-3.33){};
        \node(2)at(2.00,-3.33){};
        \node(4)at(4.00,-1.67){};
        \node[NodeST](1)at(1.00,-1.67){$\GenA$};
        \node[NodeST](3)at(3.00,0.00){$\GenA$};
        \draw[Edge](0)--(1);
        \draw[Edge](1)--(3);
        \draw[Edge](2)--(1);
        \draw[Edge](4)--(3);
        \node(r)at(3.00,1.75){};
        \draw[Edge](r)--(3);
    \end{tikzpicture}},
    \FactorPrefixSeries\Par{\SetP,
    \begin{tikzpicture}[Centering,xscale=0.17,yscale=0.17]
        \node(0)at(0.00,-1.75){};
        \node(2)at(2.00,-5.25){};
        \node(4)at(4.00,-5.25){};
        \node(6)at(6.00,-3.50){};
        \node[NodeST](1)at(1.00,0.00){$\GenA$};
        \node[NodeST](3)at(3.00,-3.50){$\GenA$};
        \node[NodeST](5)at(5.00,-1.75){$\GenA$};
        \draw[Edge](0)--(1);
        \draw[Edge](2)--(3);
        \draw[Edge](3)--(5);
        \draw[Edge](4)--(3);
        \draw[Edge](5)--(1);
        \draw[Edge](6)--(5);
        \node(r)at(1.00,2.25){};
        \draw[Edge](r)--(1);
    \end{tikzpicture}
    }},
\end{equation}
\begin{equation}
    \FactorPrefixSeries\Par{\SetP,
    \begin{tikzpicture}[Centering,xscale=0.22,yscale=0.22]
        \node(0)at(0.00,-3.33){};
        \node(2)at(2.00,-3.33){};
        \node(4)at(4.00,-1.67){};
        \node[NodeST](1)at(1.00,-1.67){$\GenA$};
        \node[NodeST](3)at(3.00,0.00){$\GenA$};
        \draw[Edge](0)--(1);
        \draw[Edge](1)--(3);
        \draw[Edge](2)--(1);
        \draw[Edge](4)--(3);
        \node(r)at(3.00,1.75){};
        \draw[Edge](r)--(3);
    \end{tikzpicture}}
    = \Leaf
    + \GenA
    \Composition \Han{
    \FactorPrefixSeries\Par{\SetP,
    \CorollaTwo{\GenA}
    \begin{tikzpicture}[Centering,xscale=0.22,yscale=0.22]
        \node(0)at(0.00,-3.33){};
        \node(2)at(2.00,-3.33){};
        \node(4)at(4.00,-1.67){};
        \node[NodeST](1)at(1.00,-1.67){$\GenA$};
        \node[NodeST](3)at(3.00,0.00){$\GenA$};
        \draw[Edge](0)--(1);
        \draw[Edge](1)--(3);
        \draw[Edge](2)--(1);
        \draw[Edge](4)--(3);
        \node(r)at(3.00,1.75){};
        \draw[Edge](r)--(3);
    \end{tikzpicture}},
    \FactorPrefixSeries\Par{\SetP,
    \begin{tikzpicture}[Centering,xscale=0.17,yscale=0.17]
        \node(0)at(0.00,-1.75){};
        \node(2)at(2.00,-5.25){};
        \node(4)at(4.00,-5.25){};
        \node(6)at(6.00,-3.50){};
        \node[NodeST](1)at(1.00,0.00){$\GenA$};
        \node[NodeST](3)at(3.00,-3.50){$\GenA$};
        \node[NodeST](5)at(5.00,-1.75){$\GenA$};
        \draw[Edge](0)--(1);
        \draw[Edge](2)--(3);
        \draw[Edge](3)--(5);
        \draw[Edge](4)--(3);
        \draw[Edge](5)--(1);
        \draw[Edge](6)--(5);
        \node(r)at(1.00,2.25){};
        \draw[Edge](r)--(1);
    \end{tikzpicture}}},
\end{equation}
\begin{equation}
    \FactorPrefixSeries\Par{\SetP,
    \begin{tikzpicture}[Centering,xscale=0.17,yscale=0.17]
        \node(0)at(0.00,-1.75){};
        \node(2)at(2.00,-5.25){};
        \node(4)at(4.00,-5.25){};
        \node(6)at(6.00,-3.50){};
        \node[NodeST](1)at(1.00,0.00){$\GenA$};
        \node[NodeST](3)at(3.00,-3.50){$\GenA$};
        \node[NodeST](5)at(5.00,-1.75){$\GenA$};
        \draw[Edge](0)--(1);
        \draw[Edge](2)--(3);
        \draw[Edge](3)--(5);
        \draw[Edge](4)--(3);
        \draw[Edge](5)--(1);
        \draw[Edge](6)--(5);
        \node(r)at(1.00,2.25){};
        \draw[Edge](r)--(1);
    \end{tikzpicture}}
    = \Leaf
    + \GenA
    \Composition \Han{
    \FactorPrefixSeries\Par{\SetP,
    \begin{tikzpicture}[Centering,xscale=0.22,yscale=0.22]
        \node(0)at(0.00,-3.33){};
        \node(2)at(2.00,-3.33){};
        \node(4)at(4.00,-1.67){};
        \node[NodeST](1)at(1.00,-1.67){$\GenA$};
        \node[NodeST](3)at(3.00,0.00){$\GenA$};
        \draw[Edge](0)--(1);
        \draw[Edge](1)--(3);
        \draw[Edge](2)--(1);
        \draw[Edge](4)--(3);
        \node(r)at(3.00,1.75){};
        \draw[Edge](r)--(3);
    \end{tikzpicture}},
    \FactorPrefixSeries\Par{\SetP,
    \begin{tikzpicture}[Centering,xscale=0.22,yscale=0.22]
        \node(0)at(0.00,-3.33){};
        \node(2)at(2.00,-3.33){};
        \node(4)at(4.00,-1.67){};
        \node[NodeST](1)at(1.00,-1.67){$\GenA$};
        \node[NodeST](3)at(3.00,0.00){$\GenA$};
        \draw[Edge](0)--(1);
        \draw[Edge](1)--(3);
        \draw[Edge](2)--(1);
        \draw[Edge](4)--(3);
        \node(r)at(3.00,1.75){};
        \draw[Edge](r)--(3);
    \end{tikzpicture}
    \begin{tikzpicture}[Centering,xscale=0.17,yscale=0.17]
        \node(0)at(0.00,-1.75){};
        \node(2)at(2.00,-5.25){};
        \node(4)at(4.00,-5.25){};
        \node(6)at(6.00,-3.50){};
        \node[NodeST](1)at(1.00,0.00){$\GenA$};
        \node[NodeST](3)at(3.00,-3.50){$\GenA$};
        \node[NodeST](5)at(5.00,-1.75){$\GenA$};
        \draw[Edge](0)--(1);
        \draw[Edge](2)--(3);
        \draw[Edge](3)--(5);
        \draw[Edge](4)--(3);
        \draw[Edge](5)--(1);
        \draw[Edge](6)--(5);
        \node(r)at(1.00,2.25){};
        \draw[Edge](r)--(1);
    \end{tikzpicture}}},
\end{equation}
\begin{equation}
    \FactorPrefixSeries\Par{\SetP,
    \CorollaTwo{\GenA}
    \begin{tikzpicture}[Centering,xscale=0.22,yscale=0.22]
        \node(0)at(0.00,-3.33){};
        \node(2)at(2.00,-3.33){};
        \node(4)at(4.00,-1.67){};
        \node[NodeST](1)at(1.00,-1.67){$\GenA$};
        \node[NodeST](3)at(3.00,0.00){$\GenA$};
        \draw[Edge](0)--(1);
        \draw[Edge](1)--(3);
        \draw[Edge](2)--(1);
        \draw[Edge](4)--(3);
        \node(r)at(3.00,1.75){};
        \draw[Edge](r)--(3);
    \end{tikzpicture}}
    = \Leaf,
\end{equation}
\begin{equation}
    \FactorPrefixSeries\Par{\SetP,
    \begin{tikzpicture}[Centering,xscale=0.22,yscale=0.22]
        \node(0)at(0.00,-3.33){};
        \node(2)at(2.00,-3.33){};
        \node(4)at(4.00,-1.67){};
        \node[NodeST](1)at(1.00,-1.67){$\GenA$};
        \node[NodeST](3)at(3.00,0.00){$\GenA$};
        \draw[Edge](0)--(1);
        \draw[Edge](1)--(3);
        \draw[Edge](2)--(1);
        \draw[Edge](4)--(3);
        \node(r)at(3.00,1.75){};
        \draw[Edge](r)--(3);
    \end{tikzpicture}
    \begin{tikzpicture}[Centering,xscale=0.17,yscale=0.17]
        \node(0)at(0.00,-1.75){};
        \node(2)at(2.00,-5.25){};
        \node(4)at(4.00,-5.25){};
        \node(6)at(6.00,-3.50){};
        \node[NodeST](1)at(1.00,0.00){$\GenA$};
        \node[NodeST](3)at(3.00,-3.50){$\GenA$};
        \node[NodeST](5)at(5.00,-1.75){$\GenA$};
        \draw[Edge](0)--(1);
        \draw[Edge](2)--(3);
        \draw[Edge](3)--(5);
        \draw[Edge](4)--(3);
        \draw[Edge](5)--(1);
        \draw[Edge](6)--(5);
        \node(r)at(1.00,2.25){};
        \draw[Edge](r)--(1);
    \end{tikzpicture}}
    = \Leaf
    + \GenA
    \Composition \Han{
    \FactorPrefixSeries\Par{\SetP,
    \CorollaTwo{\GenA}
    \begin{tikzpicture}[Centering,xscale=0.22,yscale=0.22]
        \node(0)at(0.00,-3.33){};
        \node(2)at(2.00,-3.33){};
        \node(4)at(4.00,-1.67){};
        \node[NodeST](1)at(1.00,-1.67){$\GenA$};
        \node[NodeST](3)at(3.00,0.00){$\GenA$};
        \draw[Edge](0)--(1);
        \draw[Edge](1)--(3);
        \draw[Edge](2)--(1);
        \draw[Edge](4)--(3);
        \node(r)at(3.00,1.75){};
        \draw[Edge](r)--(3);
    \end{tikzpicture}},
    \FactorPrefixSeries\Par{\SetP,
    \begin{tikzpicture}[Centering,xscale=0.22,yscale=0.22]
        \node(0)at(0.00,-3.33){};
        \node(2)at(2.00,-3.33){};
        \node(4)at(4.00,-1.67){};
        \node[NodeST](1)at(1.00,-1.67){$\GenA$};
        \node[NodeST](3)at(3.00,0.00){$\GenA$};
        \draw[Edge](0)--(1);
        \draw[Edge](1)--(3);
        \draw[Edge](2)--(1);
        \draw[Edge](4)--(3);
        \node(r)at(3.00,1.75){};
        \draw[Edge](r)--(3);
    \end{tikzpicture}
    \begin{tikzpicture}[Centering,xscale=0.17,yscale=0.17]
        \node(0)at(0.00,-1.75){};
        \node(2)at(2.00,-5.25){};
        \node(4)at(4.00,-5.25){};
        \node(6)at(6.00,-3.50){};
        \node[NodeST](1)at(1.00,0.00){$\GenA$};
        \node[NodeST](3)at(3.00,-3.50){$\GenA$};
        \node[NodeST](5)at(5.00,-1.75){$\GenA$};
        \draw[Edge](0)--(1);
        \draw[Edge](2)--(3);
        \draw[Edge](3)--(5);
        \draw[Edge](4)--(3);
        \draw[Edge](5)--(1);
        \draw[Edge](6)--(5);
        \node(r)at(1.00,2.25){};
        \draw[Edge](r)--(1);
    \end{tikzpicture}}},
\end{equation}
\end{subequations}
\end{footnotesize}%
for the $\GeneratingSet$-trees factor-avoiding $\SetP$. Observe that we
work here with a finite alphabet and a finite set of stringy patterns.
The set of patterns considered here comes from an example appearing
in~\cite{KP15}. Our system shown here is different from the ones
presented in this cited work.
\medbreak

\item Let the alphabet
\begin{math}
    \GeneratingSet := \GeneratingSet(2) := \Bra{\GenA_1, \GenA_2}
\end{math}
and the set of patterns
\begin{equation}
    \SetP :=
    \begin{tikzpicture}[Centering,xscale=0.22,yscale=0.19]
        \node(0)at(0.00,-4.67){};
        \node(2)at(2.00,-4.67){};
        \node(4)at(4.00,-4.67){};
        \node(6)at(6.00,-4.67){};
        \node[NodeST](1)at(1.00,-2.33){$\GenA_1$};
        \node[NodeST](3)at(3.00,0.00){$\GenA_2$};
        \node[NodeST](5)at(5.00,-2.33){$\GenA_1$};
        \draw[Edge](0)--(1);
        \draw[Edge](1)--(3);
        \draw[Edge](2)--(1);
        \draw[Edge](4)--(5);
        \draw[Edge](5)--(3);
        \draw[Edge](6)--(5);
        \node(r)at(3.00,2){};
        \draw[Edge](r)--(3);
    \end{tikzpicture}
    \begin{tikzpicture}[Centering,xscale=0.22,yscale=0.19]
        \node(0)at(0.00,-4.67){};
        \node(2)at(2.00,-4.67){};
        \node(4)at(4.00,-4.67){};
        \node(6)at(6.00,-4.67){};
        \node[NodeST](1)at(1.00,-2.33){$\GenA_1$};
        \node[NodeST](3)at(3.00,0.00){$\GenA_2$};
        \node[NodeST](5)at(5.00,-2.33){$\GenA_2$};
        \draw[Edge](0)--(1);
        \draw[Edge](1)--(3);
        \draw[Edge](2)--(1);
        \draw[Edge](4)--(5);
        \draw[Edge](5)--(3);
        \draw[Edge](6)--(5);
        \node(r)at(3.00,2){};
        \draw[Edge](r)--(3);
    \end{tikzpicture}
    \begin{tikzpicture}[Centering,xscale=0.22,yscale=0.19]
        \node(0)at(0.00,-4.67){};
        \node(2)at(2.00,-4.67){};
        \node(4)at(4.00,-4.67){};
        \node(6)at(6.00,-4.67){};
        \node[NodeST](1)at(1.00,-2.33){$\GenA_2$};
        \node[NodeST](3)at(3.00,0.00){$\GenA_2$};
        \node[NodeST](5)at(5.00,-2.33){$\GenA_1$};
        \draw[Edge](0)--(1);
        \draw[Edge](1)--(3);
        \draw[Edge](2)--(1);
        \draw[Edge](4)--(5);
        \draw[Edge](5)--(3);
        \draw[Edge](6)--(5);
        \node(r)at(3.00,2){};
        \draw[Edge](r)--(3);
    \end{tikzpicture}
    \begin{tikzpicture}[Centering,xscale=0.22,yscale=0.19]
        \node(0)at(0.00,-4.67){};
        \node(2)at(2.00,-4.67){};
        \node(4)at(4.00,-4.67){};
        \node(6)at(6.00,-4.67){};
        \node[NodeST](1)at(1.00,-2.33){$\GenA_2$};
        \node[NodeST](3)at(3.00,0.00){$\GenA_2$};
        \node[NodeST](5)at(5.00,-2.33){$\GenA_2$};
        \draw[Edge](0)--(1);
        \draw[Edge](1)--(3);
        \draw[Edge](2)--(1);
        \draw[Edge](4)--(5);
        \draw[Edge](5)--(3);
        \draw[Edge](6)--(5);
        \node(r)at(3.00,2){};
        \draw[Edge](r)--(3);
    \end{tikzpicture}.
\end{equation}
A direct inspection of $\SetP$ shows that a $\GeneratingSet$-tree
factor-avoids $\SetP$ if and only if any internal node labeled by
$\GenA_2$ have at least one leaf as a child. By
Theorem~\ref{thm:system_trees_avoiding}, we obtain the system
\begin{subequations}
\begin{multline}
    \FactorPrefixSeries\Par{\SetP, \emptyset}
    = \Leaf
    + \GenA_1 \Composition
    \Han{\FactorPrefixSeries\Par{\SetP, \emptyset},
    \FactorPrefixSeries\Par{\SetP, \emptyset}}
    + \GenA_2 \Composition
    \Han{
    \FactorPrefixSeries\Par{\SetP, \emptyset},
    \FactorPrefixSeries\Par{\SetP,
        \CorollaTwo{\GenA_1} \CorollaTwo{\GenA_2}}} \\
    + \GenA_2 \Composition
    \Han{
    \FactorPrefixSeries\Par{\SetP,
        \CorollaTwo{\GenA_1} \CorollaTwo{\GenA_2}},
    \FactorPrefixSeries\Par{\SetP, \emptyset}}
    -  \GenA_2 \Composition
    \Han{
    \FactorPrefixSeries\Par{\SetP,
        \CorollaTwo{\GenA_1} \CorollaTwo{\GenA_2}},
    \FactorPrefixSeries\Par{\SetP,
        \CorollaTwo{\GenA_1} \CorollaTwo{\GenA_2}}},
\end{multline}
\begin{equation}
    \FactorPrefixSeries\Par{\SetP,
        \CorollaTwo{\GenA_1} \CorollaTwo{\GenA_2}}
    =
    \Leaf,
\end{equation}
\end{subequations}
for the $\GeneratingSet$-trees factor-avoiding $\SetP$. We work here
with a finite alphabet and a finite set of non-stringy patterns.
\end{enumerate}
\medbreak

\section{Operads, enumeration, and statistics}
\label{sec:operads_for_enumeration}
This section is devoted to using operads as tools to enumerate families of
combinatorial objects, jointly with the results presented in the previous
sections enumerating trees factor-avoiding some patterns.
\medbreak

\subsection{Nonsymmetric set-operads}
We recall here the elementary notions about operads employed thereafter.
They mainly come from~\cite{Gir18}.
\medbreak

\subsubsection{Operad axioms}
A \Def{nonsymmetric operad in the category of sets}, or a
\Def{nonsymmetric operad} for short, is a graded set $\Operad$ together
with maps
\begin{equation}
    \circ_i : \Operad(n) \times \Operad(m) \to \Operad(n + m - 1),
    \qquad 1 \leq i \leq n, 1 \leq m,
\end{equation}
called \Def{partial compositions}, and a distinguished element
$\Unit \in \Operad(1)$, the \Def{unit} of $\Operad$. This data has to
satisfy, for any $x, y, z \in \Operad$, the three relations
\begin{subequations}
\begin{equation} \label{equ:operad_axiom_1}
    \Par{x \circ_i y} \circ_{i + j - 1} z = x \circ_i \Par{y \circ_j z},
    \qquad
    1 \leq i \leq |x|, 1 \leq j \leq |y|,
\end{equation}
\begin{equation} \label{equ:operad_axiom_2}
    \Par{x \circ_i y} \circ_{j + |y| - 1} z
    =
    \Par{x \circ_j z} \circ_i y,
    \qquad
    1 \leq i < j \leq |x|,
\end{equation}
\begin{equation} \label{equ:operad_axiom_3}
    \Unit \circ_1 x = x = x \circ_i \Unit,
    \qquad 1 \leq i \leq |x|.
\end{equation}
\end{subequations}
Since we consider in this work only nonsymmetric operads, we shall call
these simply \Def{operads}.
\medbreak

\subsubsection{Elementary definitions}
Given an operad $\Operad$, one defines the \Def{full composition} maps
of $\Operad$ as the maps
\begin{equation}
    \circ :
    \Operad(n) \times
    \Operad\Par{m_1} \times \dots \times \Operad\Par{m_n}
    \to \Operad\Par{m_1 + \dots + m_n},
    \qquad
    1 \leq n, 1 \leq m_1, \dots, 1 \leq m_n,
\end{equation}
defined, for any $x \in \Operad(n)$ and $y_1, \dots, y_n \in \Operad$,
by
\begin{equation}
    x \circ \Han{y_1, \dots, y_n}
    := \Par{\dots \Par{\Par{x \circ_n y_n} \circ_{n - 1}
        y_{n - 1}} \dots} \circ_1 y_1.
\end{equation}
\medbreak

When $\Operad$ is combinatorial as a graded set, $\Operad$ is
\Def{combinatorial}. In this case, the \Def{Hilbert series}
$\HilbertSeries_{\Operad}(t)$ of $\Operad$ is the generating series
$\GeneratingSeries_{\Operad}(t)$. If $\Operad_1$ and $\Operad_2$ are two
operads, a map $\phi : \Operad_1 \to \Operad_2$ is an
\Def{operad morphism} if it respects arities, sends the unit of
$\Operad_1$ to the unit of $\Operad_2$, and commutes with partial
composition maps. We say that $\Operad_2$ is a \Def{suboperad} of
$\Operad_1$ if $\Operad_2$ is a graded subset of $\Operad_1$,
$\Operad_1$ and $\Operad_2$ have the same unit, and the partial
compositions of $\Operad_2$ are the ones of $\Operad_1$ restricted on
$\Operad_2$. For any subset $\GeneratingSet$ of $\Operad$, the
\Def{operad generated} by $\GeneratingSet$ is the smallest suboperad
$\Operad^{\GeneratingSet}$ of $\Operad$ containing $\GeneratingSet$.
When $\Operad^{\GeneratingSet} = \Operad$ and $\GeneratingSet$ is
minimal with respect to the inclusion among the subsets of
$\GeneratingSet$ satisfying this property, $\GeneratingSet$ is a
\Def{minimal generating set} of $\Operad$ and its elements are
\Def{generators} of $\Operad$. An \Def{operad congruence} of $\Operad$
is an equivalence relation $\Congr$ respecting the arities and such
that, for any $x, y, x', y' \in \Operad$, $x \Congr x'$ and
$y \Congr y'$ implies $x \circ_i y \Congr x' \circ_i y'$ for any
$i \in [|x|]$. The $\Congr$-equivalence class of any $x \in \Operad$ is
denoted by $[x]_{\Congr}$. Given an operad congruence $\Congr$, the
\Def{quotient operad} $\Operad/_{\Congr}$ is the operad on the set of
all $\Congr$-equivalence classes and defined in the usual way.
\medbreak

\subsection{Presentations, rewrite relations, and bases}
We recall the notion of presentation by generators and relations of an
operad. By using rewrite systems on syntax trees, this leads to the
notion of bases of an operad. This notion is crucial to see the elements
of an operad satisfying some conditions as syntax trees factor-avoiding
some patterns.
\medbreak

\subsubsection{Free operads and presentations}
For any graded set $\GeneratingSet$, the \Def{free operad} on
$\GeneratingSet$ is the operad $\FreeOperad(\GeneratingSet)$ wherein for
any $n \geq 1$, $\FreeOperad(\GeneratingSet)(n)$ is the set
$\SyntaxTrees(\GeneratingSet)(n)$ of all $\GeneratingSet$-trees of arity
$n$. The partial compositions $\circ_i$ of $\FreeOperad(\GeneratingSet)$
are the partial compositions of $\GeneratingSet$-trees (see
Section~\ref{subsubsec:partial_composition_trees}). A \Def{presentation}
of an operad $\Operad$ is a pair $(\GeneratingSet, \Congr)$ such that
$\GeneratingSet$ is a graded set, $\Congr$ is an operad congruence of
$\FreeOperad(\GeneratingSet)$, and $\Operad$ is isomorphic to
$\FreeOperad(\GeneratingSet)/_{\Congr}$. Let us also define the
\Def{evaluation map} $\Eval : \FreeOperad(\GeneratingSet) \to \Operad$
as the unique surjective operad morphism satisfying, for any
$\GenA \in \GeneratingSet$, $\Eval(\Corolla{\GenA}) = \GenA$. A
\Def{treelike expression} on $\GeneratingSet$ of an element $x$ of
$\Operad$ is a $\GeneratingSet$-tree of the fiber $\Eval^{-1}(x)$.
\medbreak

\subsubsection{Rewrite rules on trees and pattern avoidance}
We explain here and in the next section a useful link for our purposes
between presentations of operads and pattern avoidance in syntax trees.
This link passes by rewrite rules on syntax trees. Notations and notions
about general rewrite rules used here can be found in~\cite{BN98}.
\medbreak

A \Def{rewrite rule} on $\GeneratingSet$-trees is an ordered pair
$\Par{\TreeS, \TreeS'}$ of $\GeneratingSet$-trees such that
$|\TreeS| = |\TreeS'|$. A set of rewrite rules defines a binary relation
$\Rew$ on $\FreeOperad(\GeneratingSet)$ for which we denote by
$\TreeS \Rew \TreeS'$ the fact that~$\Par{\TreeS, \TreeS'} \in \Rew$.
For any set $\Rew$ of rewrite rules, we denote by $\RewContext$ the
\Def{rewrite relation induced} by $\Rew$ as the binary relation
satisfying
\begin{equation} \label{equ:rewrite_relation_induced}
    \TreeR \circ_i \Par{\TreeS \circ
        \Han{\TreeR_1, \dots, \TreeR_{|\TreeS|}}}
    \RewContext
    \TreeR \circ_i \Par{\TreeS' \circ
        \Han{\TreeR_1, \dots, \TreeR_{|\TreeS|}}},
\end{equation}
if $\TreeS \Rew \TreeS'$ where and $\TreeR$, $\TreeR_1$, \dots,
$\TreeR_{|\TreeS|}$ are $\GeneratingSet$-trees, and $i \in [|\TreeR|]$.
In other words, one has $\TreeT \RewContext \TreeT'$ if it is possible
to obtain $\TreeT'$ from $\TreeT$ by replacing a factor $\TreeS$ of
$\TreeT$ by $\TreeS'$ whenever $\TreeS \Rew \TreeS'$. Let also
$\RewContextRT$ be the reflexive and transitive closure of
$\RewContext$. If $\TreeT$ and $\TreeT'$ are two $\GeneratingSet$-trees
such that $\TreeT \RewContextRT \TreeT'$, then $\TreeT$ is
\Def{rewritable} into $\TreeT'$. If $\TreeT$ is a $\GeneratingSet$-tree
such that, for any $\GeneratingSet$-tree $\TreeT'$,
$\TreeT \RewContextRT \TreeT'$ implies $\TreeT = \TreeT'$, then
$\TreeT'$ is a \Def{normal form} for $\RewContext$. The set of all
normal forms for $\RewContext$ is denoted by
$\NormalForms_{\RewContext}$. If there is not infinite chain
\begin{math}
    \TreeT_0 \RewContext \TreeT_1 \RewContext \TreeT_2
    \RewContext \cdots,
\end{math}
then $\RewContext$ is \Def{terminating}. Finally, if for all
$\GeneratingSet$-trees $\TreeT$, $\TreeS_1$, and $\TreeS_2$ such that
$\TreeT \RewContextRT \TreeS_1$ and $\TreeT \RewContextRT \TreeS_2$,
there exists a $\GeneratingSet$-tree $\TreeT'$ such that
$\TreeS_1 \RewContextRT \TreeT'$ and $\TreeS_2 \RewContextRT \TreeT'$,
then $\RewContext$ is \Def{confluent}.
\medbreak

Let us denote by $\SetP_{\Rew}$ the set of the $\GeneratingSet$-trees
appearing as left members of~$\Rew$.
\medbreak

\begin{Lemma} \label{lem:normal_forms_avoiding}
    If $\Rew$ is a set of rewrite rules on $\GeneratingSet$-trees,
    then $\NormalForms_{\RewContext}$ is the set of all the
    $\GeneratingSet$-trees factor-avoiding $\SetP_{\Rew}$.
\end{Lemma}
\begin{proof}
    Assume first that $\TreeT$ is a $\GeneratingSet$-tree
    factor-avoiding $\SetP_{\Rew}$. Then, due to the
    definition~\eqref{equ:rewrite_relation_induced} of $\RewContext$,
    $\TreeT$ is not rewritable by $\RewContext$. Hence, $\TreeT$ is a
    normal form for $\RewContext$. Conversely, assume that
    $\TreeT \in \NormalForms_{\RewContext}$. In this case, by definition
    of a normal form, $\TreeT$ is not rewritable by $\RewContext$, so
    that $\TreeT$ does not admit any occurrence of a tree appearing as
    a left member of~$\Rew$.
\end{proof}
\medbreak

\subsubsection{Orientations and bases}
Let $\Operad$ be an operad admitting a presentation
$(\GeneratingSet, \Congr)$. A set~$\Rew$ of rewrite rules is an
\Def{orientation} of $\Congr$ if the reflexive, symmetric, and
transitive closure of~$\RewContext$ is~$\Congr$. When $\RewContext$ is
terminating and confluent, the orientation $\Rew$ of $\Congr$ is
\Def{faithful}.
\medbreak

\begin{Lemma} \label{lem:convergent_orientation}
    Let $\Operad$ be an operad admitting a presentation
    $(\GeneratingSet, \Congr)$ and $\Rew$ be a faithful orientation of
    $\Congr$. For any $n \geq 1$, the restriction of the evaluation map
    $\Eval$ on $\NormalForms_{\RewContext}(n)$ is a bijection between
    this last set and~$\Operad(n)$.
\end{Lemma}
\begin{proof}
    Let $x \in \Operad(n)$. Since $\GeneratingSet$ is a generating set
    of $\Operad$, $x$ admits a treelike expression $\TreeT$ on
    $\GeneratingSet$. Since $\RewContext$ is terminating, there is a
    $\GeneratingSet$-tree $\TreeT' \in [\TreeT]_{\Congr}$ such that
    $\TreeT'$ is a normal form for $\RewContext$. This implies
    $\Eval\Par{\TreeT'} = x$ and shows that $\Eval$ is surjective.
    \smallbreak

    Since $\Rew$ is an orientation of $\Congr$, if $\TreeT$ and
    $\TreeT'$ are two normal forms for $\RewContext$ of arity $n$ such
    that $\Eval(\TreeT) = \Eval\Par{\TreeT'}$, then
    $\TreeT \Congr \TreeT'$. Since $\Congr$ is the reflexive, symmetric,
    and transitive closure of $\RewContext$, and since $\RewContext$ is
    confluent, any $\Congr$-equivalence class admits at most one normal
    form. Hence, $\TreeT = \TreeT'$, showing that $\RewContext$ is
    injective.
\end{proof}
\medbreak

Let $\Operad$ be an operad admitting a presentation
$(\GeneratingSet, \Congr)$. When there exists a faithful orientation
$\Rew$ of $\Congr$, the set $\NormalForms_{\RewContext}$ is the
\Def{$\Rew$-basis} of $\Operad$. By
Lemma~\ref{lem:convergent_orientation}, the is a one-to-one
correspondence between the graded sets $\NormalForms_{\RewContext}$ and
$\Operad$. Moreover, $\NormalForms_{\RewContext}$ can be described as
the set of the trees factor-avoiding certain trees, as stated by
Lemma~\ref{lem:normal_forms_avoiding}. These bases were called
\Def{Poincaré-Birkhoff-Witt basis} in~\cite{Hof10} and maintain strong
connections with Koszulity of operads~\cite{GK94,DK10}.
\medbreak

\subsection{Refinements of Hilbert series and enumeration}
We introduce a refinement of the Hilbert series of an operad with
respect to an orientation of one of its presentations. A general
strategy to count combinatorial objects with respect to their sizes and
some statistics relying on operads and factor-avoidance in trees is
provided.
\medbreak

\subsubsection{Statistics}
A \Def{statistics} on a set $X$ is a map $\Statistics : X \to \N$. Let
$\Operad$ be an operad admitting a presentation
$(\GeneratingSet, \Congr)$ faithfully oriented by $\Rew$. Let us define,
for any $\GenA \in \GeneratingSet$, the statistics $\Statistics_\GenA$
on $\Operad$ in the following way. For any $x \in \Operad$, we set
$\Statistics_\GenA(x) := \Deg_\GenA(\TreeT)$ where $\TreeT$ is a
treelike expression on $\GeneratingSet$ of $x$ which is also a normal
form for $\RewContext$. By Lemma~\ref{lem:convergent_orientation}, this
definition is consistent since $\TreeT$ is unique among the trees
satisfying these properties.
\medbreak

\subsubsection{Refined Hilbert series}
The \Def{$\Rew$-Hilbert series} of $\Operad$ is the series
$\SeriesStatistics_{\Rew}$ of
$\K \AAngle{t, q, \AlphabetQ_\GeneratingSet}$ defined by
\begin{equation}
    \SeriesStatistics_{\Rew} := \FactorEnumSeries\Par{\SetP_{\Rew}}.
\end{equation}
In other words, $\SeriesStatistics_{\Rew}$ is the enumerative image
of the characteristic series of the $\GeneratingSet$-trees
factor-avoiding the trees appearing as left members of~$\Rew$.
\medbreak

\begin{Proposition} \label{prop:hilbert_series_operads_basis}
    Let $\Operad$ be a combinatorial operad admitting a presentation
    $(\GeneratingSet, \Congr)$ faithfully oriented by $\Rew$. Then,
    $\SeriesStatistics_{\Rew}$ is the series wherein the coefficient of
    \begin{math}
        t^n q^d
        q_{\GenA_1}^{\alpha_1} \dots q_{\GenA_\ell}^{\alpha_\ell},
    \end{math}
    $n \geq 1$, $d \geq 0$, $\alpha_i \geq 0$, $i \in [\ell]$, is the
    number of elements $x$ of $\Operad$ of arity $n$, degree $d$, and
    such that $\Statistics_{\GenA_i}(x) = \alpha_i$ for
    all~$i \in [\ell]$.
\end{Proposition}
\begin{proof}
    By Lemmas~\ref{lem:normal_forms_avoiding}
    and~\ref{lem:convergent_orientation},
    $\FactorSeries\Par{\SetP_{\Rew}}$ is the characteristic series of
    the $\Rew$-basis $\NormalForms_{\RewContext}$ of $\Operad$. The
    statement of the proposition follows from the definitions of the
    statistics $\Statistics_\GenA$, $\GenA \in \GeneratingSet$, and of
    the enumerative images of $\GeneratingSet$-tree series.
\end{proof}
\medbreak

When $\Operad$ is combinatorial, observe that the $\Rew$-Hilbert series
of $\Operad$ is a refinement of the Hilbert series of $\Operad$. Indeed,
by Proposition~\ref{prop:hilbert_series_operads_basis}, the
specialization
\begin{math}
    {\SeriesStatistics_{\Rew}}_
    {|q := 1, q_\GenA := 1, \GenA \in \GeneratingSet}
\end{math}
is the Hilbert series $\HilbertSeries_\Operad(t)$ of~$\Operad$.
\medbreak

\subsubsection{Operads as tools for enumeration}
The results presented in the previous sections can be applied, together
with operad theory, for enumerative prospects. Indeed, if $X$ is a
combinatorial graded set for which we want to describe its generating
series $\GeneratingSeries_X(t)$, a strategy consists in
\begin{enumerate}
    \item endowing $X$ with partial composition maps
    \begin{equation}
        \circ_i : X(n) \times X(m) \to X(n + m - 1),
        \qquad 1 \leq i \leq n, 1 \leq m
    \end{equation}
    so that $X$ admits the structure of an operad;
    \smallbreak

    \item exhibiting a presentation $(\GeneratingSet, \Congr)$ of the
    operad on $X$ just introduced;
    \smallbreak

    \item providing a faithful orientation $\Rew$ of $\Congr$;
    \smallbreak

    \item computing the $\Rew$-Hilbert series $\SeriesStatistics_{\Rew}$
    of the considered operad on $X$.
\end{enumerate}
By Proposition~\ref{prop:hilbert_series_operads_basis},
$\SeriesStatistics_{\Rew}$ is a refinement of $\GeneratingSeries_X(t)$
and hence, the knowledge of $\SeriesStatistics_{\Rew}$ leads to the
knowledge of $\GeneratingSeries_X(t)$. Moreover, by
Lemma~\ref{lem:normal_forms_avoiding},
Proposition~\ref{prop:system_enumeration_trees_avoiding} provides a way
to express $\SeriesStatistics_{\Rew}$ by a system of equations. Also,
this strategy to enumerate $X$ passes by the definition of the
statistics $\Statistics_\GenA$, $\GenA \in \GeneratingSet$, on $X$ which
could be of independent interest.
\medbreak

\section{Examples about series from operads} \label{sec:examples}
This last section contains examples of application of the theory of
operads for enumeration. We recall here the definitions of some operads
involving combinatorial graded sets and apply the results of
Sections~\ref{sec:pattern_avoidance}
and~\ref{sec:operads_for_enumeration} to obtain expressions for their
generating series taking into account of some statistics.
\medbreak

To not overload the notation, the results of the previous sections are
used here implicitly. Moreover, we shall not explicitly prove the
faithfulness of the considered orientations. This can easily be done
by using general results about rewrite rules on trees, as presented
for instance in~\cite{Gir18}.
\medbreak

\subsection{On some classical operads}
We begin by considering some well-known and classical operads involving
families of trees: bicolored Schröder trees, binary trees, and based
noncrossing trees.
\medbreak

\subsubsection{$2$-associative operad}
The \Def{$2$-associative operad}~\cite{LR06} is the operad $\TwoAs$
having the presentation $\Par{\GeneratingSet_\TwoAs, \Congr}$ where
\begin{equation}
    \GeneratingSet_\TwoAs := \GeneratingSet_\TwoAs(2)
    := \{\GenA, \GenB\},
\end{equation}
and $\Congr$ is the finest operad congruence satisfying
\begin{subequations}
\begin{equation} \label{equ:relation_2as_1}
    \GenA \circ_1 \GenA \Congr \GenA \circ_2 \GenA,
\end{equation}
\begin{equation} \label{equ:relation_2as_2}
    \GenB \circ_1 \GenB \Congr \GenB \circ_2 \GenB.
\end{equation}
\end{subequations}
The first dimensions of this operad are
\begin{equation}
    1, 2, 6, 22, 90, 394, 1806, 8558
\end{equation}
and form Sequence~\OEIS{A006318} of~\cite{Slo}. This operad can be
realized as an operad of bicolored Schröder trees (see for
instance~\cite{Gir18}), where a \Def{bicolored Schröder tree} is a
Schröder tree such that each internal node is assigned with an element
of the set $\{0, 1\}$ and all nodes that have a father labeled by $0$
(resp. $1$) are labeled by $1$ (resp. $0$). A definition of Schröder trees is given in
Section~\ref{subsubsec:schroder_trees}.
By setting that the arity of
a bicolored Schröder tree is the number of its leaves, the set of all
bicolored Schröder trees forms a combinatorial graded set.
\medbreak

The orientation $\Rew$ of $\Congr$ obtained by
orienting~\eqref{equ:relation_2as_1} and~\eqref{equ:relation_2as_2} from
left to right is faithful. The $\Rew$-Hilbert series of $\TwoAs$
satisfies
\begin{equation}
    \SeriesStatistics_{\Rew}
    =
    \FactorPrefixEnumSeries\Par{
        \TreeLeft{\GenA}{\GenA} \TreeLeft{\GenB}{\GenB},
        \emptyset}
\end{equation}
where
\begin{subequations}
\begin{equation}
    \SeriesStatistics_{\Rew} =
    t
    + q q_{\GenA} \FactorPrefixEnumSeries\Par{\SetP_{\Rew},
        \CorollaTwo{\GenA}}
        \SeriesStatistics_{\Rew}
    + q q_{\GenB} \FactorPrefixEnumSeries\Par{\SetP_{\Rew},
        \CorollaTwo{\GenB}}
        \SeriesStatistics_{\Rew},
\end{equation}
\begin{equation}
    \FactorPrefixEnumSeries\Par{\SetP_{\Rew}, \CorollaTwo{\GenA}} =
    t +
    q q_{\GenB} \FactorPrefixEnumSeries\Par{\SetP_{\Rew},
        \CorollaTwo{\GenB}}
        \SeriesStatistics_{\Rew},
\end{equation}
\begin{equation}
    \FactorPrefixEnumSeries\Par{\SetP_{\Rew}, \CorollaTwo{\GenB}} =
    t +
    q q_{\GenA} \FactorPrefixEnumSeries\Par{\SetP_{\Rew},
        \CorollaTwo{\GenA}}
        \SeriesStatistics_{\Rew}.
\end{equation}
\end{subequations}
This series satisfies the algebraic equation
\begin{equation}
    \SeriesStatistics_{\Rew} =
    \frac{t + q^2 q_{\GenA} q_{\GenB} t \SeriesStatistics_{\Rew}^2
    + q^2 q_{\GenA} q_{\GenB} \SeriesStatistics_{\Rew}^3}
    {1 - t q q_{\GenA} - t q q_{\GenB}},
\end{equation}
and writes as
\begin{multline}
    \SeriesStatistics_{\Rew} =
    t
    + \Par{q_\GenA + q_\GenB} q t^2
    + \Par{q_\GenA^2 + 4 q_\GenA q_\GenB + q_\GenB^2} q^2 t^3
    + \Par{q_\GenA^3 + 10 q_\GenA^2 q_\GenB + 10 q_\GenA q_\GenB^2
        + q_\GenB^3} q^3 t^4 \\
    + \Par{q_\GenA^4 + 20 q_\GenA^3 q_\GenB + 48 q_\GenA^2 q_\GenB^2
        + 20 q_\GenA q_\GenB^3 + q_\GenB^4} q^4 t^5 \\
    + \Par{q_\GenA^5 + 35 q_\GenA^4 q_\GenB + 161 q_\GenA^3 q_\GenB^2
        + 161 q_\GenA^2 q_\GenB^3 + 35 q_\GenA q_\GenB^4 + q_\GenB^5}
        q^5 t^6
    + \cdots.
\end{multline}
The statistics $\Statistics_{\GenA}$ and $\Statistics_{\GenB}$ are
related to Triangle~\OEIS{A175124} of~\cite{Slo}. These statistics count
the number of internal nodes labeled by $0$ (or by $1$) in a bicolored
Schröder tree.
\medbreak

\subsubsection{Dipterous operad}
The \Def{dipterous operad}~\cite{LR03} is the operad $\Dipt$ having the
presentation $\Par{\GeneratingSet_\Dipt, \Congr}$ where
\begin{equation}
    \GeneratingSet_\Dipt := \GeneratingSet_\Dipt(2) := \{\GenA, \GenB\},
\end{equation}
and $\Congr$ is the finest operad congruence satisfying
\begin{subequations}
\begin{equation} \label{equ:relation_dipt_1}
    \GenA \circ_1 \GenA \Congr \GenA \circ_2 \GenA,
\end{equation}
\begin{equation} \label{equ:relation_dipt_2}
    \GenB \circ_1 \GenB \Congr \GenB \circ_2 \GenA.
\end{equation}
\end{subequations}
The dimensions of this operad are the same as the ones of $\TwoAs$ so
that $\Dipt$ can be realized as an operad of bicolored Schröder trees.
\medbreak

The orientation $\Rew$ of $\Congr$ obtained by
orienting~\eqref{equ:relation_dipt_1} from left to right,
and~\eqref{equ:relation_dipt_2} from right to left is faithful.
The $\Rew$-Hilbert series of $\Dipt$ satisfies
\begin{equation}
    \SeriesStatistics_{\Rew}
    =
    \FactorPrefixEnumSeries\Par{
    \TreeLeft{\GenA}{\GenA} \TreeRight{\GenB}{\GenA},
    \emptyset}
\end{equation}
where
\begin{subequations}
\begin{equation}
    \SeriesStatistics_{\Rew} =
    t
    + q q_\GenA
    \FactorPrefixEnumSeries\Par{\SetP_{\Rew}, \CorollaTwo{\GenA}}
        \SeriesStatistics_{\Rew}
    + q q_\GenB
    \SeriesStatistics_{\Rew}
        \FactorPrefixEnumSeries\Par{\SetP_{\Rew}, \CorollaTwo{\GenA}},
\end{equation}
\begin{equation}
    \FactorPrefixEnumSeries\Par{\SetP_{\Rew}, \CorollaTwo{\GenA}} =
    t
    + q q_\GenB
    \SeriesStatistics_{\Rew}
        \FactorPrefixEnumSeries\Par{\SetP_{\Rew}, \CorollaTwo{\GenA}}.
\end{equation}
\end{subequations}
This series satisfies the algebraic equation
\begin{equation}
    \SeriesStatistics_{\Rew} =
    t + t q q_\GenA \SeriesStatistics_{\Rew}
    + q q_\GenB \SeriesStatistics_{\Rew}^2,
\end{equation}
and writes as
\begin{multline}
    \SeriesStatistics_{\Rew} =
    t + \Par{q_\GenA + q_\GenB} q t^2
    + \Par{q_\GenA^2 + 3 q_\GenA q_\GenB + 2 q_\GenB^2} q^2 t^3
    + \Par{q_\GenA^3 + 6 q_\GenA^2q_\GenB + 10 q_\GenA q_\GenB^2
    + 5 q_\GenB^3} q^3 t^4 \\
    + \Par{q_\GenA^4 + 10 q_\GenA^3 q_\GenB + 30q_\GenA^2 q_\GenB^2
    + 35 q_\GenA q_\GenB^3 + 14 q_\GenB^4} q^4 t^5 \\
    + \Par{q_\GenA^5 + 15 q_\GenA^4 q_\GenB + 70 q_\GenA^3 q_\GenB^2
    + 140 q_\GenA^2 q_\GenB^3 + 126 q_\GenA q_\GenB^4 + 42 q_\GenB^5}
    q^5 t^6
    + \cdots.
\end{multline}
The statistics $\Statistics_\GenA$ is related to
Triangle~\OEIS{A060693} of~\cite{Slo}, and the statistics
$\Statistics_\GenB$ is related to Triangle~\OEIS{A088617} of~\cite{Slo}
(one is the mirror image of the other). These statistics count the
number of peaks in Schröder paths (which are some paths in one-to-one
correspondence with bicolored Schröder trees).
\medbreak

\subsubsection{Duplicial operad}
The \Def{duplicial operad}~\cite{Lod08} is the operad $\Dup$ having the
presentation $\Par{\GeneratingSet_\Dup, \Congr}$ where
\begin{equation}
    \GeneratingSet_\Dup := \GeneratingSet_\Dup(2) := \{\GenA, \GenB\},
\end{equation}
and $\Congr$ is the finest operad congruence satisfying
\begin{subequations}
\begin{equation} \label{equ:relation_dup_1}
    \GenA \circ_1 \GenA \Congr \GenA \circ_2 \GenA,
\end{equation}
\begin{equation} \label{equ:relation_dup_2}
    \GenB \circ_1 \GenA \Congr \GenA \circ_2 \GenB,
\end{equation}
\begin{equation} \label{equ:relation_dup_3}
    \GenB \circ_1 \GenB \Congr \GenB \circ_2 \GenB.
\end{equation}
\end{subequations}
The first dimensions of this operad are
\begin{equation}
    1, 2, 5, 14, 42, 132, 429, 1430
\end{equation}
and form Sequence \OEIS{A000108} of~\cite{Slo}. This operad can be
realized as an operad of binary trees.
\medbreak

The orientation $\Rew$ of $\Congr$ obtained by
orienting~\eqref{equ:relation_dup_1}, \eqref{equ:relation_dup_2},
and~\eqref{equ:relation_dup_3} from left to right is faithful. The
$\Rew$-Hilbert series of $\Dup$ satisfies
\begin{equation}
    \SeriesStatistics_{\Rew}
    =
    \FactorPrefixEnumSeries\Par{
    \TreeLeft{\GenA}{\GenA} \TreeLeft{\GenB}{\GenA}
    \TreeLeft{\GenB}{\GenB},
    \emptyset}
\end{equation}
where
\begin{subequations}
\begin{equation}
    \SeriesStatistics_{\Rew} =
    t
    + q q_\GenA
    \FactorPrefixEnumSeries\Par{\SetP_{\Rew}, \CorollaTwo{\GenA}}
        \SeriesStatistics_{\Rew}
    + q q_\GenB
    \FactorPrefixEnumSeries\Par{\SetP_{\Rew},
        \CorollaTwo{\GenA} \CorollaTwo{\GenB}}
    \SeriesStatistics_{\Rew},
\end{equation}
\begin{equation}
    \FactorPrefixEnumSeries\Par{\SetP_{\Rew}, \CorollaTwo{\GenA}} =
    t
    + q q_\GenB
    \FactorPrefixEnumSeries\Par{\SetP_{\Rew},
        \CorollaTwo{\GenA} \CorollaTwo{\GenB}}
    \SeriesStatistics_{\Rew},
\end{equation}
\begin{equation}
    \FactorPrefixEnumSeries\Par{\SetP_{\Rew},
        \CorollaTwo{\GenA} \CorollaTwo{\GenB}} =
    t.
\end{equation}
\end{subequations}
This series satisfies the algebraic equation
\begin{equation} \label{equ:series_statistics_dup}
    \SeriesStatistics_{\Rew} =
    t + t q q_\GenB \SeriesStatistics_{\Rew}
    + t q q_\GenA \SeriesStatistics_{\Rew}
    + t q^2 q_\GenA q_\GenB \SeriesStatistics_{\Rew}^2,
\end{equation}
and writes as
\begin{multline}
    \SeriesStatistics_{\Rew} =
    t
    + \Par{q_\GenA + q_\GenB} q t^2
    + \Par{q_\GenA^2 + 3 q_\GenA q_\GenB + q_\GenB^2} q^2 t^3
    + \Par{ q_\GenA^3 + 6 q_\GenA^2 q_\GenB + 6 q_\GenA q_\GenB^2
        + q_\GenB^3} q^3 t^4 \\
    + \Par{q_\GenA^4 + 10 q_\GenA^3 q_\GenB + 20 q_\GenA^2 q_\GenB^2
        + 10 q_\GenA q_\GenB^3 + q_\GenB^4} q^4 t^5 \\
    + \Par{q_\GenA^5 + 15 q_\GenA^4 q_\GenB + 50 q_\GenA^3 q_\GenB^2
    + 50 q_\GenA^2 q_\GenB^3 + 15 q_\GenA q_\GenB^4 + q_\GenB^5} q^5 t^6
    + \cdots.
\end{multline}
The statistics $\Statistics_\GenA$ and $\Statistics_\GenB$ are
related to Triangle~\OEIS{A001263} of~\cite{Slo} known as triangle of
Narayana numbers~\cite{Nar55}. These statistics count the number of
edges oriented to the right connecting two internal nodes in a binary
tree (which are in one-to-one correspondence with the elements
of~$\Dup$).
\medbreak

\subsubsection{Based noncrossing trees}
The \Def{based noncrossing trees operad}~\cite{Cha07} (a study of
algebras over this operad was provided in~\cite{Ler11}) is the operad
$\NCT$ having the presentation $\Par{\GeneratingSet_\NCT, \Congr}$ where
\begin{equation}
    \GeneratingSet_\NCT := \GeneratingSet_\NCT(2) := \{\GenA, \GenB\},
\end{equation}
and $\Congr$ is the finest operad congruence satisfying
\begin{equation} \label{equ:relation_nct_1}
    \GenB \circ_1 \GenA \Congr \GenA \circ_2 \GenB.
\end{equation}
The first dimensions of this operad are
\begin{equation}
    1, 2, 7, 30, 143, 728, 3876, 21318, 120175
\end{equation}
and form Sequence~\OEIS{A006013} of~\cite{Slo}. This operad can be realized as an operad of
based noncrossing trees (see for instance~\cite{Gir18}). A \Def{based noncrossing tree} is
a polygon endowed with some selected edges or diagonals, called \Def{chords}, with the
restriction that the bottom side of the polygon is a chord, that no chord crosses another
one, and that there is exactly one path formed by chords between any two points of the
polygon. By setting that the arity of a based noncrossing tree is its number of points
minus $1$, the set of all based noncrossing trees forms a combinatorial graded set.
\medbreak

The orientation $\Rew$ of $\Congr$ obtained by
orienting~\eqref{equ:relation_nct_1} from left to right is faithful. The
$\Rew$-Hilbert series of $\NCT$ satisfies
\begin{equation}
    \SeriesStatistics_{\Rew}
    =
    \FactorPrefixEnumSeries\Par{
    \TreeLeft{\GenB}{\GenA},
    \emptyset}
\end{equation}
where
\begin{subequations}
\begin{equation}
    \SeriesStatistics_{\Rew} =
    t
    + q q_{\GenA} \SeriesStatistics_{\Rew}^2
    + q q_{\GenB} \FactorPrefixEnumSeries\Par{\SetP_{\Rew},
        \CorollaTwo{\GenA}}
        \SeriesStatistics_{\Rew},
\end{equation}
\begin{equation}
    \FactorPrefixEnumSeries\Par{\SetP_{\Rew}, \CorollaTwo{\GenA}} =
    t +
    q q_{\GenB} \FactorPrefixEnumSeries\Par{\SetP_{\Rew},
        \CorollaTwo{\GenA}}
        \SeriesStatistics_{\Rew}.
\end{equation}
\end{subequations}
This series satisfies the algebraic equation
\begin{equation}
    \SeriesStatistics_{\Rew} =
    t + q \Par{q_\GenA + q_\GenB} \SeriesStatistics_{\Rew}^2
    - q^2 q_\GenA q_\GenB \SeriesStatistics_{\Rew}^3 = 0,
\end{equation}
and writes as
\begin{multline}
    \SeriesStatistics_{\Rew} =
    t
    + \Par{q_\GenA + q_\GenB} q t^2
    + \Par{2 q_\GenA^2 + 3 q_\GenA q_\GenB + 2 q_\GenB^2} q^2 t^3
    + 5 \Par{q_\GenA^3 + 2 q_\GenA^2 q_\GenB + 2 q_\GenA q_\GenB^2
        + q_\GenB^3} q^3 t^4 \\
    + \Par{14 q_\GenA^4 + 35 q_\GenA^3 q_\GenB + 45 q_\GenA^2 q_\GenB^2
        + 35 q_\GenA q_\GenB^3 + 14 q_\GenB^4} q^4 t^5 \\
    + 14 \Par{3 q_\GenA^5 + 9 q_\GenA^4 q_\GenB
    + 14 q_\GenA^3 q_\GenB^2 + 14 q_\GenA^2 q_\GenB^3
    + 9 q_\GenA q_\GenB^4 + 3 q_\GenB^5} q^5 t^6
    + \cdots.
\end{multline}
The triangles related to the statistics $\Statistics_\GenA$ and
$\Statistics_\GenB$ do not appear for the time being in~\cite{Slo}.
\medbreak

\subsection{On some operads from monoids}
We shall consider examples of
combinatorial objects endowed with operad structures coming from a
general construction introduced in~\cite{Gir15}. Let us recall the
construction. Let $\Monoid$ be a monoid, that is a set endowed with an
associative product $\Product$ admitting a unit $\Unit_\Monoid$. We
denote by $\T \Monoid$ the graded set wherein for any $n \geq 1$,
$\T \Monoid(n)$ is the set of all words of length $n$ on $\Monoid$, seen
as an alphabet. This graded set $\T \Monoid$ is endowed with the partial
composition maps $\circ_i$ defined for any  $u \in \T \Monoid(n)$,
$v \in \T \Monoid(m)$, and $i \in [n]$, by
\begin{equation}
    u \circ_i v
    := u_1 \dots u_{i - 1}
    \; \Par{u_i \Product v_1} \dots \Par{u_i \Product v_m} \;
    u_{i + 1} \dots u_n.
\end{equation}
It was shown in~\cite{Gir15} that $\T \Monoid$ is an operad admitting
$\Unit_\Monoid \in \T \Monoid(1)$ as unit. Let $\N$ (resp. $\N_\ell$) be
the additive monoid of nonnegative integers (resp. the cyclic monoid of
order $\ell$, $\ell \geq 1$). In particular, the operads $\T \N$ and
$\T \N_\ell$ admit suboperads whose elements can be interpreted as
combinatorial objects.
\medbreak

\subsubsection{$m$-trees}
For any integer $m \geq 0$, an \Def{$m$-tree} is a planar rooted tree
wherein all internal nodes have arity $m + 1$. By setting that the arity
of an $m$-tree is its number of internal nodes, the set of all $m$-trees
forms a combinatorial graded set.
\medbreak

Let $\FCat{m}$ be the suboperad of $\T \N$ generated by the set
\begin{equation}
    \GeneratingSet_{\FCat{m}}
    := \Bra{00, 01, \dots, 0m}.
\end{equation}
It was shown in~\cite{Gir15} that there is a one-to-one correspondence
between the set $\FCat{m}(n)$ and the set of all $m$-trees of arity
$n \geq 1$. Therefore, $\FCat{m}$ is an operad on $m$-trees. The
dimensions of this operad are provided by the Fuss-Catalan numbers so
that
\begin{equation}
    \# \FCat{m}(n) = \binom{(m + 1)n}{n} \frac{1}{mn + 1}.
\end{equation}
This operad admits the presentation
$\Par{\GeneratingSet_{\FCat{m}}, \Congr}$ where $\Congr$ is the finest
operad congruence satisfying
\begin{equation} \label{equ:congruence_fcat_m}
    \Corolla{0 k_3} \circ_1 \Corolla{0 k_1}
    \Congr
    \Corolla{0 k_1} \circ_2 \Corolla{0 k_2},
    \qquad
    k_3 = k_1 + k_2.
\end{equation}
\medbreak

The orientation $\Rew$ of $\Congr$ obtained by
orienting all relations~\eqref{equ:congruence_fcat_m} from left to right
is faithful. By denoting, for any $k \geq 0$, by $\SetQ_k$ the set
$\Bra{\Corolla{00}, \Corolla{01}, \dots, \Corolla{0k}}$, the
$\Rew$-Hilbert series of $\FCat{k}$ satisfies
\begin{equation}
    \SeriesStatistics_{\Rew}
    =
    \FactorPrefixEnumSeries\Par{
    \TreeLeft{0 k_3}{ 0 k_1} : k_1 \leq k_3, \emptyset}
\end{equation}
where
\begin{subequations}
\begin{equation}
    \SeriesStatistics_{\Rew} =
    t + q \sum_{0 \leq k \leq m} q_{0 k} \;
    \FactorPrefixEnumSeries\Par{\SetP_{\Rew}, \SetQ_k}
    \SeriesStatistics_{\Rew},
\end{equation}
\begin{equation}
    \FactorPrefixEnumSeries\Par{\SetP_{\Rew}, \SetQ_k} =
    t + q \sum_{k + 1 \leq \ell \leq m} q_{0 \ell} \;
    \FactorPrefixEnumSeries\Par{\SetP_{\Rew}, \SetQ_\ell}
    \SeriesStatistics_{\Rew},
    \qquad 0 \leq k \leq m.
\end{equation}
\end{subequations}
By a straightforward computation, we obtain
\begin{equation} \label{equ:series_statistics_fcat_m}
    \SeriesStatistics_{\Rew} =
    t \prod_{0 \leq k \leq m}
    \Par{q q_{0 k} \; \SeriesStatistics_{\Rew} + 1}.
\end{equation}
\medbreak

Let us now focus on the case $m = 1$, for which $\FCat{1}$ is an
operad on binary trees. First, as a particular case
of~\eqref{equ:series_statistics_fcat_m}, the $\Rew$-Hilbert series of
$\FCat{1}$ expresses as
\begin{equation}
    \SeriesStatistics_{\Rew} =
    t \Par{q q_{0 0} \; \SeriesStatistics_{\Rew} + 1}
    \Par{q q_{0 1} \; \SeriesStatistics_{\Rew} + 1}.
\end{equation}
This is the series~\eqref{equ:series_statistics_dup} obtained from
the operad $\Dup$. Moreover, as a particular case
of~\eqref{equ:congruence_fcat_m}, the operad $\FCat{1}$ admits the
presentation $\Par{\GeneratingSet_{\FCat{1}}, \Congr}$ where $\Congr$ is
the finest operad congruence satisfying
\begin{subequations}
\begin{equation} \label{equ:relation_fcat_1_1}
    \Corolla{00} \circ_1 \Corolla{00}
    \Congr
    \Corolla{00} \circ_2 \Corolla{00},
\end{equation}
\begin{equation} \label{equ:relation_fcat_1_2}
    \Corolla{01} \circ_1 \Corolla{00}
    \Congr
    \Corolla{00} \circ_2 \Corolla{01},
\end{equation}
\begin{equation} \label{equ:relation_fcat_1_3}
    \Corolla{01} \circ_1 \Corolla{01}
    \Congr
    \Corolla{01} \circ_2 \Corolla{00}.
\end{equation}
\end{subequations}
The orientation $\Rew$ of $\Congr$ obtained by
orienting~\eqref{equ:relation_fcat_1_1}
and~\eqref{equ:relation_fcat_1_2} from left to right,
and~\eqref{equ:relation_fcat_1_3} from right to left is faithful.
The $\Rew$-Hilbert series of $\FCat{1}$ satisfies
\begin{equation}
    \SeriesStatistics_{\Rew}
    =
    \FactorPrefixEnumSeries\Par{
        \TreeLeft{00}{00} \TreeLeft{01}{00} \TreeRight{01}{00},
        \emptyset}
\end{equation}
where
\begin{subequations}
\begin{equation}
    \SeriesStatistics_{\Rew} =
    t
    + q q_{00}
        \FactorPrefixEnumSeries\Par{\SetP_{\Rew}, \CorollaTwo{00}}
        \SeriesStatistics_{\Rew}
    + q q_{01}
        \FactorPrefixEnumSeries\Par{\SetP_{\Rew},
        \CorollaTwo{00}}^2,
\end{equation}
\begin{equation}
    \FactorPrefixEnumSeries\Par{\SetP_{\Rew}, \CorollaTwo{00}} =
    t
    + q q_{01}
        \FactorPrefixEnumSeries\Par{\SetP_{\Rew},
        \CorollaTwo{00}}^2.
\end{equation}
\end{subequations}
This series satisfies also
\begin{equation}
    \SeriesStatistics_{\Rew} =
    \frac{1 - \sqrt{1 - 4 t q q_{01}}}
    {q \Par{q_{00} \sqrt{1 - 4 t q q_{01}} - q_{00} + 2 q_{01}}}
\end{equation}
and writes as
\begin{multline}
    \SeriesStatistics_{\Rew} =
    t
    + \Par{q_{00} + q_{01}} q t^2
    + \Par{q_{00}^2 + 2 q_{00} q_{01} + 2 q_{01}^2} q^2 t^3
    + \Par{q_{00}^3 + 3 q_{00}^2 q_{01} + 5 q_{00} q_{01}^2
        + 5 q_{01}^3} q^3 t^4 \\
    + \Par{q_{00}^4 + 4 q_{00}^3 q_{01} + 9 q_{00}^2 q_{01}^2
        + 14 q_{00} q_{01}^3 + 14 q_{01}^4} q^4 t^5 \\
    + \Par{q_{00}^5 + 5 q_{00}^4 q_{01} + 14 q_{00}^3 q_{01}^2
    + 28 q_{00}^2 q_{01}^3 + 42 q_{00} q_{01}^4 + 42 q_{01}^5} q^5 t^6
    + \cdots,
\end{multline}
The statistics $\Statistics_{00}$ and $\Statistics_{01}$ are related to
Triangles~\OEIS{A033184} and~\OEIS{A009766} of~\cite{Slo}, known as
(the mirror image of) Catalan triangle. These statistics count the
jump-length in a binary tree (see for instance~\cite{Kra04}).
\medbreak

\subsubsection{Schröder trees} \label{subsubsec:schroder_trees}
A \Def{Schröder tree} is a planar rooted tree wherein all internal nodes
have arity $2$ or more. By setting that the arity of a Schröder tree is
its number of leaves minus $1$, the set of all Schröder trees forms a
combinatorial graded set.
\medbreak

Let $\Schr$ be the suboperad of $\T \N$ generated by the set
\begin{equation}
    \GeneratingSet_\Schr := \{00, 01, 10\}.
\end{equation}
It was shown in~\cite{Gir15} that there is a one-to-one correspondence
between the set $\Schr(n)$ and the set of all Schröder trees of arity
$n \geq 1$. Therefore, $\Schr$ is an operad on Schröder trees.
The first dimensions of this operad are
\begin{equation}
    1, 3, 11, 45, 197, 903, 4279, 20793
\end{equation}
and form Sequence~\OEIS{A001003} of~\cite{Slo}. This operad admits the
presentation $\Par{\GeneratingSet_\Schr, \Congr}$ where $\Congr$ is the
finest operad congruence satisfying
\begin{subequations}
\begin{equation} \label{equ:relation_schr_1}
    \Corolla{00} \circ_1 \Corolla{00}
    \Congr
    \Corolla{00} \circ_2 \Corolla{00},
\end{equation}
\begin{equation} \label{equ:relation_schr_2}
    \Corolla{01} \circ_1 \Corolla{10}
    \Congr
    \Corolla{10} \circ_2 \Corolla{01},
\end{equation}
\begin{equation} \label{equ:relation_schr_3}
    \Corolla{00} \circ_1 \Corolla{01}
    \Congr
    \Corolla{00} \circ_2 \Corolla{10},
\end{equation}
\begin{equation} \label{equ:relation_schr_4}
    \Corolla{01} \circ_1 \Corolla{00}
    \Congr
    \Corolla{00} \circ_2 \Corolla{01},
\end{equation}
\begin{equation} \label{equ:relation_schr_5}
    \Corolla{00} \circ_1 \Corolla{10}
    \Congr
    \Corolla{10} \circ_2 \Corolla{00},
\end{equation}
\begin{equation} \label{equ:relation_schr_6}
    \Corolla{01} \circ_1 \Corolla{01}
    \Congr
    \Corolla{01} \circ_2 \Corolla{00},
\end{equation}
\begin{equation} \label{equ:relation_schr_7}
    \Corolla{10} \circ_1 \Corolla{00}
    \Congr
    \Corolla{10} \circ_2 \Corolla{10}.
\end{equation}
\end{subequations}
The orientation $\Rew$ of $\Congr$ obtained by
orienting~\eqref{equ:relation_schr_1}, \eqref{equ:relation_schr_2},
\eqref{equ:relation_schr_3}, \eqref{equ:relation_schr_4},
\eqref{equ:relation_schr_5}, and~\eqref{equ:relation_schr_6} from left
to right, and~\eqref{equ:relation_schr_7} from right to left is
faithful. The $\Rew$-Hilbert series of $\Schr$ satisfies
\begin{equation}
    \SeriesStatistics_{\Rew}
    =
    \FactorPrefixEnumSeries\Par{
    \TreeLeft{00}{00} \TreeLeft{01}{10} \TreeLeft{00}{01}
    \TreeLeft{01}{00} \TreeLeft{00}{10} \TreeLeft{01}{01}
    \TreeRight{10}{10}, \emptyset}
\end{equation}
where
\begin{subequations}
\begin{multline}
    \SeriesStatistics_{\Rew} =
    t
    + q q_{00}
    \FactorPrefixEnumSeries\Par{\SetP_{\Rew},
        \CorollaTwo{00} \CorollaTwo{01} \CorollaTwo{10}}
    \SeriesStatistics_{\Rew}
    + q q_{01}
    \FactorPrefixEnumSeries\Par{\SetP_{\Rew},
        \CorollaTwo{00} \CorollaTwo{01} \CorollaTwo{10}}
    \SeriesStatistics_{\Rew} \\
    + q q_{10}
    \SeriesStatistics_{\Rew}
    \FactorPrefixEnumSeries\Par{\SetP_{\Rew}, \CorollaTwo{10}},
\end{multline}
\begin{equation}
    \FactorPrefixEnumSeries\Par{\SetP_{\Rew},
        \CorollaTwo{00} \CorollaTwo{01} \CorollaTwo{10}}
    = t,
\end{equation}
\begin{equation}
    \FactorPrefixEnumSeries\Par{\SetP_{\Rew}, \CorollaTwo{10}} =
    t
    + q q_{00}
    \FactorPrefixEnumSeries\Par{\SetP_{\Rew},
        \CorollaTwo{00} \CorollaTwo{01} \CorollaTwo{10}}
    \SeriesStatistics_{\Rew}
    + q q_{01}
    \FactorPrefixEnumSeries\Par{\SetP_{\Rew},
        \CorollaTwo{00} \CorollaTwo{01} \CorollaTwo{10}}
    \SeriesStatistics_{\Rew}.
\end{equation}
\end{subequations}
This series satisfies the algebraic equation
\begin{equation}
    t
    + \Par{t q \Par{q_{00} + q_{01} + q_{10}} - 1}
        \SeriesStatistics_{\Rew}
    + \Par{t q^2 \Par{q_{00} q_{10} + q_{01} q_{10}}}
        \SeriesStatistics_{\Rew}^2
    = 0
\end{equation}
and writes as
\begin{multline}
    \SeriesStatistics_{\Rew} =
    t
    + \Par{q_{00} + q_{01} + q_{10}} q t^2
    + \Par{q_{00}^2 + 2 q_{00} q_{01} + 3 q_{00} q_{10} + q_{01}^2
        + 3 q_{01} q_{10} + q_{10}^2} q^2 t^3 \\
    + \Par{q_{00}^3 + 3 q_{00}^2 q_{01} + 6 q_{00}^2 q_{10}
        + 3 q_{00} q_{01}^2 + 12 q_{00} q_{01} q_{10}
        + 6 q_{00} q_{10}^2
        \mright. \\ \mleft.
        + q_{01}^3 + 6 q_{01}^2 q_{10} + 6 q_{01} q_{10}^2
        + q_{10}^3} q^3 t^4
    + \cdots.
\end{multline}
The statistics $\Statistics_{00}$ and $\Statistics_{10}$ are related to
Triangle~\OEIS{A126216} of~\cite{Slo}, and the statistics
$\Statistics_{01}$ is related to Triangle~\OEIS{A114656} of~\cite{Slo}.
\medbreak

\subsubsection{Motzkin paths}
A \Def{Motzkin path} is a path in $\N^2$ connecting the points $(0, 0)$
and $(n - 1, 0)$ by steps in the set $\Bra{(1, -1), (1, 0), (1, 1)}$. By
setting that the arity of a Motzkin path is $n$, the set of all Motzkin
paths forms a combinatorial graded set.
\medbreak

Let $\Motz$ be the suboperad of $\T \N$ generated by the set
\begin{equation}
    \GeneratingSet_{\Motz} := \Bra{00, 010}.
\end{equation}
It was shown in~\cite{Gir15} that there is a one-to-one correspondence
between the set $\Motz(n)$ and the set of all Motzkin paths of arity
$n \geq 1$. Therefore, $\Motz$ is an operad on Motzkin paths. The first
dimensions of this operad are
\begin{equation}
    1, 1, 2, 4, 9, 21, 51, 127
\end{equation}
and form Sequence~\OEIS{A001006} of~\cite{Slo}. This operad admits the
presentation $\Par{\GeneratingSet_\Motz, \Congr}$ where $\Congr$ is the
finest operad congruence satisfying
\begin{subequations}
\begin{equation} \label{equ:relation_motz_1}
    \Corolla{00} \circ_1 \Corolla{00}
    \Congr
    \Corolla{00} \circ_2 \Corolla{00},
\end{equation}
\begin{equation} \label{equ:relation_motz_2}
    \Corolla{010} \circ_1 \Corolla{00}
    \Congr
    \Corolla{00} \circ_2 \Corolla{010},
\end{equation}
\begin{equation} \label{equ:relation_motz_3}
    \Corolla{00} \circ_1 \Corolla{010}
    \Congr
    \Corolla{010} \circ_3 \Corolla{00},
\end{equation}
\begin{equation} \label{equ:relation_motz_4}
    \Corolla{010} \circ_1 \Corolla{010}
    \Congr
    \Corolla{010} \circ_3 \Corolla{010}.
\end{equation}
\end{subequations}
The orientation $\Rew$ of $\Congr$ obtained by
orienting~\eqref{equ:relation_motz_1},
\eqref{equ:relation_motz_2}, \eqref{equ:relation_motz_3},
and~\eqref{equ:relation_motz_4} from left to right is faithful. The
$\Rew$-Hilbert series of $\Motz$ satisfies
\begin{equation}
    \SeriesStatistics_{\Rew}
    =
    \FactorPrefixEnumSeries\Par{
    \TreeLeft{00}{00}
    \begin{tikzpicture}[Centering,xscale=0.22,yscale=0.22]
        \node(0)at(0.00,-4.00){};
        \node(2)at(2.00,-4.00){};
        \node(4)at(3.00,-2.00){};
        \node(5)at(4.00,-2.00){};
        \node[NodeST](1)at(1.00,-2.00){$00$};
        \node[NodeST](3)at(3.00,0.00){$010$};
        \draw[Edge](0)--(1);
        \draw[Edge](1)--(3);
        \draw[Edge](2)--(1);
        \draw[Edge](4)--(3);
        \draw[Edge](5)--(3);
        \node(r)at(3.00,2){};
        \draw[Edge](r)--(3);
    \end{tikzpicture}
    \begin{tikzpicture}[Centering,xscale=0.22,yscale=0.22]
        \node(0)at(0.00,-4.00){};
        \node(2)at(1.00,-4.00){};
        \node(3)at(2.00,-4.00){};
        \node(5)at(4.00,-2.00){};
        \node[NodeST](1)at(1.00,-2.00){$010$};
        \node[NodeST](4)at(3.00,0.00){$00$};
        \draw[Edge](0)--(1);
        \draw[Edge](1)--(4);
        \draw[Edge](2)--(1);
        \draw[Edge](3)--(1);
        \draw[Edge](5)--(4);
        \node(r)at(3.00,2){};
        \draw[Edge](r)--(4);
    \end{tikzpicture}
    \begin{tikzpicture}[Centering,xscale=0.22,yscale=0.2]
        \node(0)at(0.00,-4.67){};
        \node(2)at(1.00,-4.67){};
        \node(3)at(2.00,-4.67){};
        \node(5)at(3.00,-2.33){};
        \node(6)at(4.00,-2.33){};
        \node[NodeST](1)at(1.00,-2.33){$010$};
        \node[NodeST](4)at(3.00,0.00){$010$};
        \draw[Edge](0)--(1);
        \draw[Edge](1)--(4);
        \draw[Edge](2)--(1);
        \draw[Edge](3)--(1);
        \draw[Edge](5)--(4);
        \draw[Edge](6)--(4);
        \node(r)at(3.00,2){};
        \draw[Edge](r)--(4);
    \end{tikzpicture},
    \emptyset}
\end{equation}
where
\begin{subequations}
\begin{equation}
    \SeriesStatistics_{\Rew} =
    t
    + q q_{00} \FactorPrefixEnumSeries\Par{\SetP_{\Rew},
        \CorollaTwo{00} \CorollaThree{010}}
    \SeriesStatistics_{\Rew}
    + q q_{010} \FactorPrefixEnumSeries\Par{\SetP_{\Rew},
        \CorollaTwo{00} \CorollaThree{010}}
    \SeriesStatistics_{\Rew}^2,
\end{equation}
\begin{equation}
    \FactorPrefixEnumSeries\Par{\SetP_{\Rew},
        \CorollaTwo{00} \CorollaThree{010}}
    = t.
\end{equation}
\end{subequations}
This series satisfies the algebraic equation
\begin{equation}
    \SeriesStatistics_{\Rew} =
    t + t q q_{00} \SeriesStatistics_{\Rew}
    + t q q_{010} \SeriesStatistics_{\Rew}^2
\end{equation}
and writes as
\begin{multline}
    \SeriesStatistics_{\Rew} =
    t + q q_{00} t^2
    + \Par{q^2 q_{00}^2 + q q_{010}} t^3
    + \Par{q^3 q_{00}^3 + 3 q^2 q_{00} q_{010}} t^4 \\
    + \Par{q^4 q_{00}^4 + 6 q^3 q_{00}^2 q_{010}
        + 2 q^2 q_{010}^2} t^5
    + \Par{q^5 q_{00}^5 + 10 q^4 q_{00}^3 q_{010}
        + 10 q^3 q_{00} q_{010}^2} t^6 \\
    + \Par{q^6 q_{00}^6 + 15 q^5 q_{00}^4 q_{010}
        + 30 q^4 q_{00}^2 q_{010}^2 + 5 q^3 q_{010}^3} t^7
    + \cdots.
\end{multline}
The statistics $\Statistics_{00}$ and $\Statistics_{010}$ are related to
Triangle~\OEIS{A055151} of~\cite{Slo}. These statistics count the number
of steps $(1, 1)$ in a Motzkin path.
\medbreak

\subsubsection{Directed animals}
A \Def{directed animal} is a finite subset $A$ of $\N^2$ containing
$(0, 0)$ and if $(x, y) \in A \setminus \{(0, 0)\}$, then
$(x - 1, y) \in A$ or $(x, y - 1) \in A$. By setting that the arity of
a directed animal is its cardinality, the set of all directed animals
forms a combinatorial graded set.
\medbreak

Let $\DA$ be the suboperad of $\T \N_3$ generated by the set
\begin{equation}
    \GeneratingSet_\DA := \{00, 01\}.
\end{equation}
It was shown in~\cite{Gir15} that there is a one-to-one correspondence
between the set $\DA(n)$ and the set of all directed animals of arity
$n \geq 1$. Therefore, $\DA$ is an operad on directed animals. The first
dimensions of this operad are
\begin{equation}
    1, 2, 5, 13, 35, 96, 267, 750
\end{equation}
and form Sequence~\OEIS{A005773} of~\cite{Slo}. This operad admit the
presentation $\Par{\GeneratingSet_\DA, \Congr}$ where $\Congr$ is the
finest operad congruence satisfying
\begin{subequations}
\begin{equation} \label{equ:relation_da_1}
    \Corolla{00} \circ_1 \Corolla{00}
    \Congr
    \Corolla{00} \circ_2 \Corolla{00},
\end{equation}
\begin{equation} \label{equ:relation_da_2}
    \Corolla{01} \circ_1 \Corolla{00}
    \Congr
    \Corolla{00} \circ_2 \Corolla{01},
\end{equation}
\begin{equation} \label{equ:relation_da_3}
    \Corolla{01} \circ_1 \Corolla{01}
    \Congr
    \Corolla{01} \circ_2 \Corolla{00},
\end{equation}
\begin{equation} \label{equ:relation_da_4}
    \Par{\Corolla{00} \circ_1 \Corolla{01}} \circ_2 \Corolla{01}
    \Congr
    \Par{\Corolla{01} \circ_2 \Corolla{01}} \circ_3 \Corolla{01}.
\end{equation}
\end{subequations}
The orientation $\Rew$ of $\Congr$ obtained by
orienting~\eqref{equ:relation_da_1}, \eqref{equ:relation_da_2} from
left to right, and~\eqref{equ:relation_da_3}
and~\eqref{equ:relation_da_4} from right to left, is faithful. The
$\Rew$-Hilbert series of $\DA$ satisfies
\begin{equation}
    \SeriesStatistics_{\Rew}
    =
    \FactorPrefixEnumSeries\Par{\TreeLeft{00}{00}
    \TreeLeft{01}{00} \TreeRight{01}{00}
    \begin{tikzpicture}[Centering,xscale=0.19,yscale=0.21]
        \node(0)at(0.00,-1.75){};
        \node(2)at(2.00,-3.50){};
        \node(4)at(4.00,-5.25){};
        \node(6)at(6.00,-5.25){};
        \node[NodeST](1)at(1.00,0.00){01};
        \node[NodeST](3)at(3.00,-1.75){01};
        \node[NodeST](5)at(5.00,-3.50){01};
        \draw[Edge](0)--(1);
        \draw[Edge](2)--(3);
        \draw[Edge](3)--(1);
        \draw[Edge](4)--(5);
        \draw[Edge](5)--(3);
        \draw[Edge](6)--(5);
        \node(r)at(1.00,1.31){};
        \draw[Edge](r)--(1);
    \end{tikzpicture},
    \emptyset}
\end{equation}
where
\begin{subequations}
\begin{equation}
    \SeriesStatistics_{\Rew} =
    t
    + q q_{00} \FactorPrefixEnumSeries\Par{\SetP_{\Rew},
            \CorollaTwo{00}} \SeriesStatistics_{\Rew}
    + q q_{01} \FactorPrefixEnumSeries\Par{\SetP_{\Rew},
            \CorollaTwo{00}}
        \FactorPrefixEnumSeries\Par{\SetP_{\Rew},
            \CorollaTwo{00} \TreeRight{01}{01}},
\end{equation}
\begin{equation}
    \FactorPrefixEnumSeries\Par{\SetP_{\Rew}, \CorollaTwo{00}} =
    t
    + q q_{01} \FactorPrefixEnumSeries\Par{\SetP_{\Rew},
            \CorollaTwo{00}}
        \FactorPrefixEnumSeries\Par{\SetP_{\Rew},
            \CorollaTwo{00} \TreeRight{01}{01}},
\end{equation}
\begin{equation}
    \FactorPrefixEnumSeries\Par{\SetP_{\Rew},
    \CorollaTwo{00} \TreeRight{01}{01}} =
    t
    + q q_{01} \FactorPrefixEnumSeries\Par{\SetP_{\Rew},
            \CorollaTwo{00}}
        \FactorPrefixEnumSeries\Par{\SetP_{\Rew},
            \CorollaTwo{00} \CorollaTwo{01}
            \TreeRight{01}{01}},
\end{equation}
\begin{equation}
    \FactorPrefixEnumSeries\Par{\SetP_{\Rew},
        \CorollaTwo{00} \CorollaTwo{01} \TreeRight{01}{01}} =
    t.
\end{equation}
\end{subequations}
This series satisfies
\begin{equation}
    \SeriesStatistics_{\Rew} =
    \frac{1 - \sqrt{1 - 2 t q q_{01} - 3 t^2 q^2 q_{01}^{2}}
        - t q {\Par{2 q_{00} + q_{01}}}}{
    2 t q^2 \Par{q_{00}^{2} + q_{00} q_{01} + q_{01}^{2}} - 2 q q_{00}}
\end{equation}
and writes as
\begin{multline}
    \SeriesStatistics_{\Rew} =
    t + \Par{q_{00} + q_{01}} q t^2
    + \Par{q_{00}^2 + 2 q_{00} q_{01} + 2 q_{01}^2} q^2 t^3
    + \Par{q_{00}^3 + 3 q_{00}^2 q_{01} + 5 q_{00} q_{01}^2
        + 4 q_{01}^3} q^3 t^4 \\
    + \Par{q_{00}^4 + 4 q_{00}^3 q_{01} + 9 q_{00}^2 q_{01}^2
        + 12 q_{00} q_{01}^3 + 9 q_{01}^4} q^4 t^5 \\
    + \Par{q_{00}^5 + 5 q_{00}^4 q_{01} + 14 q_{00}^3 q_{01}^2
    + 25 q_{00}^2 q_{01}^3 + 30 q_{00} q_{01}^4 + 21 q_{01}^5} q^5 t^6
    + \cdots,
\end{multline}
The statistics $\Statistics_{00}$ is related to
Triangle~\OEIS{A064189} of~\cite{Slo}, and the statistics
$\Statistics_{01}$ is related to Triangle~\OEIS{A026300} of~\cite{Slo}
(one is the mirror image of the other).
\medbreak

\bibliographystyle{alpha}
\bibliography{Bibliography}

\end{document}